\documentclass[12pt, reqno]{amsart}

\usepackage{hyperref} %
\usepackage{textcomp} 
\usepackage{amsmath, amsfonts,amsthm,amssymb,amsxtra}
\usepackage{amscd}
\usepackage{enumitem}

\usepackage{tikz-cd} 
\tikzset{cong/.style={draw=none,edge node={node [sloped, allow upside down, auto=false]{$\cong$}}},
	Isom/.style={draw=none,every to/.append style={edge node={node [sloped, allow upside down, auto=false]{$\cong$}}}}}

\usepackage{xcolor}
\usepackage{graphicx}
\usepackage{verbatim}
\newtheorem{theorem}{Theorem}[section]
\newtheorem{proposition}[theorem]{Proposition}
\newtheorem{lemma}[theorem]{Lemma}
\newtheorem{corollary}[theorem]{Corollary}

\theoremstyle{definition}

\newtheorem{definition}[theorem]{Definition}
\newtheorem{defn}[theorem]{Definition}
\newtheorem{examples}[theorem]{Example}
\newtheorem{example}[theorem]{Example}
\newtheorem{notation}[theorem]{Notation}
\newtheorem{construction}[theorem]{Construction}

\theoremstyle{remark}

\newtheorem{remark}[theorem]{Remark}

\numberwithin{equation}{section}

\renewcommand{\epsilon}{\varepsilon}

\renewcommand{\phi}{\varphi}
\newcommand{\R}{\mathbb{R}}

\DeclareMathOperator{\spec}{spec}
\DeclareMathOperator{\supp}{supp}

\newcommand{\nc}{\newcommand}
\nc{\eqn}{\begin{equation}}
\nc{\eqnn}{\begin{equation}\nonumber}
\nc{\eqnd}{\end{equation}}
\nc{\enum}{\begin{enumerate}}
	\nc{\enumd}{\end{enumerate}}

\nc{\dd}{\diamond}
\nc{\del}{\partial}
\nc{\tensor}{\otimes}
\nc{\into}{\hookrightarrow}

\nc{\defliou}{\mathcal{D}\!\operatorname{ef}}
\nc{\nbhd}{\operatorname{Nbhd}}
\nc{\lioucrit}{\mathcal{L}\!\operatorname{iou}_{\mathsf{crit}}^{\dd}}
\nc{\lioustab}{\mathcal{L}\!\operatorname{iou}^{\dd}}
\nc{\lioustr}{\mathsf{Liou}_{\mathsf{str}}}
\nc{\weincrit}{\mathcal{W}\!\operatorname{ein}_{\mathsf{crit}}^{\dd}}
\nc{\weinstab}{\mathcal{W}\!\operatorname{ein}^{\dd}}
\nc{\weinstr}{\mathsf{Wein}_{\mathsf{str}}}
\nc{\weinparam}{\mathsf{Wein}_{\mathsf{param}}}
\nc{\subcrit}{\mathfrak{s}}
\nc{\eqs}{\mathfrak{eq}}

\def\cC{\mathcal C}\def\cD{\mathcal D}
\def\cE{\mathcal E}

\def\cO{\mathcal O}

\def\cW{\mathcal W}

\def\CC{\mathbb C}
\def\EE{\mathbb E}

\def\RR{\mathbb R}\def\SS{\mathbb S}

\def\ZZ{\mathbb Z}

\newcommand{\DMO}{\DeclareMathOperator}
\DMO{\id}{id}
\DMO{\ob}{ob}
\DMO{\fun}{Fun}
\DMO{\Mod}{Mod}
\DMO{\End}{End}
\DMO{\crit}{Crit}
\DMO{\skel}{skel}
\DMO{\flex}{Flex}
\DMO{\movie}{movie}
\DMO{\colim}{colim}

\begin{document}
	
	\title{Localization and flexibilization in symplectic geometry}
	
	\author{Oleg Lazarev, Zachary Sylvan, and Hiro Lee Tanaka}

	\begin{abstract}
		We introduce the critical Weinstein $\infty$-category -- the result of 
		stabilizing the category of Weinstein sectors and inverting subcritical morphisms -- and for every finite collection $P$ of integers, construct a {\em $P$-flexibilization} endofunctor. 	Our main result is that $P$-flexibilization is an idempotent localization functor of the critical Weinstein $\infty$-category, 
		allowing us to characterize the essential image of the endofunctor by a universal property.		
		This localization has the effect of replacing every Weinstein sector with one in which $P$ is invertible in the wrapped Fukaya category and hence is a symplectic analogue of topological localization of Bousfield and Sullivan, answering a question of Abouzaid and Seidel. When $P = \{0\}$, our construction recovers Cieliebak and Eliashberg's  flexibilization procedure. Moreover, we show that $P$-flexibilization is symmetric monoidal as a functor of higher categories, and hence gives rise to a new way of constructing $E_\infty$-commutative algebra objects from symplectic geometry. 
	\end{abstract}
	\maketitle

	\tableofcontents
	
	\section{Introduction}

	In this work we introduce the \textit{critical} $\infty$-category of stable Weinstein sectors 
	\eqnn
	\weincrit
	\eqnd
	whose objects are exact symplectic manifolds with favorable corner structures, admitting symplectically well-behaved cellular decompositions, and identified up to taking products with $T^*[0,1]$. The morphism spaces are obtained by formally inverting a class of morphisms realizing subcritical handle attachments. Along the way, we also invert maps resulting from Weinstein homotopies; see Definition~\ref{def:sub_morphism}. This is a natural context for studying wrapped Fukaya categories, which are known to remain unchanged by stabilization,  subcritical handle attachments, and Weinstein homotopies \cite{ganatra_generation}. 
    
Using $\weincrit$ enables us to prove our main result: There is a symplectic geometric model for the rich algebraic operation of inverting an integer (in fact, any finite collection P of integers). Namely, given any Weinstein sector $X$, we construct its {\em $P$-flexibilization} $X[P^{-1}]$, whose Fukaya category is equivalent to that of $X$ after inverting $P$. 
We show that $P$-flexibilization is both functorial (Lemma~\ref{lemma. product is functor}) and multiplicative (Theorem~\ref{theorem. P-flexibilization is symmetric monoidal}) in $\weincrit$. 

Our work generalizes and renders functorial Murphy and Cieliebak-Eliashberg's flexibilization~\cite{CE12} (without relying on h-principles), Abouzaid-Seidel's homologous recombination~\cite{abouzaid_seidel_recombination}, and work of Lazarev-Sylvan~\cite{Lazarev_Sylvan}.  Upon a construction of the conjectural {\em spectral} wrapped Fukaya category, our methods are expected to yield purely symplectic constructions of various  localizations of the stable homotopy category, and symmetric monoidally so. 
	In the reverse direction, the critical $\infty$-category also allows us to give a clean categorical description of previously known geometry: Flexibilization {\em is} a localization 
    (Theorem~\ref{theorem. intro P flex is a localization}). A key mathematical step in the proof of our result is comparing two symplectic operations -  taking direct product with a particular sector and carving out of Lagrangian disks - and proving that they are equivalent in $\weincrit$; see Section \ref{sec: proof_outline} for a proof outline.     
	
	\begin{remark}
	Central to our constructions is a class of exact symplectic manifolds called Weinstein domains.  These domains can be equipped with symplectic handle-body decompositions analogous to the cellular decomposition of CW complexes.
	We will often pass between ``Weinstein sectors'' and ``Weinstein domains equipped with a Weinstein hypersurface'' (these data are equivalent; see Section~\ref{section. stopped domains are sectors}). 
	\end{remark}

    As we explain in the next sections, our result can be viewed as both a symplectic analog of Bousfield \cite{bousfield-localization-1975} and Sullivan's \cite{Sullivan_localization} localization construction in classical  topology and also a generalization of the h-principle for flexible Weinstein structures \cite{CE12, Murphy11} (which we reprove in our setting). Whereas the h-principle for flexible Weinstein structures shows that flexibilization projects away all symplectic information (up to homotopy), we prove that  P-flexibilization projects away \textit{some} symplectic information and hence provides a geometric interpolation between flexibility and rigidity.
	
	\subsection{Background and motivation}
	Because this work produces a symplectic analog of topological localization that simultaneously generalizes symplectic flexibilization, we review both ideas briefly.

	\subsubsection{Localization in classical topology}\label{section: intro_localization_classical}
	Localization in algebra and topology allows one to study global phenomena one prime at a time. 	
    Fix a prime $p$.
	In algebraic topology, a concrete way to localize a simply-connected space $X$ is to begin with a CW presentation of $X$ and replace all standard CW cells of $X$ by ``$p$-cells" to create a new CW complex $X[p^{-1}]$. (See the classic works of Sullivan~\cite{sullivan-genetics, sullivan-notes}.) 
	This operation satisfies the following properties:

	\newenvironment{localization-props}{
		\renewcommand*{\theenumi}{(T\arabic{enumi})}
		\renewcommand*{\labelenumi}{(T\arabic{enumi})}
		\enumerate
	}{
		\endenumerate
	}
	
	\begin{localization-props}
		\item\label{item. homotopy invariance}  \textbf{Homotopy invariance:} The assignment  $X \mapsto X[p^{-1}]$  preserves homotopy equivalences. In particular, while  
         the above description  of $X[p^{-1}]$ depends on the CW presentation of $X$, a posteriori the homotopy type of $X[p^{-1}]$ depends  only on the homotopy type of $X$.
		
		\item\label{item. homology localizes} \textbf{Localization on homology: } There is a map $\eta_X: X \to X[p^{-1}]$ exhibiting the homology groups of $X[p^{-1}]$ as the localization away from p of the homology groups of $X$.  
        So $\eta_X$ induces an isomorphism  $H_i(X[p^{-1}]; \mathbb{Z}) \cong H_i(X; \mathbb{Z}) \otimes \mathbb{Z}[p^{-1}]$ for $i > 0$.
		
		\item\label{item. idempotency}  \textbf{Idempotency: } The map $\eta_{X[p^{-1}]}$ is a 
        homotopy equivalence $X[p^{-1}] \xrightarrow{\sim} (X[p^{-1}])[p^{-1}]$.
		
	\end{localization-props}

	We have mentioned Sullivan's construction as it parallels the construction $X[P^{-1}]$ 
    (for $P=\{p\}$)  for a Weinstein $X$
 closely.
	A different construction allows one to include all (not necessarily simply-connected) CW complexes $X$ and is manifestly homotopy invariant by construction; see for example Bousfield's work~\cite{bousfield-localization-1975}.

		Let us point out two related but distinct notions of localization that appear above. In \ref{item. homology localizes}, we describe a construction that inverts certain multiplicative operations on algebraic invariants ({eg }multiplication by $p$), which is a localization of groups. 
        Upon passage to the $\infty$-category of simply-connected spaces, \ref{item. idempotency} 
        may be promoted to an idempotent endofunctor. 
        As we explain in Section \ref{sec: category}, an idempotent endofunctor is called a (categorical) \textit{localization} since it projects the entire category to some subcategory of `local' objects, in the process throwing away some information from the category. Furthermore, the idempotency of taking algebraic invariants in \ref{item. homology localizes}  follows from the idempotency of the geometric construction in \ref{item. idempotency}. In the next sections, we describe similar phenomena in symplectic geometry.

	\subsubsection{Flexibilization in symplectic geometry}
	\label{section: flexibilization_intro}
	Symplectic flexibility refers to phenomena where the underlying smooth topology (plus the class, up to smooth homotopy through non-degenerate 2-forms, of the symplectic form) determines the symplectic geometry. One of the first instances of symplectic flexibility is Gromov's h-principle~\cite{gromov_hprinciple} for \textit{subcritical} isotropic submanifolds -- {ie }isotropic submanifolds whose dimensions are less than half the ambient dimension. If two subcritical isotropics are isotopic through smoothly embedded submanifolds (covered by a homotopy of their tangent bundles to isotropic subbundles),  then they are isotopic through isotropics. 		
	Consequently, subcritical Weinstein domains, which are by definition built out of handles attached along subcritical isotropics, are also determined by their smooth topology. 
	
	Gromov's result~\cite{gromov_hprinciple} was generalized by Cieliebak-Eliashberg~\cite{CE12} and Murphy~\cite{Murphy11}, who 
	showed that there is a special class of \textit{flexible} Weinstein structures satisfying the following h-principle: If two Weinstein manifolds $X_1, X_2$ with flexible structures are diffeomorphic (plus a bit of bundle-theoretic data), then they are actually symplectomorphic. 
	In fact, for any Weinstein manifold $X$, Cieliebak and Eliashberg, Ch.7~\cite{CE12}, construct the \textit{flexibilization} $X_{flex}$ of $X$, a flexible Weinstein manifold that is diffeomorphic (but not symplectomorphic) to $X$. Just as with Sullivan's model of localization for CW complexes, $X_{flex}$ is constructed by replacing all standard Weinstein handles of $X$ with `flexible' Weinstein handles, where attaching spheres are \textit{loose Legendrians} \cite{Murphy11}. $X_{flex}$ has the following properties, analogous to the properties of localization of topological spaces: 
	
	\newenvironment{flex-props}{
		\renewcommand*{\theenumi}{(F\arabic{enumi})}
		\renewcommand*{\labelenumi}{(F\arabic{enumi})}
		\enumerate
	}{
		\endenumerate
	}
	\begin{flex-props}
		\item\label{item. flex homotopy invariance} \textbf{Homotopy invariance:} $X_{flex}$ depends only the Weinstein homotopy type of $X$ (in fact, only the diffeomorphism type of $X$, plus a bit of tangential data). In particular, though the construction of $X_{flex}$ depends on the Weinstein presentation of $X$, its Weinstein homotopy type depends a posteriori only on the Weinstein homotopy type of $X$.
		
		\item\label{item. flex localizes} \textbf{Fukaya category localizes:} $X_{flex}$ is a Weinstein subdomain of  $X$ \cite{Lazarev_crit_points}  and the wrapped Fukaya category of $X_{flex}$ is trivial, so $\mathcal{W}(X_{flex}) \cong 0 \cong \mathcal{W}(X)\otimes \mathbb{Z}[\frac{1}{0}]$.
		
		\item\label{item. flex idempotence}  \textbf{Idempotency: }
		One can arrange for the subdomain inclusion $(X_{flex})_{flex} \into X_{flex}$ to be a Weinstein homotopy equivalence.
		
	\end{flex-props}
	Property~\ref{item. flex localizes} indicates that flexibilization localizes invariants away from the integer zero. The proofs of the other two properties crucially rely on the h-principle for flexible Weinstein structures \cite{CE12}. In particular, the idempotency of the flexibilization reflects the fact that this construction projects to the symplectically trivial class of flexible objects, {ie }all symplectic information is projected away (up to Weinstein homotopy).  Next, we discuss P-flexibilization, which does \textit{not} satisfy an h-principle but by our main result (see Corollary \ref{cor: intro_X_P_properties}) is still idempotent and therefore projects away \textit{some} symplectic information.

	\subsubsection{Toward $P$-flexibilization}
	\label{section. toward p-flex}

In Section \ref{section: intro_localization_classical}, we reviewed Sullivan's localization of topological spaces away from a collection of integers $P$ and emphasized its key properties (T1), (T2), (T3). 
In Section \ref{section: flexibilization_intro}, we explained how flexibilization achieves a symplectic analog of this localization in the case when $P = 0$, with analogous properties (F1), (F2), (F3).
Our main goal is to develop a P-flexibilization construction which achieves a symplectic analog of Sullivan's result for a \textit{non-zero} set of integers $P$. In particular, given a Weinstein sector $X$, we will construct a new sector $X[P^{-1}]$ with the following analogous properties; one key difference throughout is the necessity to work in the critical Weinstein category $\weincrit$.

\newenvironment{Pflex-props}{
		\renewcommand*{\theenumi}{(P\arabic{enumi})}
		\renewcommand*{\labelenumi}{(P\arabic{enumi})}
		\enumerate
	}{
		\endenumerate
	}
	\begin{Pflex-props}
		\item
        \label{item. Pflex homotopy invariance}
        \textbf{Homotopy invariance:}  Up to stabilization and subcritical handle attachment, $X[P^{-1}]$ depends only the Weinstein homotopy type of $X$. In particular, though the construction of $X[P^{-1}]$ depends on the Weinstein presentation of $X$, its 
       isomorphism class in $\weincrit$ depends a posteriori only on the Weinstein homotopy type of $X$. See Corollary \ref{cor: intro_X_P_properties} below. 
				
		\item
        \label{item. Pflex localizes} 
        \textbf{Fukaya category localizes:} $X[P^{-1}]$ is a Weinstein subdomain of  $X$  and the wrapped Fukaya category of $X[P^{-1}]$ is the localization of that of $X$ away from $P$, {ie }    
        $\mathcal{W}(X[P^{-1}])  \cong \mathcal{W}(X)\otimes \mathbb{Z}[\frac{1}{P}]$. See the discussion below, and reference to ~\cite{Lazarev_Sylvan}. 
		
		\item
        \label{item. Pflex idempotence} \textbf{Idempotency: }
		One can arrange for the subdomain inclusion $(X[P^{-1}])[P^{-1}] \into X[P^{-1}]$ to induce an isomorphism in $\weincrit$. See Corollary \ref{cor: intro_X_P_properties} below. 
	\end{Pflex-props}
Furthermore, our construction has the property that if $P = \{1\}$, then 
$X[\{1\}^{-1}]$ is  $X$ with some subcritical handles attached, and hence isomorphic to $X$ in $\weincrit$; if $P = \{0\}$, then $X[\{0\}^{-1}]$ is flexible; see Section \ref{section. discussion of P flex} for discussion of these examples.

Before giving precise and functorial versions of these properties, we first provide some context for P-flexiblization. For a collection of integers $P$, Abouzaid and Seidel~\cite{abouzaid_seidel_recombination} showed that for any Weinstein domain $X$ with $\dim X \ge 12$,  there is a Weinstein domain $X[P^{-1}]^{AS}$ diffeomorphic to $X$ that abstractly admits a group isomorphism between symplectic cohomologies as in property \ref{item. Pflex localizes}. Their work further asked whether this construction can be viewed as a symplectic analog of topological localization; see Remark 3.14 of \cite{abouzaid_seidel_recombination}.	
	However, it was not (and is still not) clear whether there is a geometrically defined map between $X$ and $X^{AS}[P^{-1}]$ whose induced map on $SH$ realizes $SH(X[P^{-1}]^{AS})$ as the $P$-inversion of $SH(X)$. To remedy this, the first two authors~\cite{Lazarev_Sylvan} introduced a variant construction $X[P^{-1}]$ (which they conjectured is equivalent to $X[P^{-1}]^{AS}$) defined for any Weinstein sector $X$ with $\dim X\ge 10$.

	Because we will need it momentarily, let us recall the construction of $X[P^{-1}]$ in the case $X = T^*D^n$ (which also gives the local construction for arbitrary $X$). For any collection $P$ of primes, and for $n\ge 5$, one first creates a regular Lagrangian disk $D_P \subset T^*D^n$ built out of a $P$-Moore space $M_P$ (really, a wedge of $p$-Moore spaces for $p\in P$). The sector 
		\eqn\label{eqn. P-localized T*D}
		T^*D^n[P^{-1}] := T^*D^n \setminus D_P
		\eqnd is constructed by carving out this disk and endowing the result with an appropriate Weinstein structure; see Section \ref{sec: abouzaid_seidel}. 
For example, if $P = \{1\}$, then $(T^*D^n)_1$ is $T^*D^n$ with a subcritical handle attached and hence isomorphic to $T^*D^n$ in $\weincrit$; if $P$ is the empty set, we can define $D_P$ to be the empty Lagrangian disk, in which case $(T^*D^n)_\varnothing = T^*D^n$. On the other hand, if $P = \{0\}$, then $(T^*D^n)_0$ is flexible; see Section \ref{section. discussion of P flex} for more examples. 

		For general $X$, we perform this construction in a neighborhood $T^*D^n$ of every co-core of $X$ (which uses the data of the Weinstein structure of $X$).  This remedies the above issue as follows: $X[P^{-1}]$  is a Weinstein subdomain of $X$, and the resulting Viterbo functor realizes the wrapped category of $X[P^{-1}]$ as a localization inverting $P$~\cite{Lazarev_Sylvan}.	 Hence $X[P^{-1}]$ achieves property (P2) above.       
The naive intuition is that by removing this disk, we ``kill'' objects representing the Moore space, thereby inverting $P$.
	Because of the parallel with the cell-by-cell construction of~\cite{sullivan-genetics, sullivan-notes}, we refer to $X[P^{-1}]$ as the {\em Sullivan-style} construction.

The difficulty is that $X[P^{-1}]^{AS}$ and $X[P^{-1}]$ depend very much on the Weinstein presentation of $X$ ({eg }to identify the cocores of $X$; see Example \ref{example: non-homotopy}). As a result, these constructions are not homotopy invariant, as required for property (P1). 
    As discussed in Section \ref{section: flexibilization_intro}, the h-principle was crucial to showing that flexibilization is independent of the presentation. 
    However, an h-principle  \textit{cannot} exist for $X[P^{-1}]^{AS}$ or $X[P^{-1}]$ since there are plenty of diffeomorphic Weinstein domains $X, Y$ for which $X[P^{-1}], Y[P^{-1}]$ are not symplectomorphic, {eg }the exotic cotangent bundles  from~\cite{abouzaid_seidel_recombination} or~\cite{EGL}. This is because  $X[P^{-1}]^{AS}$ or $X[P^{-1}]$ still retain non-trivial Fukaya categories at primes other than $P$, unlike the flexibilization $X_{flex}$.
    For similar reasons, the h-principle cannot be applied to establish idempotency property (P3).	
	Thus, a new notion of equivalence will be needed to articulate (P1) and (P3), which we achieve by working in $\weincrit$.

	\subsection{P-flexibilization is a localization}
	
	In this paper, we propose that the natural notion of equivalence needed to prove (P1) and (P3)    is generated by the two operations of stabilization and subcritical morphisms. (These change the symplectic geometry only as much as the smooth topology is changed and, as we mentioned, do not change the wrapped Fukaya category.)
	The critical Weinstein $\infty$-category $\weincrit$ is precisely the minimal $\infty$-category for which these operations are homotopy-invertible. (See Section~\ref{sec: comparison_P_flexibilization} for details).
	
	We caution the reader that, for the rest of this introduction, ``Weinstein domain'' means a domain admitting (but not equipped with) a Weinstein Lyapunov function.
	
	Fix a finite set of non-negative  integers $P$ and let $T^*D^n[P^{-1}]$ be as in~\eqref{eqn. P-localized T*D}
    so $n$ is some fixed integer larger than 5. 
    Consider the functor 
	\eqn\label{eqn. functor times T*D^nP}
	- \times (T^*D^n[P^{-1}]): \weincrit \rightarrow \weincrit,
	\qquad
	X \mapsto X \times (T^*D^n[P^{-1}])
	\eqnd
	taking any Weinstein domain to its direct product with $T^*D^n[P^{-1}]$. 
By Proposition \ref{prop: T^*D[P^{-1}] independent of n}, this functor is independent of $n$ and depends only on the collection of integers $P$.

  We note that any proper cylindrical Lagrangian in $T^*D^n$, like $D_P \subset T^*D^n$, is disjoint from some cotangent fiber of $T^*D^n$, for example the cotangent fibers of points near the boundary $\partial D^n$ of $D^n$. By considering a Weinstein tubular neighborhood of this cotangent fiber, we obtain a proper inclusion $i: T^*D^n \hookrightarrow T^*D^n$ whose image is disjoint from $D_P$ and therefore a proper inclusion $T^*D^n \to T^*D^n[P^{-1}] = T^*D^n \backslash D_P$.  This proper inclusion induces a natural transformation
	\eqnn
	\eta: - \times T^*D^n \rightarrow - \times (T^*D^n[P^{-1}])
	\eqnd
	from the identity functor.\footnote{By construction, in $\weincrit$, any sector $X$ is naturally identified with its stabilizations $X \times T^*D^n$ for any $n \geq 0$.}
	
	\begin{theorem}\label{thm: intro_idempotent_functor}
		The functor~\eqref{eqn. functor times T*D^nP} is idempotent. More precisely, 
		the natural transformation $\eta$ evaluated at $T^*D^n[P^{-1}]$
		\eqn\label{eqn. eta on DP}
		\eta_{T^*D^n[P^{-1}]}: (T^*D^n[P^{-1}]) \times (T^*D^n) \rightarrow (T^*D^n[P^{-1}])\times (T^*D^n[P^{-1}])
		\eqnd
    is a sectorial equivalence up to stabilization and subcritical handle attachment. In particular, 
	\eqref{eqn. eta on DP} is an isomorphism in $\weincrit$.
	\end{theorem}
	
	The above theorem is the geometric fact that gives rise to all categorical results of our paper. As an immediate consequence we see that, up to subcritical equivalence and stabilization, direct product with $T^*D^n[P^{-1}]$ is a natural notion of a $P$-flexibilization of $X$. Indeed, the critical analogue \ref{item. Pflex homotopy invariance} of \ref{item. flex homotopy invariance} is obviously satisfied because taking direct products preserves (Weinstein homotopy) equivalences. The transformation $\eta_X: X \times T^*D^n \to X \times (T^*D^n[P^{-1}])$, by virtue of the Kunneth theorem, realizes~\ref{item. Pflex localizes}.\footnote{The proper inclusion $T^*D^n \to T^*D^n[P^{-1}]$ exhibits $Tw\ \cW( T^*D^n[P^{-1}])$ as the localization of $Tw\ \cW(T^*D^n) \simeq Tw \ \mathbb{Z}$ away from $P$ by \cite{Lazarev_Sylvan}.} Idempotency of~\eqref{eqn. functor times T*D^nP} is~\ref{item. Pflex idempotence}.

	Now we have two potential candidates for $P$-flexibilization; that is, two sectors $X[P^{-1}]$ and $X \times (T^*D^n[P^{-1}])$ with equivalent wrapped Fukaya categories. Our second main result is that these seemingly different constructions are in fact naturally  equivalent in the critical category, giving a geometric explanation for this algebraic equivalence.
	
	\begin{theorem}[A special case of Theorem~\ref{thm: comparison}]
		\label{thm: intro_comparison}
		Let $X$ be a Weinstein sector with $\dim X = 2n \ge 10$. Then for any choice of Weinstein structure on $X$, there is an equivalence $\phi_X: X[P^{-1}]  \tilde{\rightarrow} X \times (T^*D^n[P^{-1}])$ in $\weincrit$ satisfying the following: For every $i: X\hookrightarrow Y$ a strict proper inclusion 
		of Weinstein sectors, there is a homotopy commutative diagram in $\weincrit$:\footnote{$i_{D_P}$ is defined later as the map $i_L$ in Notation~\ref{notation. i_L}, setting $L = D_P$.}
		\eqn
		\begin{tikzcd} 
		X[P^{-1}] \arrow{rrr}{i_{D_P}}
		\arrow{d}{\phi_X} & & &
		Y[P^{-1}]  \arrow{d}{\phi_Y}\\
		X \times (T^*D^n[P^{-1}])  \arrow{rrr}{i \times Id_{T^*D^n[P^{-1}]}}	
		&&& 		
		Y \times (T^*D^n[P^{-1}]).
		\end{tikzcd}
		\eqnd
	\end{theorem}
    We note that $X[P^{-1}]$ and $X \times (T^*D^n[P^{-1}])$ have different dimensions and hence comparing them requires stabilization.

	Accordingly, we propose the following definition:
	\begin{defn}
		A Weinstein sector $X$ is \textit{$P$-flexible} if $X$ is equivalent in $\weincrit$ to an object in the image of~\eqref{eqn. functor times T*D^nP}. In other words, $X$ is $P$-flexible if -- up to stabilization and subcritical handle attachment/removal -- $X$ admits a $T^*D^n[P^{-1}]$ factor.
	\end{defn}
	This property does not depend on a Weinstein presentation of $X$, unlike the classical definition of flexibility \cite{CE12}. We note that when $P = \emptyset$ or $\{1\}$ every Weinstein sector is $P$-flexible since in this case, $T^*D^n[P^{-1}]$ is equivalent to $T^*D^n$ in $\weincrit$; see Example \ref{example: X1 is X} for discussion.

	We can now combine Theorem \ref{thm: intro_comparison} and Theorem \ref{thm: intro_idempotent_functor} to prove \ref{item. Pflex homotopy invariance} and \ref{item. Pflex idempotence} for the $X[P^{-1}]$ construction:
	
	\begin{corollary}[Porism]\label{cor: intro_X_P_properties}
		Fix two Weinstein sectors $X$ and $Y$. If $X, Y$ are isomorphic up to Weinstein (Liouville) homotopy, then $X[P^{-1}], Y[P^{-1}]$ are isomorphic up to Weinstein (Liouville) homotopy,  stabilization and subcritical handle attachment. Furthermore, $(X[P^{-1}])[P^{-1}]$ is isomorphic to $X[P^{-1}]$ up to Weinstein homotopy, stabilization and subcritical handle attachment. 
	\end{corollary}
	
	(By an isomorphism up to Weinstein homotopy, we mean a diffeomorphism $f: X \to Y$ such that the pulled back Weinstein structure may be endowed with a homotopy of Weinstein structures to the Weinstein structure of $X$; in other literature \cite{CE12} this is called a Weinstein equivalence, a term we do not use here to avoid confusion.)

	Now we consider the case $P = \{0\}$.
    By the discussion in Example \ref{example: X0 is flexible}, $X[\{0\}^{-1}]$ is flexible in the sense of Cieliebak-Eliashberg~\cite{CE12}, and is diffeomorphic to $X$ up to some smooth subcritical handles.  Corollary \ref{cor: intro_X_P_properties} gives a new proof that this flexibilization is homotopy invariant and idempotent, up to stabilization and subcritical handles.   
	\begin{corollary}\label{cor: intro_flex}
		If $X, Y$ are isomorphic up to Weinstein (Liouville) homotopy, then $X[\{0\}^{-1}]$ and $Y[\{0\}^{-1}]$ are isomorphic up to Weinstein (Liouville) homotopy, stabilization, and subcritical handles. Furthermore, $(X[\{0\}^{-1}])[\{0\}^{-1}]$ is isomorphic to $X[\{0\}^{-1}]$, up to Weinstein homotopy,  stabilization, and subcritical handles. 
	\end{corollary}
	
	The main novel feature of our proof
	is that it does not use the h-principle for flexible domains \cite{CE12} and loose Legendrians \cite{Murphy11} and hence presents a new approach to studying flexibility.
	We refer to Section~\ref{section. discussion of P flex} for examples and  further discussion. See also Section \ref{sec: proof_outline} for the role h-principles (do or do not) play in our paper. \\
	
	We turn to more structural results. The fact that $P$-flexibilization is an idempotent functor (Theorem~\ref{thm: intro_idempotent_functor}) immediately implies the following categorical fact: 	
	\begin{theorem}\label{theorem. intro P flex is a localization}
		The functor~\eqref{eqn. functor times T*D^nP} is a localization.
	\end{theorem}
	In other words, the image of~\eqref{eqn. functor times T*D^nP}---otherwise known as the $\infty$-category of Weinstein sectors critically divisible by $T^*D^n[P^{-1}]$---is characterized by a universal property: Any functor from $\weincrit$ that does not distinguish a sector from its $P$-flexibilization (more precisely, that sends $\eta$ to equivalences) automatically factors through this image. 
    See~\ref{section. idempotent localization}.
    Theorem~\ref{theorem. intro P flex is a localization} shows that the geometric operation of direct product with $(T^*D)[P^{-1}]$ has the same algebraic property that $P$-inversion does.

	\begin{remark}
		A functor being a localization is an a priori distinct notion from a functor inducing a localization of invariants. It is natural to ask whether $P$-flexibilization (and its versions equipped with tangential structures allowing for the definition of Floer-theoretic linear invariants) is the universal functor localizing wrapped-Floer invariants (and their spectral versions). We conjecture that it is.
	\end{remark}
	
	\subsection{P-flexibilization is symmetric monoidal}

	In Section \ref{sec: category}, we will see that $\weincrit$ has a symmetric monoidal structure given by direct product of sectors. Idempotency, together with the fact that~\eqref{eqn. functor times T*D^nP} is induced by direct product with an object of $\weincrit$, implies the following:
	\begin{theorem}\label{theorem. P-flexibilization is symmetric monoidal}
		The localization~\eqref{eqn. functor times T*D^nP} may be promoted to be a symmetric monoidal functor to its image.
	\end{theorem}
	The fact that we can formally obtain such symmetric monoidal structures greatly simplifies applications to high algebra; we refer the reader to Section~\ref{section. higher algebra}. For now, let us mention a geometric curiosity. 
	By Theorem~\ref{theorem. P-flexibilization is symmetric monoidal}, the symmetric monoidal unit of $\weincrit$, $T^*D^n$, has image given by the symmetric monoidal unit in the target, $T^*D^n[P^{-1}]$. Purely formally, we obtain:
	\begin{corollary}\label{corollary. P inverted disks are commutative}
		$T^*D^n[P^{-1}]$ is a commutative algebra (that is, an $E_\infty$-algebra) in $\weincrit$.
	\end{corollary}

    See Remark~\ref{remark. Eoo algebra} for a reference regarding commutative algebras ({ie }$E_\infty$-algebras) in a symmetric monoidal $\infty$-category.
    In the world of symplectic geometry, it is rare to find algebras that are commutative for geometric reasons, in contrast to objects that are {\em algebraically} forced to be commutative, such as the commutative version~\cite[Introduction]{ng-Loo} of  Legendrian dgas.
	Most commutative symplectic objects arise from SYZ fibrations, or standard variations thereof ({eg }$\RR^n$-fibers as opposed to torus fibers). At present, we do not know if the commutative structure on $T^*D^n[P^{-1}]$ from Corollary~\ref{corollary. P inverted disks are commutative} arises from such fibrations.

	Our final main result is that our results above remain true if we wish to geometrically nullify arbitrary finite CW complexes, and not just $P$-Moore spaces. This is because any finite CW complex may be represented by a regular Lagrangian disk and our results hold in this generality. 
	
	\begin{theorem}
		\label{theorem. localization along spectra}
		For any regular Lagrangian disk $L \subset T^*D^n$, the functor $-\times (T^*D^n)\setminus L$ is a symmetric monoidal localization of $\weincrit$.
		
		In particular, for any finite CW complex $K$, there is a Weinstein sector $(T^*D^n)_K := T^*D^n \setminus D_K$ so that $-\times (T^*D^n)_K$ is a symmetric monoidal localization of $\weincrit$.
	\end{theorem}
	
	In fact, the equivalence (Theorem~\ref{thm: intro_comparison}) between the direct product model and the Sullivan-style model of $P$-flexibilization extends also in this generality. See Theorem~\ref{thm: comparison}.

	\subsection{Examples}
	\label{section. discussion of P flex}
	As indicated by Theorem~\ref{theorem. localization along spectra}, our first results do not require $P$ to be a collection of {\em prime} numbers; indeed, the geometric constructions never rely on primeness. 
	
	\begin{example}
		If $P = \{1\}$ or $P = \emptyset$, then $T^*D^n[P^{-1}]$ is equivalent in $\weincrit$ to  $T^*D^n_{std}$, the usual cotangent bundle of $D^n$ with its standard Weinstein structure. See Example \ref{example: X1 is X} for extended discussion. 
	\end{example}
	
	\begin{example}\label{example. P = 0 case}
		If $P$ contains $0$, then $T^*D^n[P^{-1}]$ is flexible in the sense of Cieliebak-Eliashberg \cite{CE12}, as observed in  \cite{Lazarev_Sylvan}. See Example \ref{example: X0 is flexible} for extended discussion. 
	\end{example}

	\begin{example}
		\label{example: two sets of primes}
		Our proof of Theorem \ref{thm: intro_idempotent_functor} shows that if $P$ and $Q$ are two finite sets of integers, then  $T^*D^n[P^{-1}] \times T^*D^n[Q^{-1}]$ is equivalent to  $(T^*D^n)[(P\cup Q)^{-1}]$ in $\weincrit$; 	see Remark \ref{remark: two sets of primes, proof}.  Furthermore,  if $m = \prod_{p_i\in P} p_i$ is a product of distinct primes, then $(T^*D^n)[m^{-1}]$ is equivalent to 
		$T^*D^n[P^{-1}]$  and therefore equivalent to 	$\prod_{p_i \in P} (T^*D^n)[p_i^{-1}]$; see Remark \ref{rem: prime_factorization}.
Therefore, our construction, which a priori can depend on a finite collection of arbitrary integers, only depends on a finite collection of primes.  We expect a careful analysis of the coherences to show that $P$-flexibilization is a functor from the symmetric monoidal category of Zariski-closed subsets of $\spec \ZZ$ (where morphisms are inclusions, and the symmetric monoidal structure is union) to $\weincrit$.
	\end{example}

	\begin{example}
		Another immediate consequence of Theorem \ref{thm: intro_comparison} is that $X[P^{-1}] \times Y$ is equivalent to $(X\times Y)[P^{-1}]$ in $\weincrit$, which generalizes to P-flexibility the fact that classical flexibility is preserved by taking products \cite{MS}. Additionally, $X[P^{-1}] \times Y[Q^{-1}]$ is equivalent to $X[Q^{-1}] \times Y[P^{-1}]$ in $\weincrit$, which is non-obvious from the usual definition of $X[P^{-1}]$ in \cite{Lazarev_Sylvan} or \cite{Abouzaid_Seidel}.
	\end{example}
	
	\begin{remark}
		Previous constructions $X[P^{-1}]$ and $X[P^{-1}]^{AS}$ explicitly used the Weinstein structure of $X$ and did not apply to Liouville $X$. 
		However, $X \times (T^*D^n[P^{-1}])$ makes sense even if $X$ is a Liouville sector, and thus is a candidate definition for $P$-flexibilization in the Liouville setting. However, because of the absence of Kunneth formulas for general Liouville sectors, we do not know if~\ref{item. Pflex localizes} holds in the Liouville setting.
		See also Remark~\ref{remark. subcritical liouvilles not monoidal}.
	\end{remark}

	\subsection{Proof outline}\label{sec: proof_outline}

In this section, we give an outline of Theorem \ref{thm: intro_comparison}, the equivalence between $X[P^{-1}]$ and $X \times (T^*D^n[P^{-1}])$ in $\weincrit$; the proof of Theorem \ref{thm: intro_idempotent_functor} uses many similar ingredients. We recall that $X[P^{-1}]$ is obtained from $X$ by removing the Lagrangian disk $D_P$ from a neighborhood $T^*D^n$ of each co-core of a Weinstein structure on $X$; in particular, $X[P^{-1}]$ depends on the Weinstein structure of $X$. Broadly speaking, to relate $X[P^{-1}]$ and $X \times (T^*D^n[P^{-1}])$, we first stabilize $X[P^{-1}]$ to $X[P^{-1}] \times T^*D^n$ and then observe that both $X[P^{-1}] \times T^*D^n$ and $X \times (T^*D^n[P^{-1}])$ differ from the same sector $(X \times T^*D^n)[P^{-1}]$ by subcritical handles, and hence $X[P^{-1}]$ and $X \times (T^*D^n[P^{-1}])$ are equivalent in $\weincrit$. 

More precisely, the first step in our proof is to show that carving out Lagrangian disks and then stabilizing is equivalent to stabilizing and then carving out Lagrangian disks, up to subcritical handles. Namely, in Proposition \ref{prop. for comparison}, we prove that for any (proper) regular Lagrangian disk $L \subset X$, 
$(X \backslash L) \times T^*D^1$ is equivalent to $X \times T^*D^1 \backslash (L \times T^*_0 D^1)$, up to subcritical handles, and this they are equivalent in $\weincrit$.  This result is false without inverting subcriticals.  For example, if $X$ is a Weinstein domain (without sectorial boundary), then $X[P^{-1}]\times T^*D^n$ has sectorial divisor  $X[P^{-1}] \times T^*S^{n-1}$ while $X\times (T^*D^n[P^{-1}])$ has divisor
		$X\times T^*S^{n-1}$. These are not symplectomorphic for general $X$ (even if we pick a model for $X[P^{-1}]$ that is diffeomorphic to $X$). 	However, adding subcritical handles can change the divisor by a \textit{loose} hypersurface and hence resolves this issue; see the discussion in Example \ref{examples: subcritical_looseness}.  
	We also note that with our definition of $T^*D^n[P^{-1}]$, $T^*D^n[P^{-1}] \times T^*D^n$ and 	$(T^*D^n[P^{-1}])\times (T^*D^n[P^{-1}])$ have different cohomology in degree $2n-1$ and hence subcritical handles are required to make them equivalent.

Using Proposition \ref{prop. for comparison}, we prove that 
$X[P^{-1}] \times T^*D^n = (X \backslash \coprod_{C_X} D_P) \times T^*D^n$ (where $C_X$ are the co-cores of $X$) is equivalent in $\weincrit$ to $X \times T^*D^n \backslash (\coprod_{C_{X \times T^*D^n}} D_P \times T^*_0 D^n)$
Similarly, we prove that $X \times (T^*D^n[P^{-1}]) = X \times (T^*D^n \backslash D_P)$ is equivalent in $\weincrit$ to $X \times T^*D^n \backslash (\coprod_{C_{X \times T^*D^n}} T^*_0 D^n \times D_P)$. 
The latter two sectors  $X \times T^*D^n \backslash (\coprod_{C_{X \times T^*D^n}} D_P \times T^*_0 D^n)$ and $X \times T^*D^n \backslash (\coprod_{C_{X \times T^*D^n}} T^*_0 D^n \times D_P)$
are obtained from $X \times T^*D^n$ by carving out either $D_P \times T^*_0 D^n$ or $T^*_0 D^n \times D_P$ from neighborhoods $T^*D^n \times T^*D^n$ of the co-cores of $X \times T^*D^n$. Hence, to prove that these two sectors are equivalent, it suffices to prove the Lagrangian disks $D_P \times T^*_0 D^n$ and $T^*_0 D^n \times D_P$  in $T^*D^n \times T^*D^n$ are Lagrangian isotopic. Our second important observation is that this is indeed the case. In Proposition \ref{prop: swap_isotopic_Id}, we prove that for Lagrangians $L, K \subset T^*D^n$, the Lagrangian $L \times K$ is Lagrangian isotopic to $K \times L$ for $n$ even, and hence for any $n$ after stabilization 

\begin{remark}[Role of h-principles]
Although subcritical Weinstein handles satisfy an h-principle, our proofs of Theorem~\ref{thm: intro_idempotent_functor} through Corollary~\ref{corollary. P inverted disks are commutative} for $P$ a collection of integers do not use any h-principles.
    
Also, our techniques in fact give a new proof of properties~\ref{item. flex homotopy invariance}, \ref{item. flex localizes}, and \ref{item. flex idempotence} for classical flexibility, in our critical setting, independent of any h-principles. We also do not use the theory of wrinkled Legendrian embeddings (on which the h-principle for flexible Weinsteins was originally based).	However, h-principles do play a role for various extensions and modifications of our results, like the generalization of $P$ to arbitrary regular Lagrangian disks (the first part of Theorem \ref{theorem. localization along spectra}) does use the h-principle for subcritical isotropics; see Remark \ref{rem: use_hprinciple}). 
\end{remark}

Another useful observation throughout our work is that a Weinstein subdomain inclusion $X_0 \xrightarrow{\subset} X_1$ induces a morphism $X_1\rightarrow X_0$ in $\weincrit$; see Proposition \ref{prop: convert_subdomain_to_proper_inclusion}. The geometric ingredients of Proposition~\ref{prop: convert_subdomain_to_proper_inclusion} are those in the third author's Viterbo sector construction \cite{sylvan_talk}. Weinstein subdomain inclusions arise quite naturally in the proofs of the main results Theorems \ref{thm: intro_comparison} and \ref{thm: intro_idempotent_functor}. For example, $X[P^{-1}] \xrightarrow{\subset} X$is naturally a Weinstein subdomain inclusion and hence induces a morphism $X \rightarrow X[P^{-1}]$ in $\weincrit$. For $X = T^*D^n$, this morphism agrees with the natural transformation used to prove that the functor $\times (T^*D^n[P^{-1}])$ is an idempotent, as stated in Theorem \ref{thm: intro_idempotent_functor}.

	\subsection{Future applications: Higher algebra}
	\label{section. higher algebra}
	
	Prior to the present work, all works regarding $P$-inversion in symplectic geometry were linear, but not fully multiplicative. For example, while there were symplectic constructions of the module $\ZZ[1/P]$, there were no symplectic constructions of its natural {\em commutative} ring structure.
	
	Moreover, writing down a commutative ring structure ``by hand'' is rarely feasible for spectra. Given this difficulty, and given the emergence of spectral sectorial invariants~\cite{nadler-shende, large-whvss}, we faced the important task of producing higher-algebraic structures formally from geometric facts.
	
	Our main results accomplish a great deal of this task, establishing all the $\infty$-categorical coherences one could hope for (Theorem~\ref{thm: intro_idempotent_functor}) without engaging with holomorphic disks. Let us explain. It is widely expected that any spectral wrapped Fukaya category of Weinstein sectors is: 
	\begin{itemize}
		\item Preserved under stabilization, subcritical handle attachments, and trivial inclusions.
		\item Symmetric monoidal with respect to direct product of sectors and strict proper inclusions. (One need only verify this for {\em sets} of inclusions, not spaces of them.)
	\end{itemize}
	As a result of our present work,  it is then formal that the spectral wrapped Fukaya category will send an $E_\infty$-algebra in the critical Weinstein category to a symmetric monoidal stable $\infty$-category. In particular, the unit of this stable $\infty$-category is an example of a commutative ring spectrum.
	
	Theorem~\ref{theorem. localization along spectra} shows that the critical Weinstein category $\weincrit$ has a plethora of commutative algebra objects. In fact, because the spectral wrapped Fukaya category of a point (equipped with standard tangential structures) must be the $\infty$-category of finite spectra:
	\begin{itemize}
		\item The process of replacing $T^*D^n$ with $T^*D^n \setminus D_K$ -- as in Theorem~\ref{theorem. localization along spectra} -- nullifies (the suspension spectrum of) $K$.
	\end{itemize}
	So we expect $(T^*D^n)_K$ to be a purely symplectic way to encode ``the universal commutative ring spectrum'' whose modules nullify $\Sigma^\infty_+K$. For example, when $K=M_p$ is the $p$-Moore space, we expect the cotangent fiber of $(T^*D^n)_K$ to have endomorphism spectrum $\SS[1/p]$, the sphere with $p$ inverted.
	
\begin{remark}\label{remark. tangential structures on Wein}
		One must specify certain tangential structures of Weinstein sectors to define a spectrally enriched wrapped category.  As in~\cite{LAST_categories}, incorporating such tangential structures preserve all arguments involving localizations and symmetric monoidal structures.
	\end{remark}	

    \begin{remark}[Why $\infty$-categories?]
	
	\label{remark. why infinity cat}
	All original arguments in this work are geometric and require no higher categories to state nor to carry out. Regardless, we view and have stated our main results as infinity-categorical consequences. 
	
	There are myriad public discussions available online~\cite{MO-HTT-moral, MO-modern-alg-top, MO-applications, MO-why-higher} regarding why infinity-categorical language and techniques have proven useful.
	As we imagine a broad audience, let us explicate the role of higher-categorical outcomes of our work. 
	\begin{enumerate}
	\item\label{item. why oocat invariants} Many invariants of Weinstein sectors and sectorial embeddings land in an $\infty$-category.
	Examples include  
	the complex giving rise to symplectic cohomology~\cite[Section 4]{ganatra_covariantly_functorial} ($\infty$-category of cochain complexes), 
	the wrapped Fukaya category ($\infty$-category of $A_\infty$-categories~\cite{tanaka-Aoo}), or
	the conjectural spectral wrapped Fukaya category ($\infty$-category of stable $\infty$-categories).
The operations we invert in this work leave these invariants unchanged, so by the universal property of ($\infty$-categorical) localizations, these invariants define a functor from $\weincrit$ (or, as in Remark~\ref{remark. tangential structures on Wein}, the version where sectors are equipped with appropriate tangential structures). Moreover, by Proposition~\ref{prop. localization inherits symmetric monoidal structure}, if the original invariants are symmetric monoidal (so we can supply data identifying invariants of a direct product to the tensor product of the invariants of the factors), then so is the induced functor out of $\weincrit$. Note that invariants of {\it Liouville} sectors are not expected to satisfy Kunneth formulas, so this is one reason to consider functors out of $\weincrit$ and $\weinstab$ (as opposed to their Liouville analogues).
	\item $\weincrit$ receives a map from $\weinstab$, and $\weinstab$ is known to encode the homotopy type of spaces of stabilized Liouville automorphisms (Remark~\ref{rem: geometric_model}). Thus, the invariants from \eqref{item. why oocat invariants} all receive actions of stabilized Liouville automorphisms. When applied to wrapped Fukaya categories, such actions give rise to analogues of the Seidel homomorphism in the sectorial setting~\cite{oh-tanaka}. Because the functor out of $\weinstab$ factors through $\weincrit$, the morphism spaces of $\weincrit$ give an upper bound on how sensitive such actions are to the homotopy types of stabilized Liouville automorphisms.
	\item The $\infty$-categorical symmetric monoidal structure on $\weincrit$ allows us to speak of $E_2,E_3,\ldots$ algebras in $\weincrit$. (Compare: In a symmetric monoidal 1-category, there is no difference between $E_n$ and $E_\infty$ algebras for $n \geq 2$.) The question of how to determine which stable $\infty$-categories -- even which chain complexes or spectra -- admit $E_n$ or $E_\infty$-algebra structures is a subtle problem. When a given invariant arises from a Weinstein sector $X$ (as its wrapped Fukaya category, say) the symmetric monoidal functor out of $\weincrit$ provides symplectic tools to produce $E_n$ or $E_\infty$ structures on such invariants by exhibiting such algebraic structures on $X$. So for example, Corollary~\ref{corollary. P inverted disks are commutative} shows that any symmetric monoidal invariant associated to $T^*D^n[P^{-1}]$ is an $E_\infty$-algebra. This gives strong evidence to the expectation that the symplectic cochains of $T^*D^n[P^{-1}]$ form an $E_\infty$-algebra in cochain complexes, or that its wrapped Fukaya category is a symmetric monoidal $A_\infty$-category, exhibiting the commutative ring structure on $\ZZ[1/P]$ -- and, over the sphere, the $E_\infty$ ring structure on $\SS[1/P]$ -- by purely geometric arguments. 
    We note that this ``strong evidence'' yields a proof if one can establish that the functors (called ``symplectic cochains'' and ``the wrapped Fukaya category'') out of the usual 1-category of Weinstein sectors are symmetric monoidal.
	\end{enumerate}

	\end{remark}

	\subsection{Conventions}\label{section. conventions}
	In this work, by an $\infty$-category, we mean a simplicial set satisfying the weak Kan condition (Definition~1.1.2.4 of~\cite{htt}). Synonyms include weak Kan complex~\cite{boardman-vogt-homotopy-invariant} and quasi-category~\cite{joyal-quasi-categories-kan-complexes-jpure}.

	There are also two conventions in the literature regarding the meaning of the term {\em localization} in the $\infty$-categorical context. For example, Definition~5.2.7.2 of~\cite{htt} demands that a localization $\cC \to \cD$ admit a fully faithful right adjoint. A more general definition, and the convention we adopt in this work, is as follows. Fixing a set of morphisms $W$ in $\cC$ (or, equivalently, a sub-$\infty$-category of $\cC$), we define a functor $\cC \to \cD$ to be a localization if it exhibits $\cD$ as the initial $\infty$-category receiving a map from $\cC$ and rendering every element of $W$ invertible. Equivalently, $\cC \to \cD$ is a localization if and only if: For any $\infty$-category $\cE$, the restriction map $\fun(\cD,\cE) \to \fun(\cC,\cE)$ is fully faithful with essential image those functors $\cC  \to \cE$ sending elements of $W$ to equivalences in $\cE$~\cite[Definition~6.3.1.9, tag 01MP]{kerodon}. Abusing terminology, we often  refer to $\cD$ as the localization of $\cC$ (with respect to $W$, or along $W$), and denote it by $\cC[W^{-1}]$. A concrete model for $\cC[W^{-1}]$ is obtained from the pushout of the diagram of simplicial sets
		\begin{equation}\nonumber
			\begin{tikzcd} 
			W \arrow{r}{}	\arrow{d}{}  
			& C  \\
			{|W|}
			\end{tikzcd}
		\end{equation}
	by taking a fibrant replacement of the pushout in the Joyal model structure. Above, $|W|$ is the Kan complex generated by the simplicial set $W$, which can also be modeled by Kan's $\text{Ex}^\infty(W)$ construction. 
		More on localization can be found in~\cite[Section 2.1]{tanaka-Aoo}, \cite[Remark 3.24]{LAST_categories}, and~\cite[\href{https://kerodon.net/tag/01ME}{Tag 01ME}]{kerodon}.
	
	\subsection*{Outline}	In Section \ref{sec: Liou_geo_background} we give some background on Liouville geometry, introduce the critical Weinstein category, and prove some helpful properties. Section \ref{sec: category} reviews the necessary material on localization functors. In Section \ref{sec: first_look_flexibilization} we describe the two P-flexibilization functors $X[P^{-1}]$ and $X \times (T^*D^n[P^{-1}])$, reviewing the construction from~\cite{Lazarev_Sylvan}.  We prove Theorem \ref{thm: intro_comparison} comparing these two functors in  Section \ref{sec: comparison_P_flexibilization} and we prove Theorem \ref{thm: intro_idempotent_functor} (that these functors are localizing) 	in Section \ref{sec: idempotency}.
	
	\subsection*{Acknowledgements}	
	The authors thank Yasha Eliashberg for helpful discussions, particularly around the movie construction (Proposition \ref{prop: Weinstein_Movie_construction}) and suggesting the use of Morse-Bott Weinstein structures. 	The first author was supported by the NSF postdoctoral fellowship, award \#1705128, and NSF grant \#2305392.     
	The first and second authors were partially supported by the Simons Foundation through grant \#385573, the Simons Collaboration on Homological Mirror Symmetry. 
	The third author was supported by a Texas State Univeristy Research Enhancement Program grant, a Sloan Research Fellowship, and by an NSF CAREER Grant under Award Number 2044557.

	\section{Liouville geometry background}\label{sec: Liou_geo_background}

	\subsection{Liouville and Weinstein sectors}

	\begin{defn}[Liouville sector]
		\label{defn:liou_sector}
		Fix an exact symplectic manifold $ ( X, \omega=d\lambda ) $ possibly with corners, together with the data, for each $ x \in \partial^i X $ in a codimension $i$ corner ($i \geq 1$), of a neighborhood $\nbhd(x)$ inside $X$ and a codimension-preserving symplectic submersion
		\eqn\label{eqn. corner projections}
		\pi_x \colon \nbhd(x) \to T^*[0, 1)^i.
		\eqnd
		We say this collection of data is a {\em Liouville sector} if it satisfies the following:
		\begin{enumerate}
			\item\label{item. finite type} $\lambda$ has finite type. This means that $X$ admits a proper, smooth function $X \to \RR_{\geq 0}$ which, outside some compact subset of $X$, is linear with respect to the Liouville flow of $X$.
			\item\label{item. splitting lambda} Each $\pi_x$ is flat, and $\lambda$ is split with respect to $\{\pi_x\}_{x \in \del X}$.
            (See Remark~2.2 of~\cite{LAST_categories}.)
			\item\label{item. projections compatible} If $y \in \partial^j X \cap \nbhd(x)$ with $j \leq i$, then on the overlap we have $\pi_y = \pi_{yx} \circ \pi_x$, where $\pi_{yx} \colon T^*[0, 1)^i \to T^*[0,1)^j$ is a projection to $j$ components (not necessarily respecting the order of coordinates).
		\end{enumerate}
		The splitting in (2) means that  $\nbhd(x)$ is a product of a fiber $F$ of $\pi_x$ and some neighborhood of $\pi_x(x)$ so that $\lambda|_{\nbhd(x)} =  \lambda^F + \pi_x^*\mathbf{p}d\mathbf{q}$. One can check that $\lambda^F$ renders $F$ as an open subset of a Liouville sector. 
		
		As in the boundaryless setting (also called a completed Liouville domain or a Liouville manifold),  the \emph{skeleton} $\skel X$ is the smallest attracting set for the negative Liouville flow.
	\end{defn}
	
	We will usually denote a Liouville sector by $(X, \lambda)$ or just $X$, leaving the family of projections $\{\pi_x\}$ implicit. 
	
	\begin{remark}
		Note that the splitting condition implies that any trajectory of the Liouville vector field $v_\lambda$ which begins away from $\partial X$ must remain away from $\partial X$. This definition of Liouville sector agrees with the notion from~\cite{ganatra_generation} of a \emph{straightened} Liouville sector with corners; see Section 12.2 of ~\cite{ganatra_generation}.
We also point out that we are naturally led to consider sectors with corners because 
many of constructions involve taking products of sectors and the product of two sectors with non-trivial boundary naturally has non-trivial corners
	\end{remark}
	
	\begin{notation}[$X_{cpt}$]
		\label{notation. domain of sector}
		We will sometimes 
		identify $X$ with a compact, codimension 0 submanifold $X_{cpt} \subset X$ for which $X \setminus X_{cpt}$ is the positive half of a symplectization of a contact manifold with (convex) boundary.
	\end{notation}
	
	\begin{definition}[Sectorial boundary]
		Let $X$ be a Liouville sector. Then the {\em sectorial boundary} of $X$ is the boundary of $X$ when considered as a smooth manifold with corners -- in other words, the union of all boundary and corner strata of $X$.
	\end{definition}

	As a consequence of Definition~\ref{defn:liou_sector}, after smoothing the corners, we have a splitting of a neighborhood of $\del X$:	
	\eqn\label{eqn. bordering}
	(H \times \mathbb{R}_x^+ \times \mathbb{R}_y,  
	\lambda_H + y dx).	
	\eqnd
	
	\begin{definition}
		\label{defn. sectorial divisor}
		We will call the identification~\eqref{eqn. bordering} of a neighborhood of $\partial X$ a \textit{bordering} and call the Liouville sector $H$ from~\eqref{eqn. bordering} the \emph{sectorial divisor} of $X$. It is a Liouville sector without boundary---{ie }the completion of a Liouville domain.
		We will use the notation 
		\eqnn
		[X, H]
		\eqnd
		to denote a sector with its sectorial divisor. 
	\end{definition}
	
	\begin{remark}
		The bordering condition~\eqref{eqn. bordering} implies that the Liouville vector field is tangent to $\partial X$ and in a neighborhood of $\partial X$, the zero locus of $v_\lambda$ is the product of $[0,1]$ with the zero locus of the Liouville vector field on $H$. 
	\end{remark}
	
	\begin{remark}
		There is another approach to dealing with a sector with corners---for example, immersing sectorial hypersurfaces to cover $\del X$~\cite{ganatra_generation}. This is similar to a common smooth-topology convention that treats the boundary of $[0,1]^2$ not as a topological circle, but as a disjoint union of four intervals. 
		
		We instead choose a smoothing of $\del X$ to associate a single $H$ (well-defined up to deformation of Liouville structure). Note that any sector with corners is equivalent (in the sense of Section~\ref{section. sectorial equivalences}, though not isomorphic) to its boundary-smoothing.	This is the same way in which $[0,1]^2$ is isotopy equivalent to $D^2$.
	\end{remark}
	
	In this paper, we will consider Liouville sectors with Morse-Bott-with-corners Weinstein structures.
	\begin{definition}
		A smooth function $f$ is \emph{Morse-Bott-with-corners} if, in a neighborhood $U$ of each critical point $p$, we can find coordinates $x_i$ centered at $p$ so that 
		\[
		f(x) = \sum_i f_i(x_i)
		\]
		for all $x \in U$, where each $f_i(x_i)$ is one of
		\begin{enumerate}
			\item $\pm x_i^2$
			\item a cutoff function which is zero for $x_i \le 0$ and has strictly positive derivative for $x_i > 0$
			\item $0$.
		\end{enumerate}
		If we additionally allow the possibility $f_i=x_i^3$ ({ie }a birth-death singularity), then we'll say $f$ is \emph{generalized Morse-Bott-with-corners}.
	\end{definition}

	\begin{notation}
		We denote the critical locus of $f$ by $\crit(f)$. 
	\end{notation}
	
	\begin{remark}
		If $f$ is Morse-Bott-with-corners, $\crit(f)$ is a smooth manifold with corners.
	\end{remark}
	
	\begin{definition}[Morse index]
		If $f$ is Morse-Bott-with-corners,
		we define the \emph{Morse index} of a connected component $C \subset \crit(f)$ to be the sum of the dimensions of the non-positive eigenspaces of the Hessian $Hf(p)$ for $p \in C$, or equivalently the number of coordinates $x_i$ above for which $f_i(x_i) \ne +x_i^2$. The index of $C$ does not depend on the choice of $p$.
	\end{definition}

	   In particular, if $C$ is a codimension 0 submanifold, then it has maximal Morse index, even if $\partial C$ is nonempty.

	\begin{definition}\label{def: Weinstein_sector}
		A Weinstein sector is a Liouville sector $(X, \lambda)$ so that
		\begin{itemize}
			\item  $v_\lambda$ is gradient-like for a Morse-Bott-with-corners function $\phi: X \rightarrow \mathbb{R}$
			\item $\phi$ can be arranged to be split with respect to the boundary projections $\pi_x$.
			In other words, a neighborhood of $\partial X$ has the form
			\eqnn
			(H \times \mathbb{R}_x^+ \times \mathbb{R}_y,  
			\lambda_H + y dx, \phi_H +	e^f\lVert y\rVert^2)	
			\eqnd			
			for a Weinstein sector $H$.
			
		\end{itemize}
		
	\end{definition}
	
	\begin{remark}
		In the case without corners, our definition is similar to the notion of a \emph{Morse-Bott*} Weinstein structure appearing in Starkston's work~\cite{Starkston_arboreal}. We will not attempt to address whether it is in fact equivalent.
		
		Another similar definition is given in~\cite{Eliashberg_revisited} and called an \textit{adjusted} Weinstein structure.
	\end{remark}
	
	\begin{definition}
		\label{defn:subcrit_component}
		\ \ 
		
		\begin{itemize}
			\item 
			A connected component $C \subset \crit(\phi)$ is called \emph{subcritical} if it either has Morse index less than $n$ or has free boundary in the sense that its boundary not entirely contained in $\partial X$.
			\item
			$C$ is called \emph{critical} if it is not subcritical.
			\item 
			A Weinstein sector is called subcritical if all components of the zero locus are subcritical. 
			
		\end{itemize}
	\end{definition}
	
	\begin{remark}
		In this paper, we further require that all critical components consist of isolated points, {ie }are already Morsified. 
		A general Morse-Bott Weinstein sector can be put into this form by a $C^0$-small Liouville homotopy.	Thanks to this assumption, any component intersecting the sectorial boundary is subcritical.
	\end{remark}
	
	\begin{figure}
		\centering
		\includegraphics[scale=0.2]{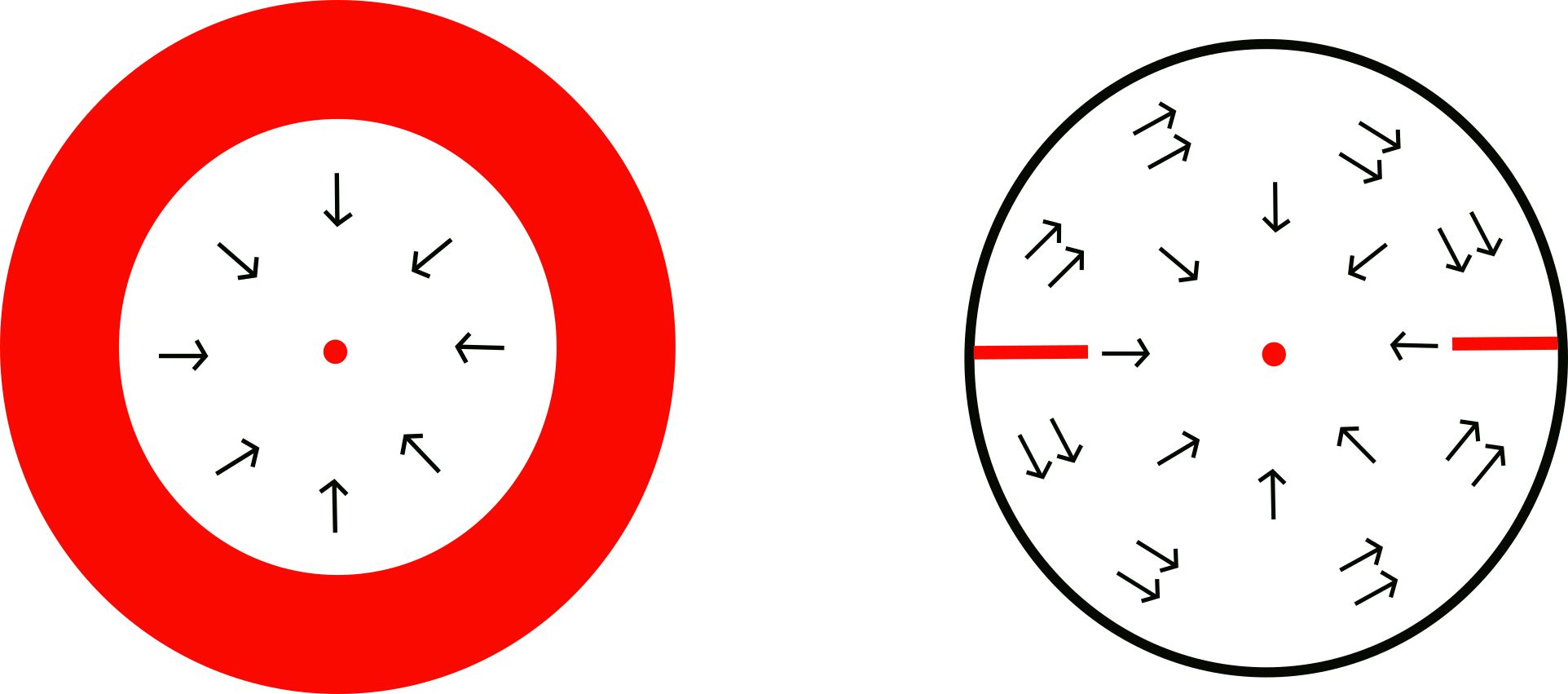}
		\caption{Two Morse-Bott Weinstein structures on $T^*D^n$, depicted via their corresponding vector fields on $D^n$; the zero locus of the vector fields is in red.
			The left figure has sectorial boundary $T^*S^{n-1}$, equipped with a Morse-Bott Weinstein structure having critical locus $S^{n-1}$. The right figure has the same sectorial boundary $T^*S^{n-1}$, but for which the boundary is equipped with a Morse Weinstein structure consisting of two points. 		
		}
		\label{fig: Weinstein_TDn}
	\end{figure}

	\begin{examples}[Cotangent bundles of disks]
		\label{ex: standard_structure_T^*D^n}
		There are two convenient Morse-Bott Weinstein structures on the symplectic manifold $T^*D^n$, both with sectorial divisor $T^*S^{n-1}$.

		The first arises from the standard Morse-Bott structure on the divisor $(T^*S^{n-1}, \lambda_{std})$, which has $S^{n-1}$ as the zero locus. We call this the {\em standard} Weinstein structure. 
		
		The second arises naturally when the divisor is given the Weinstein structure $(T^*S^{n-1}, \lambda_{Morse})$ nduced by a Morse function on $S^{n-1}$ with two isolated critical points of index $0$ and $n-1$.

		Both structures further contain an index $n$ critical point on the interior. See Figure \ref{fig: Weinstein_TDn}, where we depict the Liouville vector fields restricted to the zero-section $D^n$. Note that any vector field on $M$ has a canonical extension to a Liouville field on $T^*M$.
	\end{examples}
	
	\subsection{Stopped domains}
	\label{section. stopped domains are sectors}
	
	\begin{defn}
		
		\label{defn. stopped domain}
		Let $(X_0, \lambda_{X_0})$ and $(\Lambda, \lambda_\Lambda)$ be (compact) Liouville domains, together with a codimension 1 embedding  $i: \Lambda \hookrightarrow \partial_\infty X_0$ into the contact boundary $\partial_\infty X_0$ of $X_0$ that satisfies   
        $i^*\lambda_{X_0}= \lambda_{\Lambda}$.         
		We will call the pair 
		\eqnn
		(X_0, \Lambda)
		\eqnd
		a \textit{stopped} domain. 
		When both $X_0$ and $\Lambda$ are Weinstein, we call the pair a {\em Weinstein pair}. (We demand no compatibility between Weinstein Morse functions.)
	\end{defn}
	
	\begin{remark}
		In ~\cite{sylvan_partially_wrapped}, such a hypersurface $\Lambda$ in a contact manifold is called a \textit{stop} while in the Weinstein setting of ~\cite{Eliashberg_revisited}, it is called a \textit{Weinstein hypersurface}. 
	\end{remark}	
	
	A common maneuver in the world of Liouville geometry passes between a compact exact symplectic manifold-with-contact-boundary (where the Liouville vector field points outward along the boundary) and a complete, non-compact exact symplectic manifold-without-contact-boundary (obtained by attaching a cylinder along the Liouville vector field). There is a corresponding maneuver in the world of sectors allowing us to pass between (compact) stopped domains $(X_0,\Lambda)$ (Definition~\ref{defn. stopped domain}) and (non-compact) Liouville sectors $X$, as established in~\cite{ganatra_covariantly_functorial}.

	We review the constructions briefly.
	
	\begin{construction}[From sectors to stopped domains]
		\label{construction. sector to stopped domains}
		
		Fix a Liouville sector $X$ and let $H$ be the sectorial divisor (Definition~\ref{defn. sectorial divisor}). 
		To construct $X_0$, we consider $DT^*[-1,1] \cong T^*[-1,1]_{cpt}$ (the unit disk cotangent bundle) with the Weinstein structure induced by a vector field on $[-1,1]$ that is pointing towards $-1$ along $[-1,0)$ and vanishing along $[0,1]$. Note that the bordering~\eqref{eqn. bordering} allows us to identify $H \times T^*[0,1]$ with a neighborhood inside $X$.
		So we glue $H_{cpt} \times T^*[-1,1]_{cpt}$ to $X_{cpt}$ (Notation~\ref{notation. domain of sector}) along
		$H_{cpt} \times T^*[0,1]_{cpt}$ and call the resulting domain $X_0$. 	Note that $X$ has a proper inclusion into the completion of $X_0$.
		Next we observe that the Liouville form on $X_0$ restricts to the Liouville form on $H_{cpt} \cong H_{cpt} \times \{-1\} \subset H_{cpt} \times T^*[-1,1]_{cpt}$. Then $(X_0,H_{cpt})$ is a stopped domain.
		Note also that if $X$ is Weinstein, then so is $X_0$. We also note that the resulting Liouville stopped domain $X_0$ is not canonical since it relies on a choice of compact subset $X_0 \subset X$ whose complement is a symplectization of the contact boundary of $X_0$. However we show in Proposition \ref{prop: homotopy_sector_to_stopped} below that any two such  Liouville domains have Liouville homotopic structures.
	\end{construction}

	\begin{construction}[From stopped domains to sectors]
		\label{construction. stopped domains to sectors}
		As explained in~\cite{Eliashberg_revisited}, any stopped domain $(X_0, \Lambda)$ can be converted into an Liouville sector that we denote $\overline{(X_0, \Lambda)}$. Namely, we consider $T^*[0,1]_{cpt}$ with the Weinstein structure induced by a vector field on  $[0,1]$ that vanishes on $[0,1/4]$, pointing towards $1$ along $(1/4, 1/2)$, has an index $1$ critical point at $1/2$, and pointing towards $0$ along $(1/2, 1]$. Then we attach  $\Lambda \times T^*[0,1]_{cpt}$ to the stopped domain $(X_0, \Lambda)$ along $\Lambda \times  T^*_1[0,1]_{cpt}$. Here $T^*_1[0,1]_{cpt}$ is  the unit cotangent fiber over $1 \in [0,1]$; the gluing identifies these fibers with small integral curves of the Reeb vector field on $X$.
		See Figure \ref{fig: stop_to_sector}. We define $\overline{(X_0,\Lambda)}$ by completing. When $(X_0,\Lambda)$ is Weinstein, so is $\overline{(X_0,\Lambda)}$.
	\end{construction}
	
	\begin{remark}
		In particular, $X_0$ is a Weinstein \textit{subdomain} of $\overline{(X_0, \Lambda)}$; see Section \ref{sec: subdomain_def} for a definition.
		The Liouville vector field on $\overline{(X_0, \Lambda)}$ has zero locus that corresponds to the zero locus of $X_0$ and $([0,1/4] \cup \{1/2\}) \times C$, where $C$ is the zero locus on $\Lambda$. If $\Lambda$ has isolated critical locus, then so does the resulting   Weinstein structure on $\overline{(X_0, \Lambda)}$.
	\end{remark}
	
	\begin{remark}
		In general, if $(X_0, \Lambda)$ arises from a sector $X$ as in Construction~\ref{construction. sector to stopped domains}, then $\overline{(X_0, \Lambda)}$ is Weinstein homotopic to (a slightly larger version of) $X$; see Section \ref{sec: homotopies} below. 
	\end{remark}

	\begin{figure}
		\centering
		\includegraphics[scale=0.2]{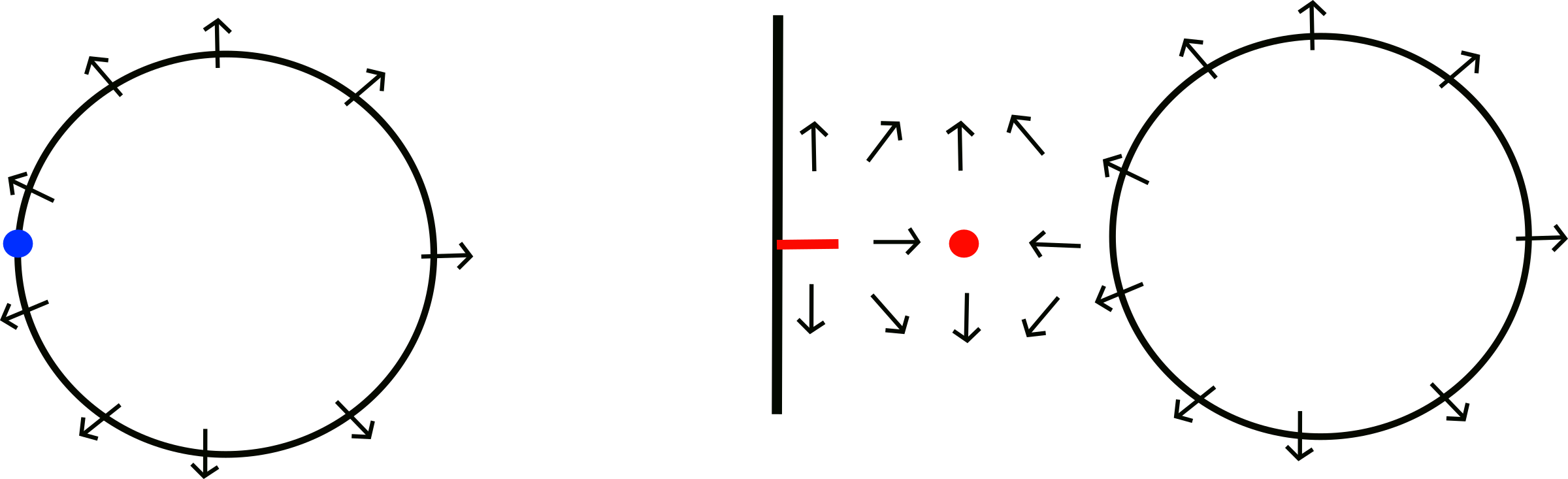}
		\caption{Converting a stopped domain, with stop in blue, to a sector, with additional critical points in red.	
		}
		\label{fig: stop_to_sector}
	\end{figure}

	\begin{examples}\label{examples: subcritical_looseness}
		Let $X$ be a subcritical Weinstein sector (Definition~\ref{defn:subcrit_component})
		with associated domain $X_0$ and stop $\Lambda$.  Since $X$ is a subcritical sector, $X_0$ is a subcritical domain. However, the stop $\Lambda$ itself need not be subcritical. For example, the sector $\Lambda \times \mathbb{C}_{Re \ge 0} = [\Lambda\times D^2, \Lambda \times \{1\}]$ is subcritical for any Weinstein domain $\Lambda$
		since the Weinstein sector $[D^2, \{1\}]$ has only critical point of index $0$, lying in its boundary. 
		A subcritical sector has no Lagrangian co-cores since it has no isolated index $n$ critical points but it does have a collection of Lagrangian `linking' disks of its sectorial divisor. As explained in Remark \ref{rem: linking_disks}, 
		these disks can be realized as co-cores of a homotopic Weinstein structure; see~\cite{ganatra_generation} for a definition of linking disks. 
		
		If $X$ is a subcritical sector, then  
        $\Lambda \subset\partial X_0$ is a loose Weinstein hypersurface, as defined in~\cite{Eliashberg_revisited}. We will prove this only in the case when $X$ is obtained by attaching subcritical handles to $\Lambda \times  (D^2, \{1\})$. We claim that $\Lambda \times\{1\}\subset\Lambda\times D^2$ is a loose Weinstein hypersurface. To see this, we proceed by induction. Let $H^n$ be a Weinstein handle of $\Lambda^{2n}$ and $C^n$ be the core of this handle. Then $C^n$ is a Legendrian disk in the boundary of $H^n \times D^2$ (a handle of $\Lambda\times D^2$) that intersects the belt sphere of $H^n \times D^2$ exactly once; so $C^n$ 
		is loose relative to its boundary by the criteria in~\cite{CE12} and adding subcritical handles preserves loose-ness.        
	\end{examples}
	
	\subsection{Products of sectors}
	
	Given two Liouville sectors $(X, \lambda_X), (Y, \lambda_Y)$, we define the product sector $X \times Y$ to be $(X \times Y, \lambda_X + \lambda_Y)$. 
	If $X$ and $Y$  further admit Weinstein  functions $\phi_X$ and $\phi_Y$, then  $\phi_X + \phi_Y$ is a Weinstein function on $X \times Y$.
	
	\begin{remark}
		Note 
		that critical points of $\phi_X+ \phi_Y$ correspond to pairs of critical points of $\phi_X$ and $\phi_Y$. Furthermore,  the unstable manifold of the critical point $p$ of $\phi_X+ \phi_Y$ corresponding to a pair of critical points of $\phi_X, \phi_Y$ is the product of the associated unstable manifolds. 
	\end{remark}
	
	\begin{example}
		If $X, Y$ are sectors associated to stopped domain domains $(X_0, H_X), (Y_0, H_Y)$, then the associated stopped Weinstein domain to $X \times Y$ is $(X_0 \times Y_0, X_0 \times H_Y \coprod_{H_X \times H_Y } H_X \times Y_0)$. 
	\end{example}
	
	\begin{definition}
    In the following, let $T^*D^k$ be the canonical Morse-Bott Weinstein structure $\lambda_{std} = \sum_{i=1}^n p_i dx_i$
 that is associated to the zero vector field on $D^k$. 
		For any integer $k \geq 1$, the sector $X \times T^*D^k$ is called a {\em stabilization} of the sector $X$. 
	\end{definition}

	\subsection{Strict proper inclusions}
	From here through Section~\ref{section. sectorial equivalences}, we discuss maps between Liouville and Weinstein sectors. Here is the most basic class:
	
	\begin{definition}
		\label{defn. strict proper inclusion}
		For Liouville sectors $(X, \lambda_X), (Y, \lambda_Y)$, a {\em strict proper inclusion} otherwise known as a {\em strict proper embedding}, is a smooth embedding $i: (X, \lambda_X) \hookrightarrow (Y, \lambda_Y)$ that is proper, and that strictly preserves the Liouville forms:  $i^*\lambda_Y  = \lambda_X$.
	\end{definition}
	
	\begin{remark}
		Note we allow the sectorial boundary of $X$ to intersect the sectorial boundary of $Y$. (Compare with Convention~3.1 of~\cite{ganatra_covariantly_functorial}.) 
		If $X$ is contained in the interior of $Y$, then  the complement  $Y \setminus i(X)$ is also a Liouville sector,  with a strict proper inclusion into $Y$; if $Y$ is further a  Weinstein sector, then so is $Y\setminus i(X)$.
	\end{remark}

	\begin{remark}
		\label{remark. cocores of X and Y}
		If $i: X \hookrightarrow Y$ is a strict proper inclusion of   Weinstein sectors, then the index $n$ Lagrangian co-cores of $X^{2n}$ are identified with a subset of the Lagrangian co-cores of $Y$. Furthermore, since the Liouville vector field on $Y$ is tangent to $\partial X$, points in $Y \setminus i(X)$ cannot flow into $i(X)$. So the Lagrangian co-cores of $Y$ are one of the co-cores of $i(X)$ (and hence identified with a cocore of $X$) or are entirely contained in $Y \setminus i(X)$.
	\end{remark}

	\begin{remark}\label{remark. strict is monoidal}
		We note that the formation of products is compatible with strict proper inclusions. That is, if $i$ and $i'$ are strict proper embeddings, so is the product $i \times i'$. 
	\end{remark}

	\subsection{Liouville deformations}\label{sec: homotopies}
	In this paper, we will also consider certain classes of non-strict Liouville embeddings, which we will call just Liouville embeddings. These embeddings allow certain deformations of the Liouville form, which we now discuss.
	
	\begin{definition}[Homotopies/deformations of Liouville structures]
		\label{defn. deformations of liouville structures}
		Let $\{\lambda_t\}_{t \in [0,1]}$ be a smooth, 1-parameter family of Liouville structures on $Y$.
		As usual, we will say that the family is {\em exact} if $\lambda_t = \lambda_0 + dh_t$ for some smooth family of smooth functions $h_t$. Throughout the paper, unless explicitly stated otherwise, it is assumed that every family is exact.
		
		Finally, we demand a tameness condition on our families at infinity: We demand there exists a {\em proper} smooth function $R: Y \times [0,1] \to \RR_{\geq 0}$ and a single compact subset $K \subset Y$ such that, for all $t \in [0,1]$, $R_t$ is $\lambda_t$-linear outside of $K$---that is, we demand that $d(R_t)( v_{\lambda_t})$ equals $R_t$ outside of K.
		
		We will call such a family---exact, and satisfying the tameness condition at infinity---a {\em deformation}, or equivalently a {\em homotopy}, of $\lambda = \lambda_0$.

		Exact deformations will further be called (in order of increasing restrictiveness):
		\begin{itemize}
			\item {\em Bordered} if for each $t$, $\lambda_t$ respects the splitting~\eqref{eqn. bordering}.	    \item {\em Interior} if $\lambda_t$ is constant ({ie }$t$-independent) near $\del Y$. 
			\item {\em Compactly supported} if there exists a compact set $K \subset Y\setminus \del Y$ for which $\supp(\lambda_t - \lambda_0) \subset K.$ 
		\end{itemize}
	\end{definition}

	\begin{remark}
		\label{remark. bordered homotopy}
		For a deformation $\lambda_t$ to be bordered means that $\lambda_t$, in a neighborhood of $\del Y$, is a direct product of deformations -- a deformation of Liouville structure of the divisor $H$, and a constant (non-)deformation of the structure on $T^*[0,1]^k$ -- with respect to the splitting in~\eqref{eqn. bordering}.
	\end{remark}
	
	\begin{remark}
		We warn the reader that we take ``compact support'' to be a condition checked on $Y \setminus \del Y$ (not on $Y$ itself). In particular, if $\lambda_t$ is a compactly supported deformation, the functions $h_t$ can be chosen to vanish near $\del Y$ and $\lambda_t$ is constant ({ie }$t$-independent) near this boundary.
		
		Put another way, we abusively use compactly supported to mean ``interior and compactly supported.''
	\end{remark}
	
	\begin{remark}(On Weinstein homotopies)
		\label{remark. homotopies could be weinstein}
		For clarity,  we say that a \emph{Weinstein homotopy} is a family $\lambda_t$ of Liouville forms admitting generalized Morse-Bott-with-corners Lyapunov functions. More generally, the reader can  consider a class of Liouville structures with Lyapunov functions whose singularities are invariant under products with other sectors, {ie }if  $(X, \lambda_t)$ is a Weinstein homotopy and $Z$ is a Weinstein sector, then  $(X, \lambda_t) \times Z$ is also a Weinstein homotopy. For any such choice, all of our results, {eg }Proposition \ref{prop: Weinstein_Movie_construction} and Theorem \ref{thm: comparison}, involve only Weinstein homotopies (assuming the regular Lagrangians $L \subset T^*D^n$  which are the input for Theorem \ref{thm: comparison} are defined using the same class of generalized Weinstein structures). 
	\end{remark}

	We have already seen that we can move between sectors and stopped domains (Constructions~\ref{construction. sector to stopped domains} and~\ref{construction. stopped domains to sectors}). The following proposition makes precise the idea that these operations are invertible up to a natural notion of equivalence; it further shows that these operations respect families of Liouville deformations/homotopies.	
	\begin{proposition}\label{prop: homotopy_sector_to_stopped}
		Let $X$ be a Liouville/Weinstein sector, and $(X_0, F)$ its associated stopped domain (Construction~\ref{construction. sector to stopped domains}). Then $X$ is interior Liouville/Weinstein homotopic to $X' := \overline{(X_0,F)}$ (Construction~\ref{construction. stopped domains to sectors}),
		and this interior homotopy can be chosen to be supported in a standard neighborhood of $\partial X$.
		Similarly, any bordered Weinstein homotopy $X_t$ \emph{between} Weinstein sectors associated to stopped domains is homotopic through interior homotopies to a homotopy \emph{through} Weinstein sectors associated to stopped domains.
	\end{proposition}
	\begin{remark}
	     See Lemma 2.32 and the discussion in Section 2.8 of \cite{ganatra_covariantly_functorial} for a similar statement, which involves the more general notion of Liouville deformation applied to the more general notion of unstraightened Liouville sectors; our result involves interior homotopies applied to straigthened Liouville sectors.  
	\end{remark}
	
	\begin{proof}
		Consider a Weinstein sector $X$, so that in a neighborhood of the sectorial boundary the Weinstein structure agrees with $F \times T^*[0,1]$, where we take the canonical Morse-Bott Weinstein structure on $T^*[0,1]$ induced by the zero vector field on $[0,1]$. 
        (Explicitly, the Weinstein structure is induced by the Morse-Bott function $p^2$ in the standard $(q,p)$ coordinates of $T^*[0,1]$.) Then  there is an interior Weinstein homotopy to $X'$ so that in a neighborhood of the sectorial boundary the Weinstein structure agrees with $F \times T^*[0,1]$, where the Weinstein structure on $T^*[0,1]$ is induced by a vector field on $[0,1]$ that is zero on $[0,1/4]$, pointing towards $1$ on $(1/4, 1/2)$, has an index $1$ critical point at $1/2$, is pointing towards $0$ on $(1/2, 1)$. See Figure \ref{fig: sector_to_stopped_domain}. 
		We note that Weinstein structure $X'$ is induced by a stopped domain $(X_0, F)$, where $X_0 \subset X'$ is a Weinstein subdomain.  The second claim follows from a parametrized version of the proof of the first claim. 
	\end{proof}

	\begin{remark}\label{rem: linking_disks}
		For every index $n-1$ critical point of $F^{2n-2}$, there is an index $n$ critical point in $(X')^{2n}$, lying over the index $1$ critical point on $[0,1]$ and the co-cores of these critical points are called the \textit{linking disks} of the sectorial divisor $F$. 
	\end{remark}
	\begin{figure}
		\centering
		\includegraphics[scale=0.2]{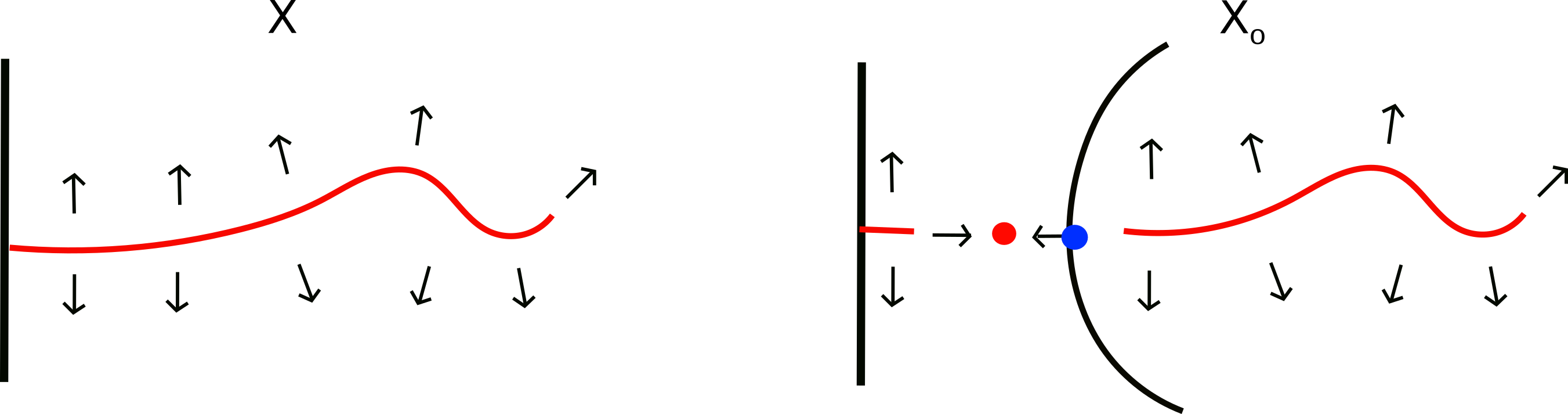}
		\caption{Homotoping an arbitrary sector to a (sector induced by a) stopped domain, with stop in blue and additional critical point in red.	
		}
		\label{fig: sector_to_stopped_domain}
	\end{figure}

	Next, we show that any bordered homotopy can be converted into an interior homotopy on a slightly larger sector which agrees with the original homotopy away from the sectorial boundary.

	\begin{proposition}\label{prop: bordered_to_interior_homotopy}
		Let  $[Y,F]$ be a sector with sectorial divisor $F$ and $[Y', F]$ be an enlargement by gluing $(F, \lambda_{F, 0})  \times (T^*[-1,1], pdq)$ to $Y$ along $F \times T^*[0,1]$. 
		\begin{enumerate}
			\item For every  (abstract) Liouville homotopy $(F, \lambda_{F, t})$ of $F$, one may choose a bordered Liouville homotopy
			$[(Y', \lambda_{Y',t}) ,(F, \lambda_{F, t})]$ of the sector $Y'$ which is constant on $Y$. 
			\item For every bordered Liouville homotopy $(Y, \lambda_{Y,t})$, one may choose an interior Liouville homotopy 
			$(Y', \lambda_{Y,t,int})$ of $Y'$ which agrees with $(Y, \lambda_{Y,t})$ on $Y$.	
		\end{enumerate}
	\end{proposition}
	\begin{remark}
		Instead of attaching $F \times T^*[-1,0]$ to $Y$ to form $Y'$ and modifying $Y'$ in $F \times T^*[-1,0]$, we can 
		identify
		a neighborhood of $\partial Y$ in $Y$ with $F \times T^*[0,1)$ and apply Proposition \ref{prop: bordered_to_interior_homotopy}
		to this neighborhood (without affecting the part of $Y$ away from $\partial Y$). In particular, we can assume that $Y'$ is $Y$. In this way, we have the following diagram, which we include for readability. 
		\begin{equation}\label{eqn.bordered_to_interior}
		\begin{tikzcd} 
		& 	Def(F) \arrow{d}{Prop \ \ref{prop: bordered_to_interior_homotopy} \ (1)} & &	\\
		Interior(Y) \arrow{r} & Bordered(Y) 
		\arrow{rr}{Prop \ \ref{prop: bordered_to_interior_homotopy} \ (2)} & & Interior(Y) 
		\end{tikzcd}
		\end{equation}
		The vertical downward arrow is a section of the natural forgetful map taking a bordered deformation of $Y$ to a deformation of the sectorial divisor $F$. Though we do not prove this here, there is a natural topology we may give to all sets in the diagram (see Section~2.7 of~\cite{LAST_categories}) for which the functions of Proposition~\ref{prop: bordered_to_interior_homotopy} are continuous, and for which the horizontal arrows of~\eqref{eqn.bordered_to_interior} are homotopy equivalences.
	\end{remark}
	
	The proof of this result will use the following construction from ~\cite[Section~3.3]{Eliashberg_revisited}, which we will use repeatedly in this paper. 
	\begin{construction}[Movie construction]
		Let $\{\lambda_t\}_{t \in [0,1]}$ be a Liouville homotopy on $Y$. 
		Then we define a Liouville sector structure $\lambda_{movie}$ on the manifold $Y \times T^*[0,1]$ following~\cite[Section~3.3]{Eliashberg_revisited}. 	Namely, if $\lambda_t = \lambda_0 + dh_t$ for a function $h_t: Y \rightarrow\mathbb{R}$ (constant in $t$ near $0, 1$), 	then  $\lambda_{movie} = \pi_Y^*\lambda_0 + \pi_{T^*[0,1]}^*\lambda_{T^*[0,1]} +  dh$, where $h: Y\times [0,1] \rightarrow\mathbb{R}$ is defined by $h(t, y ) := h_t(y)$ and $\pi_Y: Y\times[0,1]\rightarrow Y$ and $\pi_{T^*[0,1]}: Y\times T^*[0,1]\rightarrow T^*[0,1]$ 	are projections. 
		See Section 2.2 of \cite{LAST_categories} for a proof that this is a sector. 
	\end{construction}

	\begin{proof}[Proof of Proposition~\ref{prop: bordered_to_interior_homotopy}.]
	For part 1,  we note that for each fixed $t$, we can consider the restricted Liouville homotopy $(F, \lambda_{s})$ for $0 \le s \le t$ between $F_{\lambda_0}$ and $F_{\lambda_t}$; we can rescale so that this homotopy is parametrized by $[0,1]$.  
    Then we can form the movie construction     ~\cite[Section~3.3]{Eliashberg_revisited} associated to this homotopy to produce a     
 Liouville sector structure on $F \times T^*[0,1]$ that we call $(F \times T^*[-1,0], \lambda_{movie, t})$. This sector structure looks like $(F, \lambda_{F, t}) \times T^*[-1, -1+\epsilon]$ and $(F, \lambda_{F, 0}) \times T^*[-\epsilon, 0]$ near its sectorial boundary. 		Then we can append  $(F \times T^*[-1, 0], \lambda_{movie, t})$ to $(Y, (F, \lambda_{\Lambda, 0}))$ along $F \times  0$. By varying $t$, we obtain the bordered deformation $(Y', \lambda_{Y,t})$ of $Y'$. 

For part 2, a bordered Liouville homotopy $(Y, \lambda_{Y,t})$ gives a Liouville homotopy $(F, \lambda_{F,t})$ of the sectorial divisor $F$ of $Y$. Then we can form the `flipped' movie construction of the homotopy $(F, \lambda_{F,t})$  to get $(F \times T^*[-1,0], \lambda_{movie, t})$, which looks like  $(F, \lambda_{F, 0}) \times T^*[-1, -1+\epsilon]$ and $(F, \lambda_{F, t}) \times T^*[-\epsilon, 0]$ near its sectorial boundary. Then we can append  $(F \times T^*[-1, 0], \lambda_{movie, t})$ to $(Y, (F, \lambda_{\Lambda, 0}))$ along $F \times  0$ to get an interior deformation $(Y', \lambda_{Y,t, int})$ of $Y'$ which agrees with $(Y, \lambda_{Y, t})$ on $Y$.  
	\end{proof}
	\begin{remark}
		Since these two constructions are appending the movie construction and the `flipped' movie construction, the concatenation of these homotopies is an interior homotopy of $Y$ which is homotopic through interior homotopies to the constant homotopy. 
	\end{remark}
	\begin{remark}
		In this paper, we will mostly use bordered or interior homotopies but not compactly supported homotopies. By the movie construction (Proposition~\ref{prop: bordered_to_interior_homotopy}), any bordered homotopy can be converted into an interior homotopy. Then by Moser's theorem, any interior homotopy can be converted into a compactly supported homotopy; however the resulting compactly supported homotopy is not very explicit and so we prefer to work with interior homotopies. 
	\end{remark} 
	
	Next to prove a Weinstein version of Proposition \ref{prop: bordered_to_interior_homotopy}, we first construct a modified Weinstein movie. 
	
	\begin{proposition}\label{prop: Weinstein_Movie_construction}(Weinstein movie construction)
		Suppose that $(X, \lambda_t = \lambda_0 + dh_t)$ is a Liouville homotopy between Weinstein structures $(X, \lambda_0, \phi_0)$,  $(X, \lambda_1, \phi_1)$ and suppose that the Lagrangian co-cores of  $(X, \lambda_1, \phi_1)$  are $C_1$. Then there is a function $F: T^*[0,1]\rightarrow\mathbb{R}$  so that $(X\times T^*[0,1], \lambda_{movie}^{Weinstein}:=\lambda_{movie}+dF)$ admits a Weinstein structure whose only Lagrangians co-cores are $C_1 \times T^*_{3/4} [0,1]$. If $(X, \lambda_t), t\in [0,1]$ is a Weinstein homotopy, then  $(X\times T^*[0,1], \lambda_{movie, t}^{Weinstein})$, 
		constructed using the restricted homotopy 
		$(X, \lambda_s), s\in [0,t]$, is also a  Weinstein homotopy on $X \times T^*[0,1]$.

	\end{proposition}

	\begin{proof}
		Suppose that $\phi_0, \phi_1$ are Lyapunov functions on $(X, \lambda_0)$ and $(X,\lambda_1)$ respectively that have linear growth rate $d\phi_i(Z_i) = \phi_i$ (outside a compact subset of $M$)
		and $\lambda_t - \lambda_0 = dh_t$. 
		We will explain how to construct the Weinstein structure  $(\lambda_\mathrm{movie}^{Weinstein}, \Phi)$ over the region where $\frac{\partial}{\partial t}\lambda_t$ is nonzero, which after reparametrizing we take to be $[\frac13,\frac23]$. There, the only requirement for the Lyapunov condition is that $d\Phi(\lambda_\mathrm{movie}) > \varepsilon > 0$. To begin, pick a family $r_t$ of symplectization coordinates for $\lambda_t$ with $r_0 = \phi_0$ and $r_1 = \phi_1$, again assuming this family is constant on $[0,1/3]\cup [2/3,1]$. 
		Let $f\colon[0,1]\to\R_{\ge0}$ be a Morse-Bott-with-corners function which is zero near $\{0,1\}$, has a unique index 1 critical point at $\frac34$, and has non-zero, constant gradient $\xi = \nabla f$ in $[\frac13, \frac23]$ (that points toward the critical point at $\frac34$). 
	We will consider the induced functions $f$ (by abuse of notation) and $\tau(\xi)$ on $T^*[0,1]$ (the former by pullback, the latter by pairing the vector field $\xi$ on $[0,1]$ with a covector $\tau$). Consider the modified movie form  $\lambda_\mathrm{movie}^{Weinstein}$ and function $\Phi$
		\begin{eqnarray*}
		\lambda_\mathrm{movie}^{Weinstein} &=& \lambda_t + \tau dt + dh + ad (\tau(\xi)) = 
		\lambda_\mathrm{movie} + ad (\tau(\xi))
		\\
		\Phi &=& r_t + \tau^2 + Af
		\end{eqnarray*}
		Then $( \lambda_\mathrm{movie}^{Weinstein}, \Phi)$ is a Weinstein pair for $a>0$ sufficiently small and $A>0$ sufficiently large depending on $a$; here $h(x, t) = h_t(x)$. So in the notation of the statement of this proposition, $F = \tau(\xi): T^*[0,1]\rightarrow\mathbb{R}$.

		To see this, we first note that the  corresponding Liouville vector field is
		\[
		Z_\mathrm{movie} = Z_t + (\tau - \dot h - a\dot\xi \tau)\partial_\tau + a\xi,
		\]
        Indeed, this Liouville vector differs from the Liouville vector field $Z_t + (\tau-\dot h)\partial_\tau $ for the usual movie construction         by the Liouville vector field on $T^*[0,1]$ associated to            
       $ad(\tau(\xi))$. We claim that the latter is $a(\xi  - \tau \dot\xi \partial_\tau)$. Indeed, 
$d\tau(\xi) = \xi d\tau + \tau \dot \xi dt$ and hence 
$$dt\wedge d\tau(
\xi  - \tau \dot\xi \partial_\tau
, \_) =\xi d\tau + \tau \dot \xi dt
$$

Next, we note that	\begin{equation}\label{eq:Weinstein_cond_movie}
		d\Phi(Z_\mathrm{movie}) = d_Mr_t(Z_t) + 2\tau(\tau - \dot h -a\dot\xi \tau) + ad_tr_t(\xi) + Aadf(\xi)
		\end{equation}
		
		First, we observe that on  $[0,  1/3]$, the Liouville homotopy is constant and so $dh =0$. So on $M \times T^*[0,1/3]$, the structure
		$( \lambda_\mathrm{movie}^{Weinstein}, \Phi)$ is the product of the Weinstein structure $(\lambda_M, \phi_0)$ on $M$ and the Morse-Bott Weinstein structure on 	$T^*[0,1/3]$ associated to the Morse-Bott vector field $\xi$ on $[0,1/3]$ as discussed in Example 11.12 of \cite{CE12}.
		Concretely, here 
		\begin{equation}\label{eq:Weinstein_cond_movie2}
		d\Phi(Z_\mathrm{movie}) = d_Mr_t(Z_t) + 2\tau^2(1  -a\dot\xi)  + Aadf(\xi)
		\end{equation}
		since $\dot h$ and $d_t r_t$ vanish. 
		The key is that for sufficiently small $a$, the middle term is a positive multiple of $\tau^2$ and so is bounded from below by a positive multiple of the norm squared of $ (\tau  - a\dot\xi \tau)\partial_\tau$. The first and last term satisfy the Lyapunov inequality using the facts that $(M, \lambda_0,\phi_0)$ is Weinstein and the fact that $\xi$ is the gradient of $f$. The analogous result holds on $M \times T^*[2/3,1]$, where the homotopy and $r_t$ is also constant  in $t$.

		Next, we consider the pair
		$( \lambda_\mathrm{movie}^{Weinstein}, \Phi)$ on $M \times T^*[1/3, 2/3]$, where $\xi$ is bounded away from zero. We first observe that the compact region of $M\times  T^*[0 , 1]$ bounded by $r_t=C$ and $\tau^2=D$ with $D$ large compared to $C$ is attracting for the negative Liouville flow and hence it suffices to prove the Lyapunov inequality in this region. 
		This region is attracting because all points in $M \times T^*[0,1]$ flow into the non-compact region $\{r_t \le C\} \subset M\times T^*[0,1]$ for some large $C$ by considering just the $Z_t$ component of $Z_{movie}$. The projection of the region $\{r_t \le C\}$ to $M$ is compact and so $\dot h$ is bounded in $\{r_t \le C\}$. Therefore, for all points with sufficiently large $\tau$-coordinate, say $D$, we have that $\tau - \dot h - a\dot\xi \tau$ is positive (assuming that $a$ is sufficiently small), and similarly for all points with sufficiently negative $\tau$-coordinate.  
		In this compact region, all the terms of \eqref{eq:Weinstein_cond_movie} are bounded, and the last term is positive 
		and bounded away from zero, since we further assume that we are in
		$M \times T^*[1/3, 2/3]$. So we can make the whole equation positive and bounded away from zero by making $A$ sufficiently large. This proves the Lyapunov inequality since $|Z_{\movie}|, |df|$ are both positive in this region.

		Since $\xi$ is non-constant in $[1/3, 2/3]$, the Weinstein Lyapunov function $\Phi$ has no critical points in $M\times T^*[1/3, 2/3]$. In $M\times T^*[0, 1/3]$ and $M\times T^*[2/3,1]$, we have the product Weinstein structure. So the only critical points of maximal index correspond to 
        $(x, 3/4) \in 
        M\times\{3/4\} \subset M\times T^*[0,1]$, where $x$ is a maximal index critical of $(M, \lambda_1, \phi_1)$. The Lagrangian co-core of this critical point is the product of co-core of $p$ and the co-core of $\{3/4\} \subset T^*[0,1]$, which is $T^*_{3/4}[0,1]$. 
		This finishes the proof of the first claim in the proposition. 
		
		For the second claim, we consider the Weinstein movie structure	$(X\times T^*[0,1], \lambda_{movie, t}^{Weinstein})$ that is 
		constructed using the restricted homotopy 
		$(X, \lambda_s), s\in [0,t]$ (and  reparametrizing $[0,t]$ to $[0,1]$); we call this the generalized Weinstein movie construction since $(X, \lambda_t)$ may be generalized Weinstein. 
		Note that for all $t$, the Liouville vector field on $(X\times T^*[1/3,2/3], \lambda_{movie, t}^{Weinstein})$ has no zeroes while the Weinstein structure 
		$(X\times T^*[0,1/3], \lambda_{movie, t}^{Weinstein})$ is just 
		$(X, \lambda_0) \times T^*[0, 1/3]$ (with the standard Weinstein structure on structure on $T^*[0,1/3]$) and  the Weinstein structure 
		$(X\times T^*[2/3,1], \lambda_{movie, t}^{Weinstein})$ is just 
		$(X, \lambda_t) \times T^*[2/3, 1]$ (with the Weinstein structure on structure on $T^*[2/3, 1]$ induced from the vector field $\xi$). Since 		$(X, \lambda_t)$ is a Weinstein homotopy, so is 	$(X, \lambda_t) \times T^*[2/3, 1]$ since by assumption Weinstein homotopies are preserved by stabilization. 
	\end{proof}
	
	Now we use this Weinstein movie construction to prove the Weinstein analogs of Proposition Proposition \ref{prop: bordered_to_interior_homotopy}. 
	
	\begin{proposition}\label{prop: bordered_to_interior_homotopy_Weinstein}
		\ \ 
Let $[Y, F]$ be a Liouville sector and let 
$[Y', F]$ be an enlargement by gluing $(F, \lambda_{F, 0})  \times (T^*[-1,1], pdq)$ to $Y$ along $F \times T^*[0,1]$ as in Proposition \ref{prop: bordered_to_interior_homotopy}.         
		\begin{enumerate}
			\item 
			If $[Y, F]$ is a Weinstein sector and $(F, \lambda_{F, t})$ is a Weinstein homotopy with co-cores $C_{F_t}$ (at times $t$ when $\lambda_{t}$ is a Weinstein structure), there is a  bordered Weinstein homotopy $[(Y', \lambda_{Y', t}), (F, \lambda_{F,t})]$ constant on $Y$ and with the Lagrangian co-cores in $F\times T^*[-1,0]$ equal to $C_{F_t} \times T^*_{-3/4} [-1,0]$. 	
			\item
			If  $[(Y, \lambda_{Y,t}), (F, \lambda_{F,t})]$ is a bordered Weinstein homotopy, then there is an interior Weinstein homotopy
			$[(Y', \lambda_{Y,t,int}), (F, \lambda_{F, 0}) ]$ whose only Lagrangian co-cores in $F  \times T^*[-1,0]$ equal to $C_{F_0} \times T^*_{-3/4}[-1,0]$. 
			\item Furthermore, if in part 2), $(Y, \lambda_{Y,t})$ is a sector associated to the stopped domain $((Y_0, \lambda_t), (F,\lambda_{F,t}))$, 			
			then $(Y', \lambda_{Y,t,int})$ is associated to a stopped domain $((Y_0, \lambda_t)', F_0)$, and  $(Y_0, \lambda_t)'$ is a slight enlargement of the domain $(Y_0, \lambda_t)$ which has the same Lagrangian co-cores as $(Y_0, \lambda_t)$.
		\end{enumerate} 
	\end{proposition}
	\begin{proof}
		As in Proposition \ref{prop: bordered_to_interior_homotopy}, 
		For part 1), we first form the (generalized) Weinstein  movie $(F\times T^*[-1,0], \lambda_{F, movie, t})$ by applying Proposition \ref{prop: Weinstein_Movie_construction} to the (restricted) Weinstein homotopy $(F, \lambda_{F, s})$ between $0$ and $t$ (we use the diffeomorphism $\times -1: [0,1]\rightarrow [-1,0])$ to make the movie construction be on $F \times T^*[-1,0]$ instead of $F \times T^*[0,1]$). 
		Note that the movie $(F\times T^*[-1,0], \lambda_{F, movie, t})$ can be generalized Weinstein at particular $t$, with possibly birth-death singularities, if $(F, \lambda_{F, t})$ is generalized Weinstein at those $t$. Then we append $(F\times T^*[-1,0], \lambda_{F, movie, t})$ to $Y$ along $F$ to construct $[(Y', \lambda_{Y', t}), (F, \lambda_{F,t})]$. The co-cores are $C_{F_t} \times T^*_{-3/4} [-1,0]$ by the construction in Proposition \ref{prop: Weinstein_Movie_construction}. 
		
		The second part of the proposition is exactly the same except that now we use the flipped movie construction $\lambda_{1-t}$ and so the co-cores are
		given by the product of the co-cores of $F_0$ and the fiber $T^*_{-3/4}[-1,0]  \subset T^*[-1, 0]$.
		
		\begin{figure}
			\centering
			\includegraphics[scale=0.2]{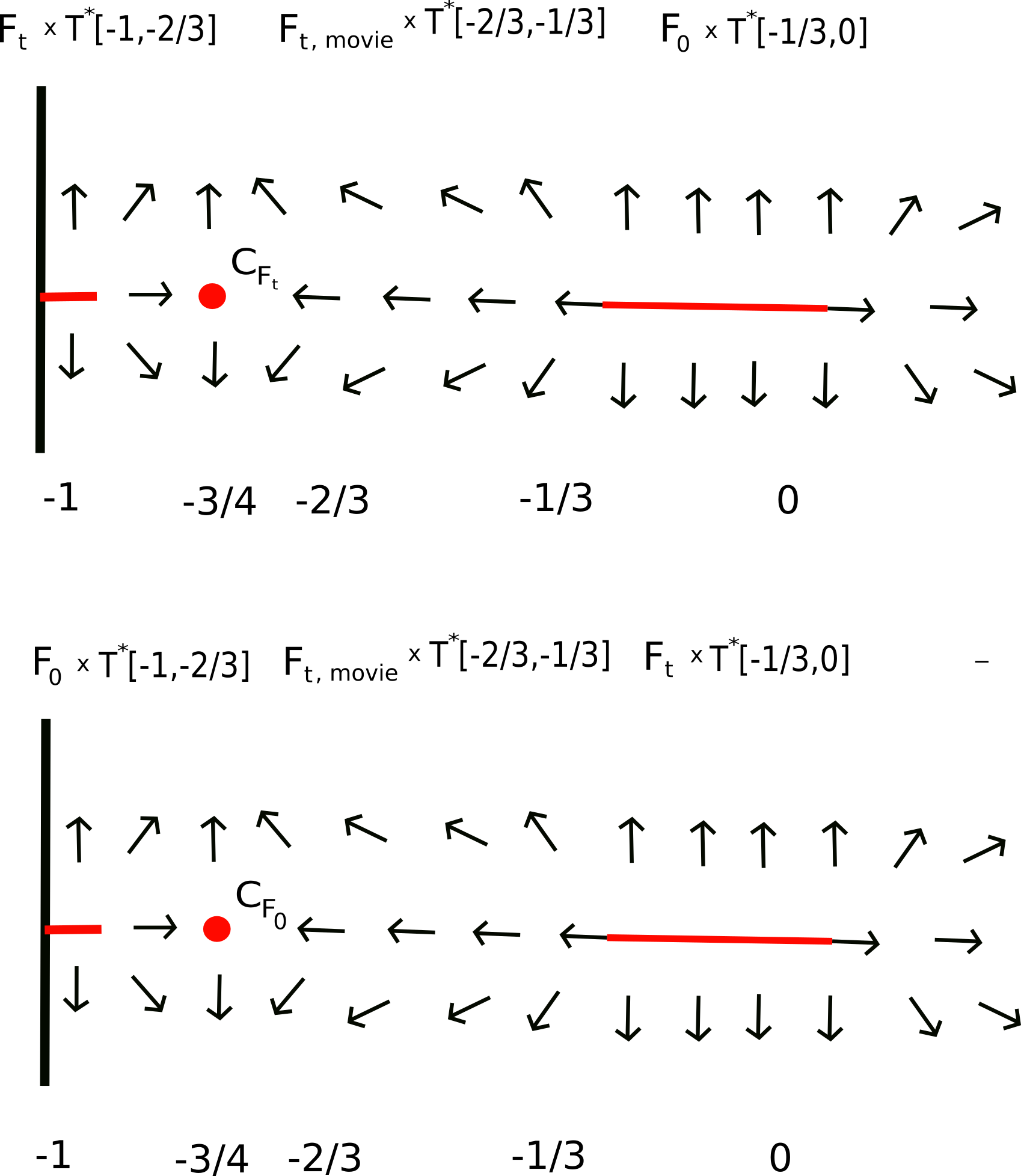}
			\caption{Bordered homotopy  at time $t$ from Part 1) (top figure)  and Part 2) (bottom figure) of Proposition \ref{prop: bordered_to_interior_homotopy_Weinstein}. For Part 1), the index $n$ co-cores are $C_{F,t} \times T^*_{-3/4} [-1,0]$, where $C_{F,t}$ are the co-cores of $(F, \lambda_{F,t})$. For Part 2), 		the index $n$ co-cores are $C_{F,0} \times T^*_{-3/4} [-1,0]$.
			}
			\label{fig: morsification_movie}
		\end{figure}

		Furthermore, if the Weinstein homotopy $(Y, \lambda_t)$ was associated to a stopped domain 
		$((Y_0, \lambda_t), F_t)$ then the original sectorial structure has the form 
		$T^*[0,1]$ for a vector field that is zero on $[0,1/4]$, has an index $1$ critical point at $1/2$, and is inward pointing near $1$. Once we attach $T^*[-1,0] \times F$ with the movie construction, this $T^*[0,1]$ is in the interior and hence we can cancel the subcritical zero locus on  $[0,1/4]$, with the index $1$ critical locus at $1/2$. The result will be a vector field on $T^*[-1,1]$ which has an index $1$ critical point at $-3/4$ and is outward pointing everywhere on $(-3/4, 1)$. So this is precisely a sector associated to a stopped domain $((Y_0, \lambda_t)', F_0)$; here $((Y_0, \lambda_t)', F_0)$ is a slight enlargement of the domain $(Y_0, \lambda_t)$ obtained by attaching $F \times T^*[-1/2, 0]$ with a Liouville vector field that has no zeroes in this region. Since the Lagrangians co-cores of $(Y_0, \lambda_t)$  are disjoint from the stop $F$, the co-cores of $(Y_0, \lambda_t)$  and $(Y_0, \lambda_t)'$ agree. See Figure \ref{fig: morsification_movie_stopped}. 
		\begin{figure}
			\centering
			\includegraphics[scale=0.21]{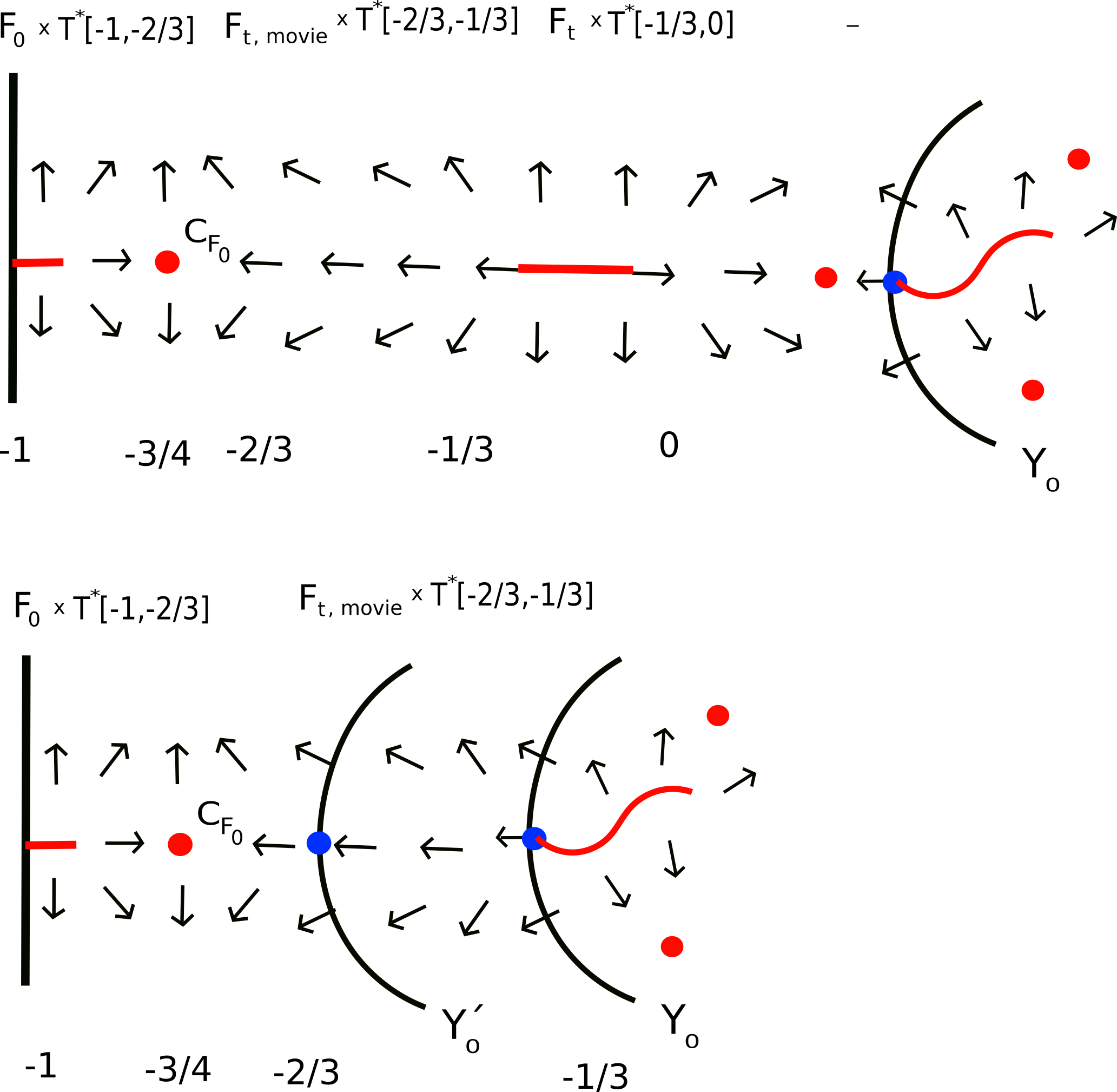}
			\caption{Bordered homotopy  at time $t$ from Part 3) of Proposition \ref{prop: bordered_to_interior_homotopy_Weinstein}, before cancellation of critical points and after cancellation. The stopped domain $Y_0$ is modified to the stopped domain $Y_0'$, which has the same interior critical points. }
			\label{fig: morsification_movie_stopped}
		\end{figure}
	\end{proof}

	\subsection{Non-strict Liouville embeddings}
	
	We will want to consider smooth proper embeddings that only respect Liouville forms after some deformation.
	
	\begin{definition}
		\label{defn. various deformation embeddings}
		Fix a (not necessarily strict) smooth, proper embedding $f: (X, \lambda_X) \rightarrow (Y, \lambda_Y)$. We say that the pair $(f, \lambda_{Y,t})$ is a \textit{bordered, interior, or compactly supported deformation embedding} if $\lambda_{Y,t}$ is a bordered, interior, or compactly supported Liouville deformation from $(Y, \lambda_Y)$ to $(Y, \lambda_Y')$ so that $f^*\lambda_Y' = \lambda_X$.
	\end{definition}
	
	We may also have occasion to require the deformation of $Y$ to be a Weinstein homotopy.

	\begin{remark}
		We note that interior and compactly supported embeddings can be composed and form the morphisms of a category; see~\cite{LAST_categories} for details. However these properties of these deformations are not invariant under products, nor under stabilization. For example, if $Y, Y'$ are  interior or compactly supported Liouville homotopic, then there is a canonical bordered homotopy of sectors between $Y \times T^*D^1, Y' \times T^*D^1$ but it is not interior or compactly supported.  
	\end{remark}
	
	The special case when $\phi$ is a diffeomorphism will appear often in this paper: 
	\begin{defn}\label{defn: isomorphism up to weinstein_homotopy}
		If $\phi: (X, \lambda_X)\rightarrow (Y, \lambda_Y)$ is a bordered, interior, or compactly supported deformation embedding which is also a diffeomorphism, then we call $\phi$ an \textit{isomorphism, up to bordered, interior, or compactly supported deformation} respectively.
	\end{defn}
	This is called a Weinstein homotopy equivalence in \cite{CE12}; we do not use this term in this paper to avoid confusion with equivalences in the categories discussed in Section \ref{sec: inf_category_stabilized_sectors}.

	\subsection{Notions of equivalence between sectors}
	\label{section. sectorial equivalences}

	There are various notions of equivalence one may define for Liouville sectors---trivial inclusions, sectorial equivalences, bordered deformation equivalences, and movie inclusions. We refer to Section 12 of~\cite{LAST_categories} for details. Here, we recall only one notion:
	
	\begin{defn}\label{defn. movie inclusions}
		Let $M \times T^*[0,1]$ be the movie construction for some bordered deformation of Liouville structures on $M$, where the deformation is constant near $t=0,1$ as usual. (In particular, there are well-defined Liouville forms $\lambda_0$ and $\lambda_1$ on $M$.) Then for any $\epsilon$ small enough, we call the inclusions
		\eqnn
		(M,\lambda_0) \times T^*[0,\epsilon] \into M \times T^*[0,1],
		\qquad
		(M,\lambda_1) \times T^*[1-\epsilon,1] \into M \times T^*[0,1]
		\eqnd
		{\em movie inclusions}. Note that both are strict proper embeddings.
	\end{defn}

	\begin{remark}
		\label{remark. equivalences dont change W}
		Movie inclusions induce equivalences of wrapped Fukaya categories; see Section~12 of~\cite{LAST_categories}.
	\end{remark}

	\subsection{\texorpdfstring{The $\infty$-category of stabilized Liouville sectors}{The infinity-category of stabilized Liouville sectors}}\label{sec: inf_category_stabilized_sectors}
	Detailed descriptions of the following appear in~\cite{LAST_categories}. 
	
	\begin{notation}[Notation 1.11 of~\cite{LAST_categories}]
		We let $\lioustr$ denote the category whose objects are Liouville sectors, and whose morphisms are strict proper embeddings (Definition~\ref{defn. strict proper inclusion}).
	\end{notation}

	The category $\lioustr$ admits a symmetric monoidal structure under direct product. (See Remark~\ref{remark. strict is monoidal}.) Accordingly, we have an endofunctor
	\eqnn
	-\times T^*[0,1] : \lioustr \to \lioustr.
	\eqnd

	\begin{notation}[Notation 1.11 and 10.1 of~\cite{LAST_categories}]
		\label{notation.lioustrdd}
		We let
		\eqnn
		\lioustr^{\dd} := \colim\left(\lioustr \xrightarrow{-\times T^*[0,1]}\lioustr \xrightarrow{-\times T^*[0,1]} \ldots\right)
		\eqnd
		denote the colimit, which one can model as an increasing union of categories. 
		
		We call $\lioustr^{\dd}$ the category of stabilized Liouville sectors.
	\end{notation}
	
	\begin{remark} Concretely, an object of $\lioustr^{\dd}$ is an equivalence class of a pair $(X,k)$ where $X$ is a Liouville sector and $k \geq 0$ is an integer. By construction, we identify $(X,k) \sim (X \times T^*[0,1]^n, k+n)$.
		
		Moreover, given two objects (represented by) $(X,k)$ and $(X',k')$, there is a morphism between them if and only if $\dim X - 2k = \dim X' - 2k'$. By definition, the set of morphisms is given by
		\eqnn
		\hom_{\lioustr^{\dd}}( (X,k),(X',k'))
		=
		\bigcup_{l \geq 0}
		\{i: X \times T^*[0,1]^{l} \to X' \times T^*[0,1]^{l+k-k'}\}
		\eqnd
		where each $i$ is required to be a strict proper embedding, and the union is taken by identifying any $i$ with $i \times \id_{T^*[0,1]}$. 
		
		Thus $\lioustr^{\dd}$ may be thought of as a disjoint union of categories indexed by the invariant $\dim X - 2k$.
	\end{remark}
	
	\begin{notation}
		\label{notation. X is X stabilized}
		Let $X$ be a Liouville sector. By abuse of notation, we will denote by $X$ the object of $\lioustr^{\dd}$ represented by the pair $(X, {\frac 1 2} \dim X)$.
	\end{notation}

	\begin{notation}
		We let $\eqs_{\movie} \subset \lioustr$ denote the collection of strict Liouville embeddings that happen to be movie inclusions (Definition~\ref{defn. movie inclusions}). We let $\eqs_{\movie}^{\dd}$ denote the collection of strict movie inclusions in $\lioustr^{\dd}$. 
	\end{notation}
	
	Now we enter the realm of $\infty$-categories. 
    See Section~\ref{section. conventions} for the notion of localization in this setting, and Remark~\ref{remark. why infinity cat} for why we take the $\infty$-categorical localization.
	\begin{notation}
		We let
		\eqnn
		\lioustab:= \lioustr^{\dd}[(\eqs_{\movie}^{\dd})^{-1}]
		\eqnd
		denote the ($\infty$-categorical) localization.
	\end{notation}
	
	\begin{remark}\label{remark. equivalent localizations}
		In fact, we show in Section~12 of~\cite{LAST_categories} that one may localize $\lioustr^{\dd}$ with respect to other natural classes of symplectic equivalences such as (i) strict trivial inclusions, (ii) strict sectorial equivalences, (iii) strict bordered deformation equivalences, and (iv) movie inclusions, and the resulting $\infty$-categories are all equivalent.
	\end{remark}

	\begin{remark}\label{rem: geometric_model}
		A priori, it is unclear what geometric information this $\infty$-categorically formal process creates. We show in~\cite{LAST_categories} that in fact, the $\infty$-category $\lioustab$ recovers, up to homotopy equivalence, the stabilized mapping {\em spaces} (not just sets) of compactly supported deformation embedding. We do not need this powerful result in the present work, but we illustrate some of the geometric utility in Proposition~\ref{prop: convert_homotopy_equivalence} below.
	\end{remark}
	
	\begin{example}[Cotangent bundles of disks and cubes]\label{example. disks and cubes}
		For every $k \geq 2$, there are two a priori different objects $T^*[0,1]^k$ (which is identified with $T^*D^0$ in $\lioustr^{\dd}$) and $T^*D^k$. As smooth manifolds, the former has corners, while the latter does not.
		
		However, there are smooth embeddings $[0,1]^k \to D^k$ and $D^k \to [0,1]^k$ whose compositions are smoothly isotopic to the identity. Results of~\cite{LAST_categories} show that, therefore, $T^*[0,1]^k$ and $T^*D^k$ are equivalent objects in $\lioustab$, even though they are not isomorphic objects in $\lioustr$.
		Notice also that the (strict) proper embeddings $T^*[0,1]^k \to T^*D^k$ are sectorial equivalences, but not trivial inclusions in the sense of~\cite[Section~2.4]{ganatra_covariantly_functorial}.
	\end{example}
	
	\begin{remark}\label{remark. lioustab symmetric monoidal}
		One of the main results of~\cite{LAST_categories} is that $\lioustab$ admits a symmetric monoidal structure whose action on objects is given by direct product of (stabilized) sectors. More precisely, given two objects $X$ and $X'$ -- {ie }pairs $(X,k)$ and $(X',k')$ with $\dim X - 2k = \dim X' - 2k' = 0$ (Notation~\ref{notation. X is X stabilized}) -- their monoidal product is given by $X \times X'$. Note in particular that the symmetric monoidal unit is the point $T^*D^0$ (and hence $T^*[0,1]^k$ for $k \geq 0$).
		
		This is not formal, and is a consequence of the fact that, after localizing with respect to sectorial equivalences, the orientation-preserving permutations of $T^*[0,1]^k$ are homotopic to the identity. The existence of such homotopies is one of the non-trivial ways in which the categorically formal process of localization detects geometrically meaningful phenomena.
	\end{remark}

	\begin{notation}[The subcategories of Weinsteins]
		Finally, we let
		\eqnn
		\weinstr \subset \lioustr,
		\qquad
		\weinstab \subset \lioustab
		\eqnd
		denote the full subcategory of those sectors that admit a Weinstein structure. 
	\end{notation}
	
	\begin{remark}
		Note that $\weinstr$ and $\weinstab$ are defined to be full subcategories -- in particular, objects are not equipped with a Weinstein structure, though they are abstractly known to admit one. We emphasize that the morphisms in these $\infty$-categories need not respect Weinstein structures in any way.
		
		The reader may well wonder why we do not stabilize $\weinstr$ and localize. This is immaterial: one can prove that 
        $\weinstr^{\dd}[(\eqs_{\movie}^\diamond \cap\weinstr^{\diamond})^{- 1}]$ is equivalent as an $\infty$-category to $\weinstab$. 
		The main reason for considering $\weinstab$ instead of 		$\weinstr^\diamond[(\eqs_{\movie}^\diamond \cap\weinstr^{\diamond})^{- 1}]$ 
		in this paper is because there is a concrete geometric model for $\weinstab$, as proven in \cite{LAST_categories}; see Remark \ref{rem: geometric_model}.
		
	\end{remark}
	
	\subsection{Converting sectorial equivalences into equivalences in the stable Liouville category}
	
	The following is a formal consequence of~\cite{LAST_categories}, but we give an explicit proof for the sake of being self-contained in our geometry. It demonstrates how geometric structures that cannot be categorically captured using only strict proper inclusions in $\lioustr$ become visible in $\lioustab$ by passing to movies. For example, \eqref{item. liouville homotopic sectors are equivalent} below is false if one replaces $\lioustab$ by $\lioustr$.

	\begin{proposition}\label{prop: convert_homotopy_equivalence}
		\begin{enumerate}
			\item\label{item. liouville homotopic sectors are equivalent}	If $(X, \lambda_{X, t})$ is a Liouville homotopy, then $(X, \lambda_{X,0})$ and $(X, \lambda_{X,1})$ are equivalent in $\lioustab$. 
			\item\label{item. inverting isos up to deformation} Furthermore, suppose there is a commutative diagram of smooth maps
			\begin{equation}\nonumber
			\begin{tikzcd} 
			X_0 \arrow{r}{f_0}	
			\arrow{d}{\phi_X} & 
			Y_0  \arrow{d}{\phi_Y}\\
			X_1  \arrow{r}{f_1}	
			& 
			Y_1 
			\end{tikzcd}
			\end{equation} 
			where $f_0, f_1$ are strict proper inclusions and 
			$\phi_X, \phi_Y$ are isomorphisms up to bordered deformation, and that the bordered deformation on $Y_1$ may be chosen to extend $f_1$ of the bordered deformation of $X_1$.  Then there is a diagram $\Delta^1 \times \Delta^1 \to \lioustab$: 
			\begin{equation}\label{eqn. induced equivalences intertwined}
			\begin{tikzcd} 
			X_0 \arrow{r}{f_0}	\arrow{d}{\overline{\phi_X}}
			& 
			Y_0 \arrow{d}{\overline{\phi_Y}} \\
			X_1  	\arrow{r}{f_1}	
			& 
			Y_1 
			\end{tikzcd}
			\end{equation}
			where $\overline{\phi_X}$ and $\overline{\phi_X}$ are equivalences ({ie }homotopy invertible morphisms) in $\lioustab$.
		\end{enumerate}
	\end{proposition}
    
	While $\lioustab$ is defined as a localization, one can -- after stabilizing $X_0,Y_0,X_1,Y_1$ -- interpret all arrows in \eqref{eqn. induced equivalences intertwined} as strict proper inclusions equipped with deformations of Liouville structures. See Remark~\ref{rem: geometric_model}.  After potentially more stabilizations, one can interpret the diagram as providing isotopies between these data.
    
	\begin{proof}
		\eqref{item. liouville homotopic sectors are equivalent}
		The movie construction  of $\lambda_{X,t}$ defines a Liouville sector $(X \times T^*[0,1], \lambda_{X, movie})$, which admits strict proper inclusions 
		$$
		i_{X,0}: (X, \lambda_{X, 0}) \times T^*[0,\epsilon] \rightarrow (X \times T^*[0,1], \lambda_{X, movie})
		$$
		and 
		$$
		i_{X,1}: (X, \lambda_{X,1}) \times T^*[1-\epsilon, 1] \rightarrow (X \times T^*[0,1], \lambda_{X, movie})
		$$
		These two strict proper inclusions are movie inclusions (Definition~\ref{defn. movie inclusions}), hence are equivalences in $\lioustab$. In $\lioustab$, $(X, \lambda_{X,0})$ and $(X, \lambda_{X,1})$ are identified with their stabilizations, so we see they are equivalent in $\lioustab$.
		
		\eqref{item. inverting isos up to deformation}
		Because $\phi_X: (X_0, \lambda_{X_0}) \rightarrow (X_1, \lambda_{X_1})$ is an isomorphism up to deformation,  there is a homotopy of forms $\lambda_{X_1, t}$ on $X_1$ from $\lambda_{X_1}: =\lambda_{X_1, 0}$  to $\lambda_{X_1, 1}$ and $\phi: (X_0, \lambda_{X_0}) \rightarrow (X_1, \lambda_{X_1,1})$ is a strict isomorphism.  As in the previous paragraph, we have strict inclusions
		$$
		(X_0, \lambda_{X_0}) \times T^*D^1 \rightarrow (X_1, \lambda_{X_1,1}) \times T^*D^1 \rightarrow  (X_1\times T^*D^1, \lambda_{X_1,movie}) \leftarrow  (X_1, \lambda_{X_1}) \times T^*D^1
		$$
		where the first map is an isomorphism in $\lioustr$ (it is in fact the diffeomorphsim $\phi_X$) and the last two morphisms are (strict) movie inclusions (given by $i_{X_1,0}$ and $i_{X_1,1}$). Because $\lioustab$ is a localization of $\lioustr^{\dd}$ along movie inclusions, this zig-zag defines an equivalence $\overline{\phi_X}: (X_0, \lambda_{X_0}) \rightarrow (X_1, \lambda_{X_1})$, well-defined up to contractible space of choices, in $\lioustab$. This also defines $\overline{\phi_Y}$. 
		
		To see that the square~\eqref{eqn. induced equivalences intertwined} can be made to commute,
		use the assumption that $\lambda_{Y_1, t}$ extends $\lambda_{X_1, t}$ to construct the strict proper inclusion $f_{movie}: (X_1 \times T^*[0,1], \lambda_{X_1, movie}) \rightarrow (Y_1 \times T^*[0,1], \lambda_{Y_1, movie})$ fitting into the following commutative diagram in $\lioustr$:
		
		\begin{tikzcd}[column sep=tiny]
		(X_0, \lambda_{X_0}) \times T^*D^1 \arrow{d}{f_0}	\arrow{r}
		& 
		(X_1, \lambda_{X_1, 1})\times T^*D^1	\arrow{d}{f_0}		\arrow{r}
		&
		(X_1 \times T^*D^1, \lambda_{X_1, movie})  \arrow{d}{f_{movie}}		
		& (X_1, \lambda_{X_1,0}) \times T^*D^1 \arrow{l} \arrow{d}{f_1}\\
		(Y_0, \lambda_{Y_0}) \times T^*D^1 \arrow{r} 
		& 
		(Y_1, \lambda_{Y_1, 1})\times T^*D^1 \arrow{r} 
		&
		(Y_1 \times T^*D^1, \lambda_{Y_1, movie}) 
		&
		(Y_1, \lambda_{Y_1, 0})\times T^*D^1			\arrow{l} 
		\end{tikzcd}
		Then we take~\eqref{eqn. induced equivalences intertwined} to be the induced diagram in $\lioustab$.
	\end{proof}

	\subsection{Subcritical morphisms}
	\begin{definition}\label{def:sub_morphism}
		A strict proper inclusion $i: X \hookrightarrow Y$ is said to {\em realize a subcritical handle removal} if, after attaching  some subcritical Weinstein handles to $Y\setminus i(X)$ to produce a new sector $Y'$, the induced inclusion $i': X\rightarrow Y'$ is a movie inclusion. Likewise, a strict proper inclusion $i: X \hookrightarrow Y$ is said to {\em realize a subcritical handle attachment} if, after removing  some subcritical Weinstein handles to $Y\setminus i(X)$ to produce a new sector $Y'$, the inclusion $i: X\rightarrow Y'$ is a movie inclusion. 
		
		More generally, we say that $i$ is {\em subcritical} if it can be written as a composition of strict proper inclusions realizing subcritical handle attachments or subcritical handle removals.
		
	\end{definition}

	In particular, a movie inclusion is a subcritical proper inclusion (obtained by removing or attaching no subcritical handles to $Y\setminus i(X)$). 
	
	The next proposition shows that subcritical morphisms are preserved under taking products. 
	
	\begin{proposition}\label{prop. subcritical is monoidal}
		Let $Y_i$ be Weinstein sectors and let $f_i: X_i \to Y_i$ be  subcritical morphisms for $i=1,2$. Then	the map $f_1 \times f_2: X_1 \times X_2 \to Y_1 \times Y_2$ is subcritical.
	\end{proposition}
	
	\begin{proof}
		Without loss of generality, we may assume $f_1$ realizes a subcritical handle attachment/removal, and likewise for $f_2$. So 
		there are sectors $Y_i'$ so that $f_i': X_i \rightarrow Y_i'$ is a movie inclusion and $Y_1'$ is obtained from $Y_1$ by removing/attaching some subcritical  cobordism $C_1$ in  $Y_1\setminus i(X_1)$; similarly, $Y_2'$ is obtained from $Y_2$ by removing/attaching some subcritical cobordism  $C_2$ in  $Y_2\setminus i(X_2)$. 
		Then $f_1' \times \id: X_1 \times X_2 \to Y_1' \times X_2$ and $\id \times f_2': X_1  
		\times X_2 \to X_1 \times Y_2'$  are movie inclusions. Moreover, $Y_1' \times X_2$ is obtained from $Y_1 \times X_2$ by removing/attaching some cobordism of the form $C_1 \times X_2$. Since $C_1$ is subcritical and $X_2$ is Weinstein, $C_1 \times X_2$ is subcritical for dimension reasons. This shows $f_1 \times \id$ is subcritical. Likewise, we see $\id \times f_2$ is subcritical. 
		Then $f_1 \times f_2 = (\id \times f_2) \circ (f_1 \times \id)$ is also subcritical. 			\end{proof}
	
	\begin{remark}
		\label{remark. subcritical liouvilles not monoidal}
		One may have wondered why our main results concern $\weincrit$ rather than $\lioucrit$. Proposition~\ref{prop. subcritical is monoidal}'s proof shows why. In general, the product of two subcritical morphisms of Liouville sectors need not be subcritical. 
        		
		For example, let $Z$ be McDuff's example of a Liouville manifold that is not Weinstein~\cite[Proof of Theorem 1.1]{mcduff-inventiones}. $Z$ is diffeomorphic to a product of $[0,1]$ with the unit sphere bundle of $T^*\Sigma$, where $\Sigma$ is a hyperbolic surface. In particular, $Z$ is a 4-manifold with $H_3 \cong \ZZ$, showing $Z$ is not Weinstein. (Alternatively, $\del Z$ is disconnected.) 
        Letting $\CC = Y_1$ and letting $X_1 = \emptyset$, the proper inclusion $X_1 \to Y_1$ is subcritical as it has associated Weinstein cobordism  $C = Y_1$ given by a 0-handle.
        Letting $X_2 = Y_2 = Z \times Z$, consider the identity map $f_2 = \id: X_2 \to Y_2$. The direct product of $f_1$ with $f_2$ is not a Weinstein cobordism, as the complement has homology above middle degree.

        One may of course define a natural class of morphisms among Liouville sectors that are generated by subcriticals under direct product, but we prove no notable properties of such a class of morphisms here. 	
	\end{remark}

	\subsection{\texorpdfstring{The critical $\infty$-category}{The critical infinity-category}}
	
	\begin{notation}
    \label{notation. weincritstab}
		We let $\subcrit \subset \weinstr$ denote the collection of strict proper embeddings which happen to be subcritical morphisms (Definition~\ref{def:sub_morphism}). We let $\subcrit^{\dd}$ denote the image of the collection of such morphisms in $\weinstab$. We let
		\eqnn
		\weincrit :=\weinstab[(\subcrit^{\dd})^{-1}]
		\eqnd
		denote the $\infty$-categorical localization (see Section~\ref{section. conventions}).
		We will refer to $\weincrit$ as the {\it critical $\infty$-category of stabilized Weinstein sectors} or as the {\it critical category} for short.
	\end{notation}

	\begin{remark}
		Note that it also makes sense to ask whether a morphism between Liouville sectors (not Weinstein sectors) is subcritical. Thus, one could also profitably consider
		\eqnn
		\lioucrit := \lioustab[(\subcrit^{\dd})^{-1}]
		\eqnd
		but we do not do so in this paper. See Remark~\ref{remark. subcritical liouvilles not monoidal}.
	\end{remark}

	\subsection{Subdomain embeddings and sectorial cobordisms}\label{sec: subdomain_def}
	
	It will also be convenient to consider a class of symplectic embeddings between Liouville or Weinstein sectors that are not necessarily proper inclusions.  
	
	\begin{definition}
		Let $X$ and $Y$ be Liouville sectors.
		We say that a (not necessarily proper) smooth, codimension zero  embedding $i: X \rightarrow Y$ with $i^*\lambda_Y = \lambda_X$ is a strict \textit{subdomain} embedding, or strict subdomain inclusion, if $i(\partial X) \subset \partial Y$.
	\end{definition}
	
	\begin{notation}
		We will use the notation $i: X \xrightarrow{\subset} Y$ for subdomain inclusions and the notation $i: X \hookrightarrow Y$ for strict proper inclusions.
	\end{notation}
	
	\begin{remark}
		As we will see later (Proposition~\ref{prop: convert_subdomain_to_proper_inclusion}), subdomain inclusions can be converted (contravariantly and up to homotopy) into proper inclusions in the 		 critical Weinstein category.  
	\end{remark}

	\begin{defn}
		If $i: X \xrightarrow{\subset} Y $ is a strict subdomain inclusion, then we say that $Y\setminus i(\skel X)$ is a Liouville \textit{sectorial cobordism}. We will frequently abuse notation and write $Y \setminus i(X)$ to mean $Y \setminus i(\skel X)$.
	\end{defn}

	We can also define an abstract Liouville sectorial cobordism to be an exact 
	symplectic manifold satisfying all conditions of a Liouville sector except that there are  two `boundaries at infinity', namely $\partial_{\pm \infty} C$, and the Liouville vector field is outward, inward pointing at $\partial_{+\infty} C, \partial_{-\infty} C$ respectively. 
	
	\begin{defn}We say that $i: X \xrightarrow{\subset} Y$ is a (subcritical) Weinstein subdomain inclusion if the sectorial cobordism $Y\setminus i(X)$ admits a (subcritical) Weinstein structure. 
	\end{defn}
	
	\begin{remark}
		\label{remark. subcritical inclusions and subcritical subdomains}
		Note that a subcritical proper inclusion $i: X \rightarrow Y$ as in Definition 
		\ref{def:sub_morphism} implies that there is a subcritical subdomain inclusion $Y \xrightarrow{\subset} Y'$ or $Y'\xrightarrow{\subset} Y$.	
	\end{remark}
	
	\subsection{Converting subdomain inclusions into morphisms in the critical category}
	
	In this section, we explain how to convert sectorial subdomain inclusions into morphisms in $\weincrit$, {ie }strict inclusions up to stabilization and subcritical inclusions. 
	
	\begin{proposition}\label{prop: convert_subdomain_to_proper_inclusion}
		Suppose there is a commutative diagram of symplectic embeddings
		\begin{equation}\label{comm: subdomain_inclusions}
		\begin{tikzcd} 
		X_0 \arrow{r}{f_0}	
		\arrow{d}{\phi_X} & 
		Y_0  \arrow{d}{\phi_Y}\\
		X_1  \arrow{r}{f_1}	
		& 
		Y_1 
		\end{tikzcd}
		\end{equation} 
		where $f_0, f_1$ are strict proper inclusions and $\phi_X, \phi_Y$ are strict Weinstein subdomain inclusions; furthermore, assume that this is a pullback diagram of sets, {ie }$X_1 \cap Y_0 = X_0$. Then 	there is a homotopy commutative diagram in $\weincrit$: 
		\begin{equation}\label{comm: subdomain_inclusions_converted}
		\begin{tikzcd} 
		X_0 \arrow{r}{f_0}	
		& 
		Y_0  \\
		X_1  	\arrow{u}{\phi_X^*}\arrow{r}{f_1}	
		& 
		Y_1 \arrow{u}{\phi_Y^*}
		\end{tikzcd}
		\end{equation} 
		Furthermore, if $\phi_X, \phi_Y$ respectively are subcritical subdomain inclusions, then $\phi_X^*, \phi_Y^*$ are isomorphisms in $\weincrit$ respectively. 		
	\end{proposition}
	\begin{proof}
		We first discuss the case of a single subdomain inclusion $\phi_X: X_0 \xrightarrow{\subset} X_1$. We construct an intermediate sector, which we call the \textit{Viterbo} sector, that admits strict proper inclusions from the stabilization of both $X_0$ and $X_1$; this sector was introduced by the second author~\cite{sylvan_talk}. 			
		The stabilization of $X_1$ is $X_1 \times T^*D^1 = [X_1 \times T^*D^1, X_1\times 0 \coprod X_1\times 1]$. Then we define the Viterbo sector $V(X_0, X_1)$ by removing $(X_1\setminus X_0) \times\{1\}$ from the 
		sectorial divisor $X_1\times \{0,1\}$ of $X_1 \times T^*D^1$ by making the Liouville vector field point outward along $(X_1 \setminus X_0) \times \{1\}$. 		 
		Since $V(X_0, X_1)$ is obtained by stop removal from $X_1\times T^*D^1$, there is a proper inclusion 
		$$
		i_{X_1, V}: X_1 \times T^*D^1 \hookrightarrow V(X_0, X_1).
		$$
		The sectorial divisor of $V(X_0, X_1)$ looks like $X_0\times \{0,1\}$ as a set and using the condition that $\phi_X(\partial X_0) \subset\partial X_1$ (since $\phi_X: X_0 \overset{\subset}{\rightarrow} X_1$ is a subdomain inclusion), we get a  proper inclusion
		$$
		i_{X_0, V}: X_0 \times T^*D^1\hookrightarrow V(X_0, X_1)
		$$
		We note that if $S = \phi_X(\partial X_0) \setminus \partial X_1$ is non-empty, then we would only get a proper inclusion  $(X_0\setminus S)\times T^*D^1 \hookrightarrow  V(X_0, X_1)$, where $(X_0 \setminus S)$ is the result of stop removal of $S$ from $X_0$. 
		After homotopy, the proper inclusion $i_{X_0, V}$ is always a subcritical morphism. Namely, the sector $V(X_0, X_1)$ is obtained from $X_0 \times T^*D^1$ by attaching handles, one for each handle of $X_1 \setminus X_0$, and these handles have the same index in $V(X_0, X_1)$ as in $X_1 \setminus X_0$ (which has lower dimension than  $V(X_0, X_1)$) since the Liouville vector field is outward pointing along $(X_1 \setminus X_0) \times \{0,1\}$; in particular, these handles are subcritical. 
		Hence we have a  zig-zag of strict proper inclusions  
		$$
		 X_1 \times T^*D^1 \overset{i_{X_1, V} }{\hookrightarrow} V(X_0, X_1) \overset{i_{X_0, V}}{\hookleftarrow}  X_0 \times T^*D^1 
		$$
		with the second map $i_{X_0, V}$ is a subcritical morphism, which defines the morphism $\phi_X^{*}: X_1 \rightarrow X_0$ in the critical category $\weincrit$.

		Next, suppose have a commutative diagram of symplectic embeddings as in the second part of the proposition. Then we have a commutative diagram in $\weinstr$ of the form
		\begin{equation}\label{comm: Viterbo_sector_comm_diagram}
		\begin{tikzcd} 
		X_1 \times T^*D^1 \arrow{rr}{f_1 \times Id_{T^*D^1}}	\arrow{d}{i_{X_1, V}}
		&& 
		Y_1 \times T^*D^1  \arrow{d}{i_{Y_1, V}} \\
		V(X_0, X_1)	\arrow{rr}{\overline{f_1\times Id_{T^*D^1}}}	
		&& 
		V(Y_0, Y_1)\\
		X_0 \times T^*D^1\arrow{rr}{f_0\times Id_{T^*D^1}}	\arrow{u}{i_{X_0, V} }
		&&
		Y_0 \times T^*D^1\arrow{u}{i_{Y_0, V} }
		\end{tikzcd}
		\end{equation} 
		The middle horizontal map from $V(X_0, X_1)$ to $V(Y_0, Y_1)$ is constructed as follows. By the commutative diagram \eqref{comm: subdomain_inclusions} and the assumption that $X_1\cap Y_0 = X_0$, we have 
		$f_1\ (X_1\setminus X_0)$ is a subset of  $(Y_1\setminus Y_0)$. Since  $V(X_0, X_1), V(Y_0, Y_1)$ are obtained from $X_1 \times T^*D^1, Y_1\times T^*D^1$ by making the Liouville vector fields outward pointing on  $(X_1\setminus X_0) \times \{0,1\}, (Y_1\setminus Y_0) \times \{0,1\}$ respectively, we get an induced map $\overline{f_1\times Id_{T^*D^1}}$ from $V(X_0, X_1)$ to $V(Y_0, Y_1)$ making the diagram commute. 	The bottom vertical maps in equation \eqref{comm: Viterbo_sector_comm_diagram} are subcritical morphisms and hence this induces the desired commutative diagram in $\weincrit$.

		Finally, we note that if $X_0 \subset X_1$ is a \textit{subcritical} subdomain inclusion, then the first map $i_{X_1, V}$ is also an equivalence in $\weincrit$.  In general, we can construct a sector $V(X_0, X_1)'$ by attaching handles to $V(X_0, X_1)$ in the complement of the image of $i_{X_1, V}$ so that $i_{X_1, V}: X_1 \times T^*D^1 \rightarrow V(X_0, X_1)'$ is a strict movie inclusion, after an interior homotopy of  $V(X_0, X_1)'$.
		More precisely, for each handle of $X_1 \setminus X_0$ of index $i$, we attach a Weinstein cobordism with a pair of cancelling handles of index $i, i+1$ (with the index $i$ handle in the sectorial boundary of $ V(X_0, X_1)'$). Therefore, if all handles of $X_1\setminus X_0$ are subcritical, {ie }$i < \dim X$,   
        then $V(X_0, X_1)'$ is obtained by attaching subcritical handles to $V(X_0, X_1)$, since $i+1 < \dim X +1 = \dim(V(X_0, X_1))$. 
        See Figure \ref{fig: subcritical_viterbo}.
	\end{proof}
	\begin{figure}
		\centering
		\includegraphics[scale=0.2]{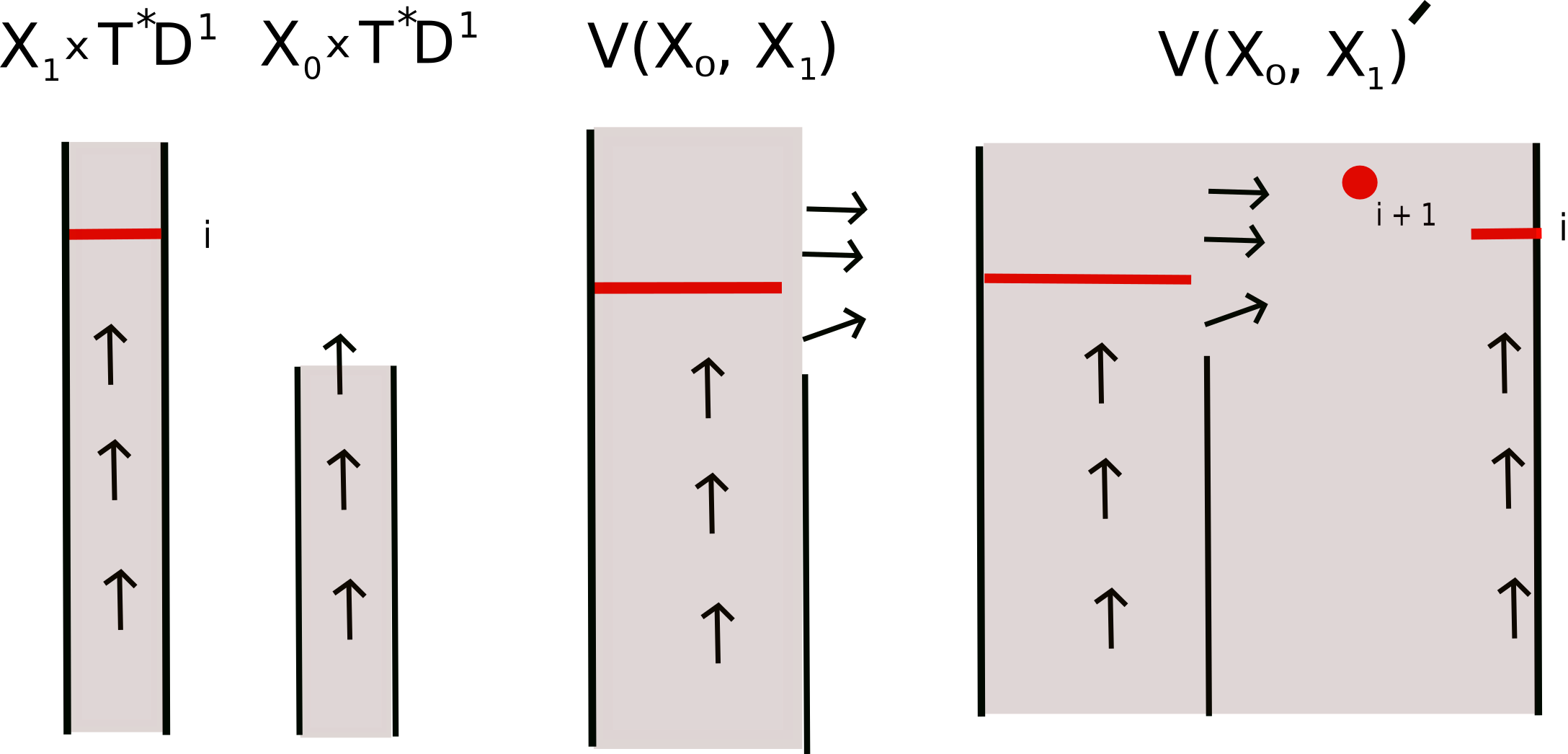}
		\caption{Sectors $X_1\times T^*D^1$, $V(X_0, X_1)$, and $V(X_0, X_1)'$, with critical points in red.		
		}
		\label{fig: subcritical_viterbo}
	\end{figure}

	\subsection{Lagrangians}
	In this paper, we consider properly embedded Lagrangians $L^n$ in $(X^{2n}, \lambda)$ that are exact, {ie }$\lambda|_L$ is an exact 1-form.  From the perspective of stopped domains, the properness condition means exactly that the Lagrangian boundary of $L$ avoids the stop. 
	We will also equip $L$ with the data of a bordered Liouville homotopy $\lambda_{L, t}$ (Definition~\ref{defn. deformations of liouville structures} and Remark~\ref{remark. bordered homotopy}) from $\lambda$ to $\lambda_L$ for which $L$ is \textit{strictly} exact, {ie }$\lambda_L|_L = 0$. Then by Proposition \ref{prop: bordered_to_interior_homotopy}, we can also assume that this bordered homotopy is an interior homotopy. We will keep track of $(L, \lambda_{L, t})$ as a tuple.
	If $L$ is already strictly exact for $\lambda$, then we say that a homotopy $\lambda_t$ of $\lambda$ is \textit{relative to} $L$ if $L$ is strictly exact for all $\lambda_t$. 
	
	\begin{example}
		For example, if $L$ has connected Legendrian boundary for $\lambda$, then there is a compactly supported function $h: X \rightarrow \mathbb{R}$ so that $\lambda+ dh|_L = 0$ and $\lambda+ dh$ is (canonically) compactly supported homotopic to $\lambda$. 
	\end{example}
	
	\begin{example}
		The issue is that the condition of having Legendrian boundary is not invariant under taking products.  If $L \subset X, K \subset Y$ are exact Lagrangians, then $L \times K \subset X \times Y$ is another exact Lagrangian in the product. If $\lambda_X, \lambda_Y$ do not vanish everywhere on $L, K$, then $L \times K$ does not have Legendrian boundary with respect to the product Liouville form $\lambda_X + \lambda_Y$.
		However, $L \times K$ is strictly exact for $\lambda_{X, L} + \lambda_{Y, K}$, where $\lambda_{X, L}, \lambda_{Y, K}$ are the forms on $X, Y$ for which $L, K$ are strictly exact. Then $\lambda_{X, L} + \lambda_{Y, K}$  is bordered Liouville homotopic to $\lambda_X + \lambda_Y$.  In particular, there is an interior  Liouville homotopy of $\lambda_X + \lambda_Y$ to a new form $\lambda_{X \times Y, L\times K}$ for which $L \times K$ is strictly exact.
		Note that there cannot be a \textit{compactly supported} deformation of $\lambda_X + \lambda_Y$ to a form for which $L \times K$ is strictly exact since this would imply that $L \times K$ has Legendrian boundary for $\lambda_X + \lambda_Y$. 
	\end{example}
	
	\begin{defn}\label{defn. regular disk}
		If $(X, \lambda)$ is Weinstein, a Lagrangian $L \subset X$ is called {\em regular} if the Liouville deformation $\lambda_{L, t}$ from $\lambda$ to $\lambda_L$ is a bordered Weinstein homotopy.  
	\end{defn}
	
	Lemma 2.2 of~\cite{EGL} proves that one can apply a further Weinstein homotopy (with only birth-death singularities) supported near $L$ so that a neighborhood $N$ of $L$ for which  $(N, \lambda_L)$ can be identified with $(T^*L, \lambda_{T^*L, std})$ equipped with its canonical Weinstein structure (induced by any proper Morse function on $L$ or 
	the zero function). By further applying Proposition \ref{prop: bordered_to_interior_homotopy}, we can also assume that there is an \textit{interior} Weinstein homotopy from $\lambda$ to such $\lambda_L$. 
	
	\section{\texorpdfstring{$\infty$-categorical background}{Infinity-categorical background}}\label{sec: category}
	
	Let $A$ be a Weinstein sector of real dimension $2n$, and suppose that there is a sectorial embedding $u: T^*\RR^n \to A$ so that
	\eqnn
	\id_A \times u: A \times T^*\RR^n \to A\times A
	\eqnd
	is, after localizing $\weinstab$ along some collection of morphisms $S$, an equivalence. (For example, $\id_A \times u$ could itself be in $S$.) The goal of this section is to prove that -- so long as $S$ is closed under direct product of sectors -- the existence of $u$ implies that
	\enum
	\item the functor $\weinstab[S^{-1}] \to \weinstab[S^{-1}]$ given by $-\times A$ is a localization onto its image, and
	\item $A$ inherits a natural $E_\infty$-algebra structure  (Remark~\ref{remark. Eoo algebra}) in $\weinstab[S^{-1}]$ with unit $u$.
	\enumd
	See Corollary~\ref{corollary. idempotent objects are algebras}.
	We will apply this result when $A = (T^*D)[P^{-1}]$ and $S$ is the class of subcritical morphisms to obtain Theorem~\ref{thm: intro_idempotent_functor}.
	
	\begin{remark}
		There are geometric consequences of the $\infty$-categorical results here. For example, by applying Theorem~\ref{thm: intro_idempotent_functor}, we can conclude that the diagram of sectorial embeddings
		\eqnn
		\begin{tikzcd} 
		(T^*D)[P^{-1}] \times (T^*D)[P^{-1}] \arrow{d}{\text{swap}}
		&& (T^*D)[P^{-1}] \times T^*D^n  \arrow[ll] \arrow[dll]\\
		(T^*D)[P^{-1}] \times (T^*D)[P^{-1}]
		\end{tikzcd} 
		\eqnd
		commutes up to homotopy in $\weincrit$. (Further, all maps are equivalences in $\weincrit$.)
		See also Remark~\ref{remark. why infinity cat}.
	\end{remark}
	
	The statements in the rest of this section are purely $\infty$-categorical. All of them are already contained in, or are immediate consequences of, the extensive machinery constructed in~\cite{higher-algebra,htt}.
	We have included this section to save the reader the trouble of navigating these tomes, and to give a streamlined presentation for the applications we have in mind. Further references include~\cite{htt, higher-algebra, joyal-notes, nadler-tanaka,LAST_categories}.    
	\subsection{Idempotent functors and localizations}
    \label{section. idempotent localization}
	We review the passage between idempotent endofunctors and localizations. More specifically, let us declare an {\em idempotent structure} on an endofunctor $L: \cC \to \cC$ to be a natural transformation $\eta: \id_\cC \to L$ for which the induced maps $\eta_{L(X)}, L(\eta(X)): L(X) \to L \circ L(X)$ are both equivalences.
	
	\begin{proposition}\label{prop. idempotents are localizations}
		Let $L: \cC \to \cC$ be a functor. The following are equivalent:
		\enum
		\item $L$ admits an idempotent structure.
		\item\label{item. localization less general condition} Think of $L$ as a functor from $\cC$ to the full subcategory $L\cC \subset \cC$ spanned by the image of $L$. Then $L$ is a left adjoint to the (fully faithful) inclusion $L\cC \to \cC$.
		\enumd
	\end{proposition}
	
	\begin{proof}
		This follows from Proposition~5.2.7.4 of~\cite{htt}. 
	\end{proof}
	
	\begin{remark}
		Recall that the localization along some class of morphisms $S$ is the initial $\infty$-category inverting all morphisms in $S$ (see Section~\ref{section. conventions}).  Proposition~\ref{prop. localizations are localizations} below (which is an immediate consequence of Proposition~5.2.7.12 of~\cite{htt}) shows that any functor satisfying~\eqref{item. localization less general condition} of Proposition~\ref{prop. idempotents are localizations} is a localization.
	\end{remark}
	
	\begin{example}[Smashing localizations]\label{example. smashing localization}
		Let $\cC^{\tensor}$ be a symmetric monoidal $\infty$-category and fix $A \in \ob \cC$. Let $1_\cC$ be the monoidal unit, and fix a map $u: 1_\cC \to A$. This defines a natural transformation $\eta: \id_{\cC} \to A \tensor - $, and for any object $X \in \cC$, we have induced maps
		\eqnn
		\eta_{A \tensor X}: A \tensor X \simeq  1_\cC \tensor (A \tensor X) \xrightarrow{u \tensor \id_{A \tensor X}} A \tensor A \tensor X,
		\eqnd
		\eqnn
		A \tensor \eta_X: A \tensor X \simeq A \tensor 1_\cC \tensor X \xrightarrow{\id_A \tensor u \tensor \id_X } A \tensor A \tensor X.
		\eqnd
		Of course, by permuting the two $A$ factors (using the symmetric monoidal structure) either both of these maps are equivalences, or neither is. So a sufficient condition for $u$ to induce an idempotent structure on the endofunctor $X \mapsto A \tensor X$ is for 
		\eqnn
		A \simeq 1 \tensor A \xrightarrow{u \tensor \id_A} A \tensor A
		\eqnd
		to be an equivalence.
	\end{example}
	
	\begin{example}\label{example. idempotents}
		Here are three examples; they are all the ``same'' example in spirit.
		
		\enum
		\item Let $\cC^{\tensor}$ be the category (in the classical sense) of ($\ZZ$-linear) rings, with the usual symmetric monoidal structure of tensor product (over $\ZZ$). Letting $A = \ZZ[1/p]$, we see that the unit map $u: \ZZ \to \ZZ[1/p]$ -- otherwise known as the inclusion of the integers -- satisfies the property that $u \tensor \id_A$ is an isomorphism from $A$ to $A \tensor A$. 
		\item More generally, let $\cC^{\tensor}$ be the opposite category (in the classical sense) of schemes over some base scheme $S$, equipped with the symmetric monoidal structure of fiber product over $S$. Let $A$ be any open subscheme of $S$, and let $u$ be (the opposite of) the inclusion of $A$ into $S$ -- {eg }a degree one open immersion. Then $u \times_S \id_A$ is an isomorphism from $A$ to $A \times_S A$.
		\item Likewise, let $\cC^{\tensor}$ be the category (in the classical sense) of topological spaces equipped with a continuous map to a fixed space $X$, equipped with symmetric monoidal structure given by fiber product over $X$. Then any open subset $U \subset X$, equipped with the inclusion $u: U \to X$, has the property that $u \times_X \id_U: U \to U \times_X U$ is a homeomorphism.
		\enumd
	\end{example}
	
	\begin{proposition}\label{prop. localizations are localizations}
		Let $L: \cC \to \cC$ be a functor satisfying either of the conditions of Proposition~\ref{prop. idempotents are localizations}. 
		Let $S \subset \cC$ be the class of morphisms that are sent to equivalences under $L$. Then the natural map $\cC[S^{-1}] \to L\cC$ from the localization along $S$ is an equivalence.
	\end{proposition}
	
	\begin{proof}
		This follows from Proposition~5.2.7.12 of~\cite{htt}, which shows that $L\cC$ satisfies the mapping property that characterizes $\cC[S^{-1}]$.
	\end{proof}
	
	\begin{example}
		Following the enumeration of Example~\ref{example. idempotents}:
		\enum
		\item $L\cC$ is the full subcategory consisting of those rings in which $p$ is multiplicatively invertible.
		\item $L\cC$ is the full subcategory consisting of those objects $A \to S$ for which the map from $A$ factors through $U$.
		\item $L\cC$ is the full subcategory consisting of those objects $A \to X$ for which the map from $A$ factors through $U$.
		\enumd
	\end{example}
	
	\subsection{Symmetric monoidal localization}
	Let $\cC^\tensor$ be a symmetric monoidal $\infty$-category with underlying $\infty$-category $\cC$. Fix also a collection of morphisms $S \subset \cC$.
	
	\begin{definition}
		We say that $S$ is {\em compatible} with  $\tensor$ if for all $f, g \in S \times S$ we have $f \tensor g \in S$. 
	\end{definition}
	
	\begin{proposition}\label{prop. localization inherits symmetric monoidal structure}
		Suppose $S$ is compatible with $\tensor$.
		Then there is an induced symmetric monoidal structure on $\cC[S^{-1}]$ so that the functor $\cC \to \cC[S^{-1}]$ may be promoted to a symmetric monoidal functor.
		
		In fact, more is true: for any symmetric monoidal $\infty$-category $\cD^{\tensor}$, the $\infty$-category of symmetric monoidal functors $\fun^\tensor(\cC[S^{-1}]^{\tensor},\cD^{\tensor})$ is identified with the full subcategory of $\fun^\tensor(\cC^{\tensor},\cD^{\tensor})$ sending morphisms in $S$ to equivalences in $\cD$. This identification is given by composing with the symmetric monoidal functor $\cC^{\tensor} \to \cC[S^{-1}]^{\tensor}$ 
	\end{proposition}
	
	\begin{proof}
		This is a consequence of Proposition 4.1.3.4 of~\cite{higher-algebra}.
	\end{proof}
	
	In~\cite{LAST_categories} we show that the $\infty$-category $\lioustab$ of stabilized Liouville sectors admits a symmetric monoidal structure, which on objects acts by the direct product of stabilized sectors (Remark~\ref{remark. lioustab symmetric monoidal}). This restricts to a symmetric monoidal structure on $\weinstab$. Because $\subcrit$ is compatible with direct product (Proposition~\ref{prop. subcritical is monoidal}), we see that $\subcrit^{\dd}$ is also compatible with the symmetric monoidal structure of $\weinstab$. Thus, we have
	
	\begin{corollary}\label{cor.lioucrit is symmetric}
		$\weincrit$ admits a symmetric monoidal structure, which on objects is given by direct product of (stabilized) Liouville sectors.
	\end{corollary}
	
	\subsection{Idempotent algebras}

	\begin{remark}[$E_\infty$-algebra background]\label{remark. Eoo algebra}
		Let us make some recollections on $E_\infty$ structures. From both homotopy theory and symplectic geometry, one may be used to the idea that $A_\infty$-algebras are a generalization of associative algebras wherein one demands the data of an infinite collection of coherences -- articulating associativity not through equalities, but through homotopies. Given a symmetric monoidal $\infty$-category $\cC$, one may speak of the commutative version of an $A_\infty$-algebra in $\cC$: an $E_\infty$-algebra in $\cC$. In fact, one often defines the term ``commutative algebra'' to mean $E_\infty$-algebra in this context -- see Definition~2.1.3.1 of~\cite{higher-algebra}. Unwinding the definitions there, an $E_\infty$-algebra in $\cC$ is a map of $\infty$-operads from the $E_\infty$-operad to the $\infty$-operad $\cC^\tensor$ defining the symmetric monoidal structure on $\cC$. Informally, an $E_\infty$-algebra structure on an object $A \in \cC$  consists of an infinite collection of higher data articulating both the associativity and commutativity of a multiplicative structure on $A$.
		
	Let us note that if $A$ is an $E_\infty$-algebra in $\cC$, then the endomorphisms of $A$ also form an $E_\infty$-algebra. For example, if the Fukaya category of $M$ admits a symmetric monoidal structure and a Lagrangian $L \subset M$ is an $E_\infty$-algebra in this Fukaya category, then the Floer cochain complex of $L$ admits an $E_\infty$-algebra structure in chain complexes.
	\end{remark}
	
	\begin{proposition}\label{prop. idempotent objects give symmetric monoidal localizations}
		Let $\cC^{\tensor}$ be a symmetric monoidal $\infty$-category and fix $A \in \ob \cC$. Endow the endofunctor  $L: \cC \to \cC$ given by $X \mapsto A \tensor X$ with an idempotent structure. (For example, via Example~\ref{example. smashing localization}.) 
		
		Then the essential image of $L$ -- {ie }the full subcategory $L\cC \subset \cC$ -- can be given a symmetric monoidal structure for which the functor $X \mapsto A \tensor X$ is symmetric monoidal, and for which $A$ is the symmetric monoidal unit.
	\end{proposition}
	
	\begin{proof}
		
		Let $S$ be the class of morphisms $f: X \to Y$ in $\cC$ such that the induced map $\id_A \tensor f: A \tensor X \to A \tensor Y$ is an equivalence. 
		It follows that $S$ is compatible with $\tensor$ by using the idempotent structure. So $\cC[S^{-1}]$ is symmetric monoidal and the map $\cC \to \cC[S^{-1}]$ may be promoted to be symmetric monoidal by Proposition~\ref{prop. localization inherits symmetric monoidal structure}. Further, the natural map from $\cC[S^{-1}]$ to the essential image of $L$ is an equivalence by Proposition~\ref{prop. localizations are localizations}. Transferring the symmetric monoidal structure accordingly, we obtain a symmetric monoidal functor
		\eqnn
		\cC \to L\cC,
		\qquad
		X \mapsto A \tensor X.
		\eqnd
		Moreover, because $\cC \to L\cC$ is the underlying functor of a symmetric monoidal functor, the image of the unit of $\cC$ is the unit of $L\cC$; this means $1_\cC \tensor A \simeq A$ is the symmetric monoidal unit. 
	\end{proof}
	
	\begin{remark}
		Let us explicate the symmetric monoidal structure on $L\cC$ guaranteed by Proposition~\ref{prop. idempotent objects give symmetric monoidal localizations}. Choose two objects of $L\cC$, which we may as well write as $A \tensor X$ and $A \tensor Y$. Because $\cC \to L\cC$ admits a symmetric monoidal enhancement, we conclude that the monoidal product on $L\cC$ may be computed as
		\eqnn
		(A \tensor X)  \tensor_{L\cC} (A \tensor Y)
		\simeq
		LX \tensor_{L\cC} LY \simeq
		L(X \tensor Y)
		\simeq
		A \tensor (X \tensor Y),
		\eqnd
		where the middle equivalence uses the fact that $L$ is (symmetric) monoidal.
		Thus, it would not betray one's intuition to notate the monoidal product of $L\cC$ as $\tensor_A$, rather than $\tensor_{L\cC}$.
	\end{remark}
	
	\begin{corollary}\label{corollary. idempotent objects are algebras}
		Let $\cC^\tensor$ be a symmetric monoidal $\infty$-category with unit $1_\cC$, and fix a map $u: 1_\cC \to A$. Suppose that the induced map $u \tensor \id_A : 1_\cC \tensor A \to A \tensor A$ is an equivalence. Then 
		\enum
		\item The functor $ \cC \to L \cC$ may be promoted to be symmetric monoidal.
		\item\label{item. L is lax} The functor $L : \cC \to \cC$ may be promoted to be lax symmetric monoidal. 
		\item $A$ inherits the structure of a commutative algebra in $\cC^{\tensor}$ with unit map $u$.
		\enumd
	\end{corollary}
	
	\begin{proof}
		By Example~\ref{example. smashing localization}, the endofunctor $L: X \mapsto X \tensor A$ is idempotent. By Proposition~\ref{prop. idempotent objects give symmetric monoidal localizations}, the induced localization $\cC \to L \cC$ may be promoted to be symmetric monoidal.  This proves the first claim.
		
		Moreover, any right adjoint to a symmetric monoidal functor is automatically lax symmetric monoidal. (This follows from Corollary~7.3.2.7 of~\cite{higher-algebra}. In the notation of loc. cit., one sets $\cO^\tensor$ to be the nerve of the category of finite pointed sets, and notices that if $F$ is a symmetric monoidal functor, then the fact that the underlying functor $\cC \to L\cC$ is a left adjoint implies that the product functors $\cC^n \to (L\cC)^n$ are also left adjoints. 
		Finally, note that a map of $\infty$-operads between symmetric monoidal $\infty$-categories is precisely a lax symmetric monoidal functor -- see the comments before Definition~2.1.3.7 of~\cite{higher-algebra}.) Noting that $L$ factors as a composition
		\eqnn
		\cC \to L\cC \into \cC
		\eqnd
		where the first arrow may be promoted to be symmetric monoidal, and the latter arrow to a lax symmetric monoidal functor, we see that $L: \cC \to \cC$ is lax symmetric monoidal. This proves the second claim.
		
		Because lax symmetric monoidal functors send commutative algebras to commutative algebras, we conclude that $A$ (the unit of $L\cC$, hence a commutative algebra in $L\cC$) is sent to a commutative algebra object in $\cC$. This proves the last claim.
	\end{proof}
	
	\begin{remark}
		If one likes, one need not use the ``lax'' terminology in the above proof. Corollary~7.3.2.7 of~\cite{higher-algebra} guarantees that the right adjoint becomes a map of $\infty$-operads $j:(L \cC)^{\tensor} \to \cC^{\tensor}$, hence one can compose any map of $\infty$-operads from the category of finite pointed sets to $L\cC$ with $j$. On the other hand, an $\infty$-operad map from the category of finite pointed sets is the definition of a commutative algebra object.
	\end{remark}
	
	\begin{remark}
		Though we will not need this, we note that Corollary~\ref{corollary. idempotent objects are algebras} is functorial in the $(A,u)$ and $\cC^{\tensor}$ variables. Call an $\EE_0$-algebra $u: 1_\cC \to A$  {\em idempotent} if the condition in Corollary~\ref{corollary. idempotent objects are algebras} is satisfied. It is straightforward to modify the proof to exhibit a functor from the $\infty$-category of idempotent $\EE_0$-algebras in $\cC^\tensor$ to the $\infty$-category of commutative algebra objects in $\cC^\tensor$. This can further be promoted to be a map of coCartesian fibrations over the $\infty$-category of symmetric monoidal categories, mapping the coCartesian fibration classifying idempotent $\EE_0$-algebras to the coCartesian fibration classifying commutative algebras.
	\end{remark}
	
	\begin{remark}
		In the setting of Corollary~\ref{corollary. idempotent objects are algebras}, let $m: A \tensor A \to A$ denote the product of $A$ associated to the guaranteed commutative algebra structure of $A$. Since $A$ is a unital (commutative) algebra, we know that the composition
		\eqnn
		A \simeq 1 \tensor A \xrightarrow{u \tensor \id_A} A \tensor A \xrightarrow{m} A
		\eqnd
		is homotopic to the identity. On the other hand, $u \tensor \id_A$ is an equivalence by assumption. Thus, we see that $m$ is in fact an equivalence, and is naturally (up to structure maps guaranteed by the symmetric monoidal structure of $\cC$) identified as a homotopy inverse to $u \tensor \id_A$.
	\end{remark}
	
	\begin{remark}
		In the setting of Corollary~\ref{corollary. idempotent objects are algebras}~\eqref{item. L is lax},
		the lax monoidal structure maps
		\eqnn
		L(X) \tensor L(Y) \to L(X \tensor Y)
		\eqnd
		are all equivalences. 
		To see why, one can unwind the definitions and observe that the structure maps are the equivalences
		\eqnn
		A \tensor X \tensor A \tensor Y \to A \tensor X \tensor Y
		\eqnd
		guaranteed by the idempotent structure.
		
		However, the structure map from the unit $1_{\cC} \to L(1_{\cC}) \in \cC$ may not be an equivalence---see Example~\ref{example. idempotents}. This is the only reason that the right adjoint is not symmetric monoidal, and is only lax. 
	\end{remark}

	\section{Two models for P-flexibilization}\label{sec: first_look_flexibilization}
	
	In this section, we review a construction from the first two authors' previous work~\cite{Lazarev_Sylvan} that localizes the wrapped Fukaya category.  This construction has the advantage of taking Weinstein sectors to equi-dimensional Weinstein domains (and takes Weinstein domains to domains, instead of sectors). A priori it is not Weinstein homotopy invariant and not defined for general Liouville sectors (as opposed to Weinstein sectors).  In Section~\ref{sec: htpy_invariant_P-flex}, we introduce an alternative P-flexibilization functor, which  manifestly preserves homotopy equivalences.  
	
	\subsection{Carving out Lagrangian disks}\label{sec: carving_out}

	The construction from~\cite{Lazarev_Sylvan} relies on a procedure for removing Lagrangian disks, which we now explain. 
    Consider $T^*D^n$ with Liouville structure induced by the smooth vector field on the zero-section $D^n$ that vanishes on $\partial D^n \times [1-\epsilon, 1]$, points towards the origin in $\partial D^n \times (0, 1-\epsilon]$, and has an isolated zero at the origin; see the left diagram in Figure \ref{fig: Weinstein_TDn}. 
    Let $T^*D^n_{1/2}$ denote 	the cotangent bundle of the disk of radius $1/2$, with Liouville structure inherited from $T^*D^n$. Note that the Liouville vector field points inward near $\partial T^*D^n_{1/2}$ so that $T^*D^n_{1/2}$ is not a sector.
	Let $X^{2n}$ be a Liouville sector with a properly embedded Lagrangian disk $C$ so that 
	$\lambda_X|_C  = 0$. 
	\begin{defn}
		\label{defn. parametrization}
		A strict proper inclusion
		$$
		\phi_C: T^*D^n_{1/2} \hookrightarrow X
		$$
		that sends the cotangent fiber $T^*_0 D^n$ to the Lagrangian $C$ is called a {\em parametrization} of a neighborhood of $C$. 
	\end{defn}	
	
	\begin{remark}
		We note that for any Lagrangian disk $C$ with Legendrian boundary, there is a Liouville homotopy supported in a neighborhood of $C$ (hence an interior Liouville homotopy)  so that a neighorhood of $C$ has a parametrization for the new Liouville form. If $C$ is strictly exact, then there is a Liouville homotopy $\lambda_t$ relative to $C$, {ie }vanishes on $C$, so that $C\subset (X, \lambda_1)$ has a parametrization. 
	\end{remark}

    Since the Liouville vector field points into $\phi_C(T^*D^n_{1/2})$ near $\phi_C(\partial T^*D^n_{1/2} )$, $X \setminus \phi_C(T^*D^n_1/2)$ is a new Liouville sector. More precisely, since $\phi_C(T^*D^n_{1/2})$ is contained in the interior of $X$ and away from the sectorial boundary of $X$,   
    $X \setminus \phi_C(T^*D^n_{1/2})$ has the same sectorial boundary as $X$. However, when we remove $\phi_C(T^*D^n_1/2)$  from $X$, we remove part of the contact boundary of $X$, namely the intersection of $\phi_C(T^*D^n_{1/2})$ with the contact boundary of $X$, {ie }a neighborhood of the Legendrian boundary of  $C$. On the other hand,     
    $X \setminus \phi_C(T^*D^n_1/2)$ has a new portion of contact boundary corresponding to $\phi_C(\partial T^*D^n_{1/2})$; since the Liouville vector field points into $\phi_C(T^*D^n_{1/2})$ near $\phi_C(\partial T^*D^n_{1/2})$, the Liouville vector field on   $X \setminus \phi_C(T^*D^n_{1/2})$ points \textit{outward} along  $\phi_C(\partial T^*D^n_{1/2})$, which therefore is part of the contact boundary of $X \setminus \phi_C(T^*D^n_{1/2})$.
    That is, the contact boundary of  $X \setminus \phi_C(T^*D^n_{1/2})$ is obtained from the contact boundary of $X$ by contact \textit{anti}-surgery on the Legendrian boundary $\partial C$, which lies in the contact boundary of $X$.

	\begin{notation}[$X \setminus C$]
		\label{notation. X minus C}
		Often, we will drop the choice of parametrization $\phi_C$ and use the notation $X\setminus C$ to denote this sector, which we say is obtained from $X$ by \textit{carving out} $C$. 
	\end{notation}
	
	\begin{remark}	
		The following fact will be used repeatedly in the proofs of the main results Theorems \ref{thm: intro_idempotent_functor}, \ref{thm: intro_comparison}: If $(X, \lambda_t)$ is a homotopy relative to $C$, then there is an induced homotopy $(X \backslash C, \lambda_t)$.
	\end{remark}
	
	Next we discuss the Weinstein case. 
	\begin{example}
		Any Weinstein sector $X^{2n}$ has a distinguished collection of Lagrangian disks, namely the co-cores $C_i$ of the index $n$ critical points, which can be parametrized if they are properly embedded.
		Note that if the critical values of the index $n$ critical points is larger than the critical values of the lower index critical points, then the Lagrangians co-cores are properly embedded. Any Weinstein  structure can be homotoped through Weinstein  structures to one of this form and hence to one whose index $n$ co-cores are properly embedded. 
		Then for each $i$, $X \setminus \phi_{C_i}(T^*D^n_{1/2})$ is  a Weinstein sector and $X \setminus \coprod_i \phi_{C_i}(T^*D^n_{1/2})$ is a subcritical Weinstein sector.

		More generally, by Proposition 2.3 of~\cite{EGL}, any regular Lagrangian disk (where the Liouville vector field is taken to point \textit{outward} near $\partial T^*D^n_{1/2}$) has a parametrized neighborhood  (where the Liouville vector field points \textit{inward} near $\partial T^*D^n_{1/2}$)  after a further Weinstein homotopy. 	\end{example}

\begin{examples}\label{examples: carving out Lagrangian unknot}
Next, we define a class of Lagrangians called Lagrangian unknots and describe the result of carving out these Lagrangians. Any point of the contact boundary of $X$ has a standard Darboux chart $B^{2n-1}$ in this contact boundary, and by extending this Darboux chart into $X$, we obtain a standard chart in $X$. This chart in $X$ can viewed as a proper inclusion of 
$T^*D^n_{+} \hookrightarrow X$. Here $T^*D^n_+$  has Liouville structure induced by the vector field on $D^n_+$ that is inward-pointing along the one hemisphere $D^{n-1}_-$ of $\partial D^n = S^{n-1}$ and outward-pointing along the other hemi-sphere $D^{n-1}_+$, and the proper inclusion sends $\partial D^{n_1}_+$ to the contact boundary of $X$. 
Then a Lagrangian $L$ in $X$ is called a \textit{Lagrangian unknot} if it arises as  the image of the cotangent fiber of $T^*D^n_+$ in $X$. We emphasize that Lagrangian unknots are disks rather than spheres, as in smooth topology.
Alternatively, a Lagrangian $L$ is a Lagrangian unknot if there is a Weinstein homotopy of $(X, \lambda_0)$ to $(X, \lambda_1)$ that creates two \textit{cancelling} Weinstein handles of index $n-1, n$, above all existing Weinstein handles, and the $L$ is the co-core of the index $n$ handle. Consequently, by removing $L$ and forming 
$X\backslash L$, we are still left with the index $n-1$ handle. Hence $X\backslash L$ is obtained from $X$ by adding an index $n-1$ Weinstein handle along a subcritical isotropic sphere. 
\end{examples}

	\begin{defn}
		Let $\weinparam$ be the category whose objects are $(X, \{\phi_{C_X}\})$, where $X$ is a  Weinstein sector with  parametrized  \textit{properly embedded} Lagrangian co-cores $\{C_X\}$ and $\phi_{C_X}: T^*D^n_{1/2} \hookrightarrow X$ is a strict proper inclusion parametrizing $C_X$.		Morphisms are strict proper inclusions of Liouville sectors that respect this parametrization. Namely, if $i: X\hookrightarrow Y$ is a strict proper inclusion, then
		each co-core $C_X$ of $X$ is also a co-core of $Y$ and we require $i \circ \phi_{C_X}  = \phi_{C_Y}$, where $\phi_{C_X}, \phi_{C_Y}$ are the parametrizations of the same co-core in $X, Y$ respectively. Note that there is a forgetful functor from $\weinparam$ to $\weinstr$.	
		
	\end{defn}

	\subsection{P-flexibilization for strict, parametrized Weinstein sectors}\label{sec: P-flexibilization_strict_parametrized}
	
	Recall that if $L \subset (T^*D^n, \lambda_{T^*D^n, std})$ is a regular disk (Definition~\ref{defn. regular disk}), then there is an interior Weinstein homotopy from the standard structure to a Weinstein structure $({T^*D^n}, \lambda_L)$ for which $L$ is strictly exact and has a parametrized neighborhood. 
	\begin{notation}
		\label{notation. L disk}
		We set
		$$
		(T^*D^n)_L := (T^*D^n, \lambda_L) \setminus \phi_L(T^*D^n_{1/2}).
		$$
		(See Notation~\ref{notation. X minus C}.)
	\end{notation}
	
	This is a Weinstein sector whose sectorial boundary has a canonical identification with the sectorial boundary $\partial T^*D^n$  of $T^*D^n$. In fact, we will assume that $\phi_L(T^*D^n_{1/2})$ is contained in $T^*D^n_{1/2}$ and that the Weinstein homotopy from $(T^*D^n, \lambda_{T^*D^n, std})$ to $({T^*D^n}, \lambda_L)$ is supported inside this region; hence, we can also consider the subset $(T^*D^n_{1/2})_L: = (T^*D^n_{1/2}, \lambda_L)\setminus \phi_L(T^*D^n_{1/2})$.	
	Later we will take a more explicit model for $(T^*D^n)_L$ suitable to our purposes. 
	Finally, we also fix parametrizations of the Lagrangian 
	co-cores of $(T^*D^n)_L$, which we can assume are properly embedded.

	More generally, consider a Weinstein  sector $(X, \phi)$ with parametrized middle-dimensional co-cores $\{C_X\}$. So we have, for each middle-dimensional cocore $C_X$, a strict proper inclusion $\phi_{C_X}: T^*D^n_{1/2} \hookrightarrow X$ that takes $T^*_0 D^n$ to $C_X$. 
	\begin{notation}
		We define
		$$
		X_L: = (X \setminus \coprod_{C_X} \phi_{C_X}(T^*D^n_{1/2})) \cup \coprod_{C_X} (T^*D^n_{1/2})_L
		$$
		where we glue $\phi_{C_X}(\partial T^*D^n_{1/2})$ and the  copy of $\partial (T^*D^n_{1/2})_L$ corresponding to $C_X$. 
	\end{notation}
	This is a Weinstein sector and has parametrized critical Weinstein handles since this is true for 
	$(X \setminus \coprod_{C_X} \phi_{C_X}(T^*D^n_{1/2}))$, which is subcritical, and also true for $\coprod_{C_X} (T^*D^n_{1/2})_L$. 
	Note that $X_L$ makes sense even if $L \subset T^*D^n$ is a finite set of several disjoint Lagrangian disks. 
	
	\begin{remark}
		The whole construction of $X_L$ can be summarized by taking a certain Weinstein homotopy $(X, \lambda_X)$ to $(X, \lambda_{X, L})$ (supported near the index $n$ co-cores) which now has many copies of $L$ as co-cores and a Liouville vector field that points toward those co-cores, and then carving out those copies of $L$. 
		More precisely, if $C_X$ is a co-core of $X$, then we let $C^L_X := \phi_{C}(L) \subset X$ denote the copy of $L$ in the tubular neighborhood $\phi_{C_X}(T^*D^n)$ of $C_X$; then we also have
		$$
		X_L \cong (X, \lambda_{X, L}) \setminus \coprod_{C_X} {C_X^L}
		$$
		or just 
		$$
		X \setminus \coprod_{C_X} {C_X^L}
		$$
		when $\lambda_{X, L}$ is clear. 
	\end{remark}

	By construction,  there is a sectorial subdomain inclusion $X_L \subset X$; by Proposition \ref{prop: convert_subdomain_to_proper_inclusion}, this gives rise to a morphism $X \rightarrow X_L$ in the critical category $\weincrit$. 	
	By~\cite{ganatra_generation}, this implies that the Fukaya category of $X_L$ is a localization of the Fukaya category of $X$ obtained by nullifying $L$ viewed as an object the Fukaya category; that is, $Tw \  \mathcal{W}(X_L) \simeq Tw \ \mathcal{W}(X)/\coprod_{\{C_X\} } C_X^L$.

	Next we discuss functoriality. 
	\begin{notation}[$i_L$]
		\label{notation. i_L}
		Given a strict proper inclusion of Weinstein sectors $i: X \hookrightarrow Y$ preserving parametrizations of properly embedded Lagrangian co-cores, there is an induced strict proper inclusion
		$$
		i_L: X_L \hookrightarrow Y_L 
		$$
		of Weinstein  sectors (again preserving parametrizations of co-cores). To see this, recall that the Lagrangian co-cores of $X$ are a subset of the Lagrangian co-cores of $Y$ (Remark~\ref{remark. cocores of X and Y}) so when we carve out copies of $L$ near the Lagrangian disks of $Y$, we do the same for $X$. Furthermore, $(j \circ i)_L = j_L \circ i_L$ for the same reason.		
		Then $( \ )_L$ defines an endofunctor of $\weinparam$.
	\end{notation}

	The main issue is that the functor $(\ )_L$ does not preserve Weinstein (or Liouville) homotopy equivalences. 
	More precisely, if $X, X'$ are Weinstein homotopic Weinstein  sectors, then $X_L$ and $X_L'$ need not be Weinstein homotopic, {ie }$X_L$ depends on the Weinstein presentation of $X$. 
	For example, we can Weinstein homotope $X$ 
	to a Weinstein presentation $X'$ with many more index $n$ handles, in which case $X_L'$ is not even homotopy equivalent to $X_L$ as topological spaces. Next, we give a more drastic example. 
    
	\begin{examples}\label{example: non-homotopy}
		Consider the standard Weinstein presentation $(T^*S^n, \lambda_0)$ with a single co-core $T^*_x S^n$ or  homotopic structure $(T^*S^n, \lambda_1)$ with two co-cores $T^*_{x_1} S^n, 
		T^*_{x_2} S^n$; these structures are obtained by taking Morse functions on $S^n$ with either one or two index $n$ critical points. Then $(T^*S^n, \lambda_0)_L$ is $T^*S^n\setminus (T^*_x S^n)^L$,  while
		$(T^*S^n, \lambda_1)_L$ is $T^*S^n\setminus (T^*_{x_1} S^n)^L \coprod (T^*_{x_2} S^n)^L$. So in the former case we carve out one copy of $L$ while in the latter case we carve out two copies of $L$. Even if we add flexible handles to make these two domains diffeomorphic, it is not at all clear that they would become symplectomorphic (and we do not know whether this is the case). The issue is that  ${(T^*_{x_1} S^n)}^L \coprod {(T^*_{x_2} S^n)}^L$ is not necessarily a \textit{parallel} Lagrangian link, as we discuss further in Section \ref{ssec: idempotency_proof}. However, the realization that these spaces become Weinstein homotopic after stabilizing and inverting subcriticals was the original impetus for this project. To show that $P$-flexibilization is completely Weinstein homotopy invariant, one further needs to consider arbitrary Weinstein homotopic structures $(T^*S^n, \lambda)$ on $T^*S^n$, not just the simple structure considered above. 
	\end{examples}
	Furthermore, $X_L$ is only defined for Weinstein sectors, but not general Liouville sectors. We will remedy these issues in the next section.

	\subsubsection{Abouzaid-Seidel and Lazarev-Sylvan constructions}\label{sec: abouzaid_seidel}
	Next, we explain how to construct $X[P^{-1}]$ for a set of integers $P$
	and any Weinstein sector $X$	 with $\dim X \ge 10$. We follow previous work of the first two authors \cite{Lazarev_Sylvan}, which is a variant of the construction of Abouzaid and Seidel \cite{abouzaid_seidel_recombination}.
	First for $p \in P$, we consider  a certain Lagrangian disk $D_p \subset T^*D^n$ from \cite{abouzaid_seidel_recombination}, Section 3b. We briefly recall its definition: $D_p$ is the graph of the differential $d(f_p) \subset T^*D^n$ for a function $f_p: D^n \rightarrow \mathbb{R}$, which near $\partial D^n$ looks like $r^2 g_p(\theta)$ 
	for a function $g_p: S^{n-1}\rightarrow \mathbb{R}$ that is positive on the tubular neighborhood of a $p$-Moore space $U_p \subset S^{n-1}$ and negative on the complement of this $p$-Moore space; here $r$ is radial coordinate and $\theta$ is a coordinate on $S^{n-1} = \partial D^n$. This construction requires embedding a $p$-Moore space in $S^{n-1}$ and hence works only for $n \ge 5$. 

    Next, consider the disjoint embedding $\coprod_{p \in P} D_p \subset T^*D^n$ (obtained by embedding $|P|$ copies of $D^n$ disjointly into $D^n$ and considering the induced proper inclusion $\coprod_{p \in P} T^*D^n \hookrightarrow T^*D^n$).     
We also note that $D_p \subset T^*D^n$ is a regular Lagrangian disk because $D_p$ is Lagrangian isotopic to the zero-section $D^n \subset T^* D^n$. Indeed,  $T^*D^n$ is associated to the stopped domain $(B^{2n}, \partial D^n)$ and $B^{2n}$ is a tubular Weinstein neighborhood $T^*D^n$ of the zero-section $D^n$. Since $D_p$ is Lagrangian isotopic to the zero-section $D^n$, $B^{2n}$ is also a tubular Weinstein neighborhood of $D_p$. Hence, $T^*D^n$ is obtained from a Weinstein tubular neighborhood of $D_p$ by adding a stop, 
and so $D_p$ is a regular Lagrangian.

	The definition of $X[P^{-1}]$ by the first two authors \cite{Lazarev_Sylvan} is essentially 
	\begin{notation}\label{notation. X_P}
		$X[P^{-1}]: = X_{\coprod_{p\in P} D_p}$ 
	\end{notation}
	More precisely, they also added flexible handles to $X_{\coprod_{p\in P} D_p}$ to ensure that the result is diffeomorphic to $X$. We will not do this in this paper; see the discussion in the next section. 
    \begin{examples}\label{example: X1 is X}    
If $P$ is the empty set, then no Lagrangian disks are carved out of $X$ and so $X[P^{-1}] = X$. 

If $P = 1$, then the 1-Moore space is homotopy equivalent to a point and $D_1$ is a Lagrangian unknot. Hence by Example \ref{examples: carving out Lagrangian unknot}, $X[\{1\}^{-1}]$ is equivalent to $X$ plus some subcritical handles of index $n-1$, one for each index $n$ handle of $X$. In particular, $X[\{1\}^{-1}]$ is equivalent to $X$ in $\weincrit$. 
\end{examples}

	\begin{remark}
		\label{remark. D_U for other U}
		Note the construction of the Lagrangian disk $D_p$ can be generalized by using any codimension zero subdomain $U \subset S^{n-1}$ instead of just the p-Moore space $U_p$. Indeed, Abouzaid and Seidel \cite{abouzaid_seidel_recombination}
		showed that the Lagrangian disk $D_U \cong \tilde{C}^{*-1}(U) \otimes T^*_0 D^n$ in the wrapped Fukaya category $Tw \  \mathcal{W}(T^*D^n)$; see~\cite{Lazarev_Sylvan} for details. 
		So if $U$ is a $p$-Moore space, then $D_{U} = D_p \cong (\mathbb{Z}[1] \overset{p}{\rightarrow} \mathbb{Z}) \otimes T^*_0 D^n$ in $Tw \ \mathcal{W}(T^*D^n)$. Using this result along with the localization formula of Ganatra-Pardon-Shende \cite{ganatra_generation}, the first two authors \cite{Lazarev_Sylvan} proved that  
		$$
		Tw \  \mathcal{W}(X[P^{-1}]) \cong Tw \ \mathcal{W}(X) \otimes \mathbb{Z}\left[\frac{1}{P}\right].
		$$
		We observe that since they are all graphical Lagrangians, the disks $D_U \subset T^*D^n$ are all Lagrangian isotopic to the zero-section $D^n \subset T^*D^n$ if we allow the boundary of $D_U$ to intersect the sectorial divisor, {ie }$D_U$ is isotopic to $D^n$ in the \textit{unstopped} domain $B^{2n}$. 
	\end{remark}

	\subsubsection{Smooth topology of $X_L$ and flexible handles}\label{sec: topology_flex_handles}

	From the point of view of classifying symplectic structures on a fixed smooth manifold, it can be desirable to have symplectic constructions preserve diffeomorphism type. 
	For example, the flexibilization $X_{flex}$ of $X$ defined by Cieliebak  and Eliashberg \cite{CE12} is a flexible domain diffeomorphic to $X$. 
	As we will see in Proposition~\ref{prop: carve_double L}, in this paper we can ensure that all constructions preserve the diffeomorphism type up to smoothly subcritical handles; since we work in $\weincrit$, Weinstein sectors up to subcritical handles and stabilization, this is the only sensible notion.  
	
	Since $X_L$ is obtained by removing a disk from $X$,  it is never diffeomorphic to $X$. Furthermore, if the (Poincar\'{e} dual) class $[C_X^L] \in H^n(X; \mathbb{Z})$ is non-zero,  
	then 
	$X$ and $X_L$ have different middle-dimensional cohomology and hence fail to be diffeomorphic up to subcriticals. However, if $L \subset T^*D^n$ is \textit{smoothly} isotopic to an unknotted disk, then $X_{L}$ is diffeomorphic to X (but not necessarily symplectomorphic to X)  up to adding subcritical handles.
	We also observe that for any Lagrangian disk $L\subset T^*D^n, n \ge 5$, the double $L \natural \overline{L}$ is \textit{smoothly} isotopic to the unknotted disk; here $ \overline{L}$ is $L$ with the opposite orientation and $\natural$ is the isotropic boundary connected sum. Furthermore, if $L$ is regular, then so is $L \natural\overline{L}$. 
	\begin{proposition}\cite{Lazarev_geometric_presentations}\label{prop: carve_double L}
		For any regular Lagrangian disk $L \subset T^*D^n, n \ge 3$,
		$X_{L\natural \overline{L}}$ is Weinstein homotopic to $X_L$ plus some flexible handles  and is diffeomorphic to $X$ up to smooth subcritical handles. 
	\end{proposition}
	So whatever we can prove about $X_L$ (for arbitrary regular Lagrangians $L$), we can also prove about $X_{L\natural\overline{L}}$, which is diffeomorphic to $X$ up to smooth subcritical handles. 
	We also note that $Tw\  \mathcal{W}(X_{L\natural \overline{L}}) \cong Tw\ \mathcal{W}(X_L)$ since they are related by flexible handles. 
	
	In Proposition \ref{prop: carve_double L},    $X_{L\natural \overline{L}}$ is Weinstein homotopic to $X_L \cup H^{n-1} \cup H^n_{flex}$, where one first attaches a subcritical handle $H^{n-1}$ and a flexible handle $H_{flex}^n$ (attached along a loose Legendrian) for each handle of $X$. In particular, there is no need to discuss flexible cobordisms separately in this paper since they appear naturally by removing the Lagrangian $L\natural\overline{L}$ (instead of $L$). Using the h-principle for loose Legendrians in a certain local setting, the first author~\cite{Lazarev_geometric_presentations}  also showed that the flexible handle $H^n_{flex}$ can be attached \textit{before} the subcritical handle $H^{n-1}$
	and hence 
	$X_{L\natural \overline{L}} = (X_L \cup H^{n-1}) \cup H^{n}_{flex} 
	$ is Weinstein homotopic to 
	$(X_L \cup  H^{n}_{flex}) \cup H^{n-1} =: X_L'\cup H^{n-1}$, where $X_L'$ is actually diffeomorphic to $X$; as before, $Tw\  \mathcal{W}(X_{L}') \cong Tw\ \mathcal{W}(X_L)$ since attaching subcritical handles doesn't affect the Fukaya category. So in fact, $X_L'$ and $X_{L\natural L}$ are equivalent in $\weincrit$, assuming this local form of h-principle.

	\begin{example}\label{example: X0 is flexible}
If $P = \{0\}$, we observe that $X[\{0\}^{-1}]$ is flexible.   To see this, we first consider the cotangent fiber $L = T^*_0 D^n$ and observe that $X_{T^*_0 D^n}$ is the subcritical part $X_{sub}$ of $X$. Also, $X_{T^*_0 D^n\natural\overline{T^*_0 D^n}} = X_{sub}\cup H^{n-1}\cup H^n_{flex}$ is flexible  and diffeomorphic to $X$ up to some smooth  subcritical handles. 
		Using the h-principle for loose Legendrians in the local setting,~\cite{Lazarev_crit_points, Lazarev_geometric_presentations} verified that  $X_{T^*_0 D^n\natural\overline{T^*_0 D^n}}$ is actually Weinstein homotopic to $X_{flex} \cup H^{n-1}$, where $X_{flex}$ is the flexibilization defined in~\cite{CE12} (a flexible domain diffeomorphic to $X$).
		In particular,  $X_{T^*_0 D^n\natural\overline{T^*_0 D^n}}$ is equivalent to $X_{flex}$ in $\weincrit$. 
Finally, we note that $T^*_0 D^n \natural \overline{T^*_0 D^n}$ is Lagrangian isotopic to $D_0$, the Abouzaid-Seidel disk constructed using a $0$-Moore space $S^1 \vee S^2$; see \cite{Lazarev_Sylvan} for a proof. 
		So by definition $X[\{0\}^{-1}]$ is $X_{{T^*_0 D^n} \natural \overline{T^*_0 D^n}}$. Therefore $X[\{0\}^{-1}]$ is flexible and diffeomorphic to $X$ up to subcritical handles (and equivalent to $X_{flex}$ is $\weincrit$ assuming the local h-principle). 

        However, if one is not interested in comparing $X_{T^*_0 D^n\natural \overline{T^*_0 D^n}}$ to  $X_{flex}\cup H^{n-1}$ (but only using the fact that $X_{T^*_0 D^n\natural \overline{T^*_0 D^n}}$ is flexible),  there is no need to use the h-principle for loose Legendrians. We take this approach, viewing $X_{T^*_0 D^n\natural \overline{T^*_0 D^n}}$ as the flexibilization (since it is flexible and diffeomorphic to $X$ up to subcriticals) 
		and proving idempotency and independence of presentation for this domain. 
		Since the h-principle for loose Legendrians is used in a single local setting, we also expect that it is possible to compare these domains explicitly without using the h-principle at all.

	\end{example}
	
	\begin{remark}[Carving Lagrangians may be assumed connected]
		Returning briefly to the construction of $X[P^{-1}]$ in the previous section, 
		we note that $D_p$ is trivial in $H^n(X)$
		(this is true for $D_U$ for any $U$ with Euler characteristic $\chi(U) = 1)$. So by the smooth h-cobordism theorem, $X[P^{-1}]$ is diffeomorphic to $X$, up to subcritical handles. We also mention a slight variant which uses connected Lagrangian disks. Namely, we form the  Lagrangian disk $\natural_{p \in P} D_p$, the isotropic boundary connected sum of $\coprod_{p\in P} D_p$. Then, as proven in \cite{Lazarev_geometric_presentations},  $X_{\natural_{p \in P} D_p}$
		is obtained by adding a flexible cobordism to 	$X[P^{-1}]:= X_{\coprod_{p \in P} D_p}$.
		However, since $D_p$ (and hence both $\coprod_{p \in P} D_p$ and $\natural_{p \in P} D_p$) is trivial in $H^n(X)$, this flexible cobordism is in fact a subcritical Weinstein cobordism; so $X_{\natural_{p \in P} D_p}$ and $X[P^{-1}]$ are equivalent in $\weincrit$. In particular, we can assume that $X[P^{-1}]$ is constructed by carving out the \textit{connected} Lagrangian disk 
		$\natural_{p \in P} D_p$ near the co-cores of $X$. 
	\end{remark}
	\begin{remark}\label{rem: prime_factorization}
		If $n = p_1 \cdots p_k$ is a product of distinct primes (or relatively coprime integers), then the $n$-Moore space $M_n$ is homotopy equivalent to the wedge sum 
		$M_{p_1}\vee\cdots \vee M_{p_k}$ (assuming the Moore spaces are simply-connected). To see this, note that there are always  maps $M_n\rightarrow \vee_{p_i \in P} M_{p_i}$ and the coprime condition implies that this is a homology isomorphism (and hence a homotopy equivalence by Whitehead's theorem). The construction of Lagrangian disks $D_U$ in $T^*D^n$  from spaces $U$ in $S^{n-1}$ takes wedge sum to isotropic connected sum and therefore the Lagrangian disk $D_n$
		is Lagrangian isotopic to the Lagrangian disk $D_{p_1} \natural \cdots\natural D_{p_k} = \natural_{p_i} D_{p_i}$. So by the previous remark,
		$X_n$ is equivalent to $X[P^{-1}]$, where 
		$P$ is the set of \textit{primes}  $\{p_1, \cdots, p_k\}$. 
	\end{remark}

	\subsection{A homotopy invariant P-flexibilization functor}
	\label{sec: htpy_invariant_P-flex}
	
	Next, we discuss the present work's P-flexibilization functor which is manifestly homotopy invariant, and which we prove in Section \ref{sec: comparison_P_flexibilization} to be equivalent to the non-homotopy-invariant construction 
	$(\ )_L$ defined in the previous section. 
	
	To do so, we fix a regular Lagrangian disk $L \subset T^*D^n$ and  a Weinstein structure $(T^*D^n)_L$. Then taking the product with $(T^*D^n)_L$ defines a functor
	$$
	\times (T^*D^n)_L: \lioustr \rightarrow \lioustr
	$$
	that takes a strict proper inclusion $f$ to $f\times Id_{(T^*D^n)_L}$, 	
	as well as a restricted functor
	$$
	\times (T^*D^n)_L: \weinstr \rightarrow \weinstr
	$$
	since the product of Weinstein sectors is a Weinstein sector. 
The next result shows that this functor is homotopy-invariant and also subcritical invariant.
    \begin{lemma}
    \label{lemma. product is functor}
        The endofunctor $\times (T^*D^n)_L: \weinstr \rightarrow \weinstr$ descends to an endofunctor (also denoted $\times (T^*D^n)_L$) on $\weinstab$ and $\weincrit$.  That is, there are commutative diagrams
\begin{equation}
\begin{tikzcd}
   \weinstr \arrow{r}{\times (T^*D^n)_L}\arrow{d} & \weinstr \arrow{d}\\
   \weinstab \arrow{r}{\times (T^*D^n)_L} & \weinstab
\end{tikzcd}
\end{equation}
\begin{equation}
\begin{tikzcd}
   \weinstr \arrow{r}{\times (T^*D^n)_L}\arrow{d} & \weinstr \arrow{d}\\
   \weincrit \arrow{r}{\times (T^*D^n)_L} & \weincrit\\
\end{tikzcd}  
\end{equation}
where the vertical functors are the localization functors used to define $\weinstab$ and $\weincrit$. 
    \end{lemma}
    \begin{proof}
Consider the composite functors $\weinstr \overset{\times (T^*D^n)_L}{\rightarrow} \weinstr\rightarrow \weinstab$ and $\weinstr \overset{\times (T^*D^n)_L}{\rightarrow}\weinstr\rightarrow \weincrit$, where in both cases the second functor is the localization functors used to define $\weinstab$ or $\weincrit$. If $f$ is a movie inclusion or subcritical morphism, then so is $f\times Id_{(T^*D^n)_L}$, which are isomorphisms in $\weinstab$ or $\weincrit$ respectively. So by the universal properties of $\weinstab, \weincrit$, the functors $\weinstr \overset{\times (T^*D^n)_L}{\rightarrow} \weinstr\rightarrow \weinstab$ and $\weinstr \overset{\times (T^*D^n)_L}{\rightarrow}\weinstr\rightarrow \weincrit$ factor through $\weinstab$, $\weincrit$ respectively, yielding the desired commutative diagrams. 
    \end{proof}

    Furthermore, $X \times (T^*D^n)_L$ has the advantage that it can be defined without parametrizing the co-cores of $X$; in fact, the definition of $X \times (T^*D^n)_L$ makes sense when $X$ is an arbitrary Liouville sector. 
	
	If $X$ is Weinstein, or the Kunneth formula holds for $X$, and $L = D_P$, then
	\begin{eqnarray}
	Tw \ \mathcal{W}(X \times (T^*D^n[P^{-1}])) \cong Tw \ \mathcal{W}(X) \otimes Tw \ \mathcal{W}(T^*D^n[P^{-1}]) \\
	\cong Tw \ \mathcal{W}(X) \otimes Tw \ \mathbb{Z}\left[\frac{1}{P}\right] \cong Tw \ \mathcal{W}(X)\left[\frac{1}{P}\right]
	\end{eqnarray}
	so that $Tw \ \mathcal{W}(X \times (T^*D^n[P^{-1}]))$ is equivalent to $Tw \ \mathcal{W}(X[P^{-1}])$, the Fukaya category of the non-homotopy-invariant P-flexibilization $X[P^{-1}]$ defined in the previous section.
	In the next section, we show that in fact $X_L$ and $X \times (T^*D^n)_L$ are  equivalent in the critical category $\weincrit$. 	

    Finally, we observe that there exist Weinstein homotopic structures  $(T^*D^n)_L$ with different number of critical points; however, the resulting endofunctors $\times (T^*D^n)_L$ of $\weinstab$ or $\weincrit$ are all equivalent since if $((T^*D^n)_L, \lambda_0), ((T^*D^n)_L, \lambda_1)$ are Weinstein homotopic, so are $X \times ((T^*D^n)_L, \lambda_0), X \times ((T^*D^n)_L, \lambda_1)$. Hence we are free to take slightly different models for $(T^*D^n)_L$ at different stages of the proof.

	\begin{remark}\label{rem: geometric_model2}		
		More generally, it is a priori possible that the Liouville homotopy type  $(T^*D^n)_L$ depends on more than just $L$. If $L$ is strictly exact with respect to two forms $\lambda_L, \lambda_L'$ that have the same behavior at infinity (and hence are linearly homotopic), then  
		$(T^*D^n, \lambda_L)\setminus L,
		(T^*D^n, \lambda_L')\setminus L $ are Liouville homotopic by the linear homotopy, which necessarily vanishes on $L$. However, if $L$ does not have Legendrian boundary with respect to $\lambda_{T^*D^n, std}$, then we need to pick a homotopic form $\lambda_L$ (or $\lambda_L'$) for which $L$ is strictly exact. Then $\lambda_L, \lambda_L'$ are Liouville homotopic but not necessarily through a family of forms vanishing on $L$; hence the homotopy does not descend to a homotopy on $T^*D^n \setminus L$. Said differently, we need to Lagrangian isotope $L$ to make it have Legendrian boundary with respect to $\lambda_{T^*D^n, std}$ and there are different ways of doing this. 			It is better to think of $\lambda_L$ as a whole path from $\lambda_{T^*D^n, std}$ to $\lambda_{L}$ and we need a path of paths to the path from $\lambda_{T^*D^n, std}$ to $\lambda_{L}'$; at the endpoints of this path, we get a Liouville homotopy that vanishes on $L$ as desired. 	
	\end{remark}

Finally, in the construction of $T^*D^n[P^{-1}] := T^*D^n \backslash (\coprod_{p \in P} D_p)$, there was a dependence on an embedding of a $p$-Moore space into $S^{n-1}$ (and a choice of $n$). Indeed, recall that we used this embedding to form a  tubular neighborhood $U_p \subset S^{n-1}$ of the p-Moore space, construct a function $f_p: D^n \rightarrow \mathbb{R}^1$, and then form the Lagrangian disk $D_p = d(f_p)$. In the following, we show that $T^*D^n[P^{-1}]$ is independent of these choices in $\weincrit$. 
\begin{proposition}\label{prop: T^*D[P^{-1}] independent of n}
The isomorphism class of $T^*D^n[P^{-1}]$ in $\weincrit$ is independent of $n$ and the choice of embedding of the p-Moore space into $S^{n-1}$. 
\end{proposition}
\begin{proof}
First, we consider the effect of increasing $n$. We can embed a p-Moore space into $S^{n-1}$ (for $n \ge 5$) and form its tubular neighborhood $U_p^{n-1} \subset S^{n-1}$. Alternatively, we can use the equatorial embedding $S^{n-1} \subset S^n$ to embed the p-Moore space into $S^n$ and take its  neighborhood $U_p^{n} \subset S^{n}$. Then the Lagrangian disk $D_p^{n+1} \subset T^*D^{n+1}$ (obtained using $U_p^{n}$) is Lagrangian isotopic to the stabilization of the Lagrangian disk $D_p^n\subset T^*D^n$ (obtained using $U_p^{n-1} \subset S^{n-1}$), {ie }$D_p^{n+1}$ is isotopic to $D_p^n \times T_0^*D^1 \subset T^*D^{n+1}$. 
Hence, $T^*D^{n+1} \backslash D_p^{n+1}$ is Weinstein homotopic to $T^*D^{n+1} \backslash (D_p^{n} \times T_0^* D^1)$, which is equivalent to $(T^*D^n\backslash D_p^n) \times T^*D^1$ in $\weincrit$ by Proposition \ref{prop. for comparison} below (a more precise version of Theorem \ref{thm: intro_comparison}). Hence, the sectors $(T^*D^{n+1})[P^{-1}], (T^*D^{n})[P^{-1}]$ (formed using the embeddings of the p-Moore space into $S^{n-1}$ or the induced embedding into $S^n$) are equivalent in $\weincrit$. 
Next,  we note that any two embeddings of p-Moore spaces into $S^{n-1}$ become smoothly isotopic in $S^{n-1}$ if $n \ge 6$ (since all cells in a p-Moore have dimension at most 2). Such smooth  isotopies induce Lagrangian isotopies of the respective Lagrangian disks $D_p$ in $T^*D^n$ and hence induce Weinstein homotopy equivalences of the respective sectors $T^*D^n\backslash D_p$.
\end{proof}

	\section{Comparison of P-flexibilization functors}\label{sec: comparison_P_flexibilization}
	
	The goal of this section is to prove the following result comparing $X_L := X \setminus 	\coprod C_X^L$  and $X \times (T^*D^n)_L$. 
	\begin{theorem}\label{thm: comparison}
		Consider a regular Lagrangian disk $L^n \subset T^*D^n$ (equipped with a Weinstein homotopy $\lambda_{T^*D^n, L, t}$ as in Definition \ref{defn. regular disk}). 
		Then for any parametrized Weinstein sector $X^{2n}$, there is an equivalence $\phi_X: X_L \tilde{\rightarrow}
		X \times (T^*D^n)_L$ in $\weincrit$. Furthermore, for any strict proper inclusion $f: X^{2n} \hookrightarrow Y^{2n}$ in $\weinparam$, the following is a homotopy commuting  diagram in the critical Weinstein category $\weincrit$:
		\begin{equation}
		\begin{tikzcd} 
		X_L \arrow{rr}{f_L}	
		\arrow{d}{\phi_X} 
        && 
		Y_L  \arrow{d}{\phi_Y}\\
		X \times (T^*D^n)_L  \arrow{rr}{f \times Id_{(T^*D^n)_L}}	
		&& 
		\hspace{0.3cm} Y\times (T^*D^n)_L
		\end{tikzcd}
		\end{equation} 
	\end{theorem}
	\begin{remark}
    Since $\weincrit$ is an $\infty$-category, it does not have a notion of strictly commuting diagrams; instead, it only has a notion of homotopy commuting diagrams, as in Theorem \ref{thm: comparison}. Furthermore, since there are no strict composition of morphisms in an infinity category, it does not have a notion of isomorphism, only equivalence (which gives an isomorphism in the homotopy category).

    The vertical equivalences $\phi_X, \phi_Y$ in the diagrams above are, in fact, \textit{zig-zags} of subcritical morphisms in $\weinstab$, and hence are equivalences in $\weincrit$. They arise by combining the vertical equivalences in Proposition \ref{prop. for comparison}.
	\end{remark}

	\subsection{Compatibility of taking products and carving out disks}\label{sec: compatibility_prod_carving}
	In this section we start the proof of
	Theorem \ref{thm: comparison} comparing the functors $\times (T^*D^n)_L $ and $(\ )_L$. To do so, we prove some slightly more general results concerning the compatibility between taking products and carving out Lagrangian disks. 
	
	Recall that given a strict proper inclusion $f: X \hookrightarrow Y$ of Weinstein  sectors $X,Y$ with parametrized Weinstein handles, 
	there is an induced strict proper inclusion  $i_L: X_L \hookrightarrow Y_L$ (Notation~\ref{notation. i_L}). 
	Also, for any Weinstein  sector $Z$ with parametrized handles, there   are induced inclusions $f \times Id: X \times Z \hookrightarrow Y \times Z$ and 	$(f\times Id)_L: (X \times Z)_L \hookrightarrow (X \times Z)_L$. 
	
	The following proposition compares these maps. 
	
	\begin{proposition}\label{prop. for comparison}
		Let $X^{2n}, Y^{2n}, Z^{2m}$
		be Weinstein  sectors with parametrized Weinstein handles, $f: X \hookrightarrow Y$ a strict proper inclusion, and $L \subset T^*D^n, K \subset T^*D^m$ parametrized regular Lagrangian disks (equipped with Weinstein homotopies $\lambda_{T^*D^n, L, t}, \lambda_{T^*D^m, K, t}$ as in Definition \ref{defn. regular disk}). 
Then there exist equivalences $\phi_{(X,L), Z}, \phi_{(Y,L), Z}$ and 
$\phi_{X,(Z,K)}, \phi_{Y, (Z,K)}$ so that there are homotopy commuting diagrams in $\weincrit$	
		\begin{equation}\label{comm: comparison_thm_statement1}
		\begin{tikzcd} 
		X_L \times Z  \arrow{rr}{f_L \times Id}	
		\arrow{d}{\phi_{(X,L), Z}} 
        && 
		Y_L \times Z  \arrow{d}{\phi_{(Y,L), Z}}\\
		(X \times Z)_{L\times T^*_0 D^m} \arrow{rr}{(f \times Id)_{T^*_0 D^n \times K}}  && \ \  (Y \times Z)_{L\times T^*_0 D^m} 			
		\end{tikzcd}
		\end{equation}
		\begin{equation}\label{comm: comparison_thm_statement2}
		\begin{tikzcd} 
		X  \times Z_K   \arrow{rr}{f \times Id_K}	
		\arrow{d}{\phi_{X,(Z,K)}} && 
		\ \ Y \times Z_K  \arrow{d}{\phi_{Y, (Z,K)}}\\
		(X \times Z)_{T^*_0 D^n\times K} \arrow{rr}{(f \times Id)_{T^*_0 D^n\times K}}  &&  \ \ (Y \times Z)_{T^*_0 D^n\times K} 
		\end{tikzcd}
		\end{equation}
		
	\end{proposition}
	
	\begin{remark}
		If Weinstein sectors $X, Y, Z$ have Weinstein stops $H_X, H_Y, H_Z$, then 
		$X_L \times Z$ has stop $H_X \times Z \coprod X_L \times H_Z$ 
		while 
		$(X\times Z)_{L \times T^*_0 D^m}$ has stop $H_X \times Z \coprod X \times H_Z$.  So the second components 
		$X_L \times H_Z, X \times H_Z$  of these two respective stops are different. 
		Part of the content of this proposition is that these two stops differ by a loose Legendrian hypersurface coming from a subcritical Weinstein cobordism, that gives rise to the equivalences in $\weincrit$.
	\end{remark}
	
	\begin{proof}[Proof of Proposition~\ref{prop. for comparison}.]
		The proof of Proposition \ref{prop. for comparison} follows from the following Proposition 
		\ref{prop: commutative_diagram_sector_localization} combined with Proposition \ref{prop: convert_subdomain_to_proper_inclusion} and Proposition \ref{prop: convert_homotopy_equivalence} which convert subcritical subdomain inclusions and isomorphisms up to deformation into equivalences in the critical Weinstein category.
For example, the morphism $\varphi_{(X,L),Z}$ is the composition of the strict subcritical subdomain inclusion $i_{(X, L), Z}$ from the following Proposition \ref{prop: commutative_diagram_sector_localization} and a bordered Weinstein homotopy. 
	\end{proof}

	Recall that $C_X$ is a co-core of $X$ with its original Weinstein structure $\lambda_X$, $C_X^L$ is a copy of $L$ embedded in a neighorhood of $C_X$,  $\lambda_{X, L}$ is the homotopic Weinstein structure for which $C_X^L$ is a co-core, and $X_L: = (X, \lambda_{X, L}) \setminus \coprod_{C_X} C_X^L.$
	
	\begin{proposition}\label{prop: commutative_diagram_sector_localization}
		Let $X^{2n}, Y^{2n}, Z^{2m}$ be Weinstein sectors with parametrized Weinstein handles, $f: X \hookrightarrow Y$ a strict proper inclusion, and $L \subset T^*D^n, K \subset T^*D^m$ parametrized regular Lagrangian disks (equipped with Weinstein homotopies as in Definition \ref{defn. regular disk}). Then the following hold:	
		\begin{enumerate} 
			\item\label{item. claim one of carving}  There is a strict subcritical subdomain inclusions 
			$$
			i_{(X, L), Z}: X^{2n}_L \times Z^{2m} := ((X, \lambda_{X, L}) \setminus \coprod_{C_X} C^L_X) \times Z  \xrightarrow{\subset} (X, \lambda_{X, L})\times Z\setminus \coprod_{C_X, C_Z} C_X^L \times C_Z
			$$
			where $C_X, C_Z$ are the co-cores of $X, Z$ respectively. 
            
			\item\label{item. claim two of carving} The identity map $Id_{(X, L), Z}$  
			$$
			(X, \lambda_{X, L})\times Z\setminus \coprod_{C_X, C_Z} {C_X^L \times C_Z} \rightarrow (X \times Z)_{L \times T^*_0 D^m}:= (X\times Z, \lambda_{X \times Z, L\times T^*_0 D^m})\setminus \coprod_{C_{X\times Z}} C_{X\times Z}^{L\times T^*_0 D^m}
			$$ 	is an isomorphism, up to   bordered Weinstein homotopy, {ie }these two sectors are bordered Weinstein homotopic.
			
			\item\label{item. claim three of carving}  	There is a strictly commuting diagram of symplectic embeddings: 
			\begin{equation}\label{comm: comparison_prop_XLZ}
			\begin{tikzcd} 
			X_L \times Z  \arrow{r}{f_L \times Id}	
			\arrow{d}{i_{(X,L), Z}} & 
			Y_L \times Z  \arrow{d}{i_{(Y,L), Z}}\\
			(X,\lambda_{X,L}) \times Z \setminus \coprod_{C_X, C_Z} {C_X^L\times C_Z}  \arrow{d}{Id_{(X,L), Z}} \arrow{r}{f \times Id}  &  (Y,\lambda_{Y,L}) \times Z \setminus \coprod_{C_X, C_Z} {C_X^L \times C_Z}  \arrow{d}{Id_{(Y,L), Z}}\\
			(X \times Z)_{L\times T^*_0 D^m} \arrow{r}{(f \times Id)_{T^*_0 D^n \times K}}  &  (Y \times Z)_{L\times T^*_0 D^m} 			
			\end{tikzcd}
			\end{equation}
			Furthermore, the top square is a pullback diagram of sets and the bordered Weinstein homotopies for $Id_{(Y,L), Z}$ extend those for $Id_{(X,L), Z}$. 	
			There are similar maps for $Z$ and a strictly commuting diagram of exact symplectic embeddings:
			\begin{equation}\label{comm: comparison_prop_XZK}
			\begin{tikzcd} 
			X  \times Z_K   \arrow{r}{f \times Id_{Z_K}}	
			\arrow{d}{i_{X,(Z,K)}} & 
			Y \times Z_K  \arrow{d}{i_{Y, (Z,K)}}\\
			X \times (Z, \lambda_{Z, K})
			\setminus \coprod_{C_X, C_Z} {C_X \times C_Z^K}  \arrow{d}{Id_{X, (Z, K)}}\arrow{r}{f \times Id_{Z_K}}  &  Y \times  (Z, \lambda_{Z, K}) \setminus \coprod_{C_X, C_Z} C_X \times C_Z^K \arrow{d}{Id_{Y, (Z, K)}}\\			
			(X \times Z)_{T^*_0 D^n\times K} \arrow{r}{(f \times Id)_{T^*_0 D^n\times K}}  &  (Y \times Z)_{T^*_0 D^n\times K} 
			\end{tikzcd}
			\end{equation}
		\end{enumerate}
	\end{proposition}

	\begin{proof}[Proof of Proposition \ref{prop: commutative_diagram_sector_localization}]

        We begin by proving \eqref{item. claim one of carving}. Our arguments will also prove the claims in \eqref{item. claim three of carving} regarding the top square in \eqref{comm: comparison_prop_XLZ}.
    
		There is a commutative diagram of symplectic embeddings of the form 
		\begin{equation}\label{comm: X_L-> X and Y}
		\begin{tikzcd} 
		X_L   \arrow{r}{f_L}	
		\arrow{d}{i_{(X,L)}} & 
		Y_L  \arrow{d}{i_{(Y, L)}}\\
		(X, \lambda_{X, L}) \arrow{r}{f}  &  (Y, \lambda_{Y, L})
		\end{tikzcd}
		\end{equation}
		where $	(X, \lambda_{X, L}), (Y, \lambda_{Y, L})$ are Weinstein homotopic to $X, Y$ respectively as in the construction of $X_L, Y_L$; in particular, $i_{(X, L)}, i_{(Y,L)}$ are strict Weinstein subdomain inclusions. 			
		This diagram commutes since the map $f$ takes parametrized co-cores of $X$ to those of $Y$ and hence the construction of $Y_L$ extends that of $X_L$. In particular, this is a pullback diagram of sets. 
		
		Next, we can take the product of this diagram with $Z$ and the identity map $Id_Z$ to obtain another commutative diagram of symplectic embeddings: 
		\begin{equation}\label{comm: X_L x Z -> X x Z, and Y}
		\begin{tikzcd} 
		X_L  \times Z \arrow{r}{f_L \times Id_Z}	
		\arrow{d}{i_{(X,L)} \times Id_Z} & 
		Y_L \times Z \arrow{d}{i_{(Y, L)} \times Id_Z }\\
		(X, \lambda_{X,L}) \times Z \arrow{r}{f \times Id_Z}  &  (Y, \lambda_{Y, L}) \times Z
		\end{tikzcd}
		\end{equation}
		This is also a pullback diagram of sets. 
		Since $X_L \xrightarrow{\subset} (X, \lambda_{X, L})$ is a strict Weinstein subdomain inclusion, then
		$X_L \times Z \xrightarrow{\subset} (X, \lambda_{X, L}) \times Z$ is a  strict sectorial subdomain inclusion, 
		and similarly for $Y$. 
		In particular, $((X, \lambda_{X, L})  \times Z) \setminus (X_L \times Z)$ is a Weinstein cobordism.
		Let $p_X^L$ be the index $n$ critical points of the Weinstein cobordism $(X, \lambda_{X, L}) \setminus i_{(X, L)}(X_L)$, whose co-core is $C_X^L$. 	
		Then the index $n+m$ critical points of the cobordism $((X, \lambda_{X, L})  \times Z) \setminus (X_L \times Z)$ correspond to pairs consisting of an index $n$ critical point $p_X^L$ and an index $m$ critical point $p_Z$ of the Weinstein structure on $Z^{2m}$. 
		The co-cores of these critical points are the products of the co-cores of the respective critical points, {ie }$C_X^L \times C_Z$.  Hence, if we carve out the Lagrangian co-core $C_X^L \times C_Z$ from this Weinstein cobordism, we get a subcritical Weinstein cobordism. 
		This is the desired strict subcritical subdomain inclusion $i_{(X, L), Z}$, 
        proving \eqref{item. claim one of carving}.
        
        Since the co-cores $C_X^L \times C_Z$ are a subset of the co-cores $C_Y^L \times C_Z$, the commutative diagram in equation \eqref{comm: X_L x Z -> X x Z, and Y} induces a commutative diagram which is the top square in equation \eqref{comm: comparison_prop_XLZ}; this remains a pullback diagram of sets since the co-cores $C_X^L \times C_Z$ are a subset of the co-cores $C_Y^L \times C_Z$.	
        This proves the claims in \eqref{item. claim three of carving} regarding the top square in \eqref{comm: comparison_prop_XLZ}.

        Now we establish \eqref{item. claim two of carving} and the claims in \eqref{item. claim three of carving} regarding the bottom square in \eqref{comm: comparison_prop_XLZ}.
		Let us first see that 
        the symplectic manifold $((X, \lambda_{X,L}) \times Z) \setminus \coprod_{C_X, C_Z} C_X^L \times C_Z$ 
        is symplectomorphic to
		$$(X\times Z)_{L \times T^*_0 D^m}:= 
		(X \times Z, \lambda_{X\times Z,L\times T^*_0 D^m}) \setminus \coprod_{C_{X\times Z}} C_{X\times Z}^{L\times T^*_0 D^n}.
		$$
        Here,
        $\lambda_{X\times Z,L\times T^*_0 D^m}$ is a Weinstein structure on $X\times Z$ that has $C_{X\times Z}^{L \times T^*_0 D^m}$ as co-cores and  is Weinstein homotopic to $\lambda_X + \lambda_Z$ via a homotopy supported near the co-cores $C_{X\times Z}$ of $X\times Z$.
		We observe that $(X, \lambda_{X, L}) \times Z = (X \times Z , \lambda_{X, L} + \lambda_Z)$ and $(X \times Z, \lambda_{X \times Z, L\times T^*_0 D^m})$ are Weinstein homotopic structures on $X \times Z$ (as both are homotopic to $\lambda_X + \lambda_Z$). 
		Using the identification between the product $C_X \times C_Z$ of co-cores and the co-core $C_{X \times Z}$ of the product $X \times Z$ and the fact that $C_Z^{T^*_0 D^m} = C_Z$, we have that  $C_X^L \times C_Z$ is equal to $C_{X \times Z}^{L \times T^*_0 D^n}$. Hence the identity map is a symplectomorphism between these two sectors. 

        To establish \eqref{item. claim two of carving},
		we need the stronger statement that the identity map is an isomorphism up to deformation. To do this, we need to show that the  homotopic structures $\lambda_{X, L} + \lambda_Z$ and 
		$\lambda_{X \times Z, L\times T^*_0 D^m}$ are homotopic \textit{relative} to $C_X^L \times C_Z = C_{X\times Z}^{L \times T^*_0 D^m}$, {ie }through a family of forms that vanish on this Lagrangian.		
		Note that the the Weinstein structure $\lambda_{X, L} + \lambda_Z$ near neighborhoods $T^*D^n \times T^*D^m$ of the co-cores $C_{X\times Z}$ is $\lambda_{T^*D^n, L} + \lambda_{T^*D^m, std}$. 
		Let $\lambda_{T^*D^n, L, t}$ be the interior Weinstein homotopy from $\lambda_{T^*D^n, L}$ to $\lambda_{T^*D^n, std}$ and let $\lambda_{X, L, t}$ be the homotopy from $\lambda_{X, L}$ to $\lambda_X$ obtained by extending by the identity. This induces a homotopy $\lambda_{X, L, t} + \lambda_Z$ on $X \times Z$, is \textit{not} relative to $L$ since for example the form $\lambda_{T^*D^n, std}+ \lambda_{T^*D^m, std}$ does not vanish on $L \times T^*_0 D^m$. We explain how to modify it and produce a homotopy that does vanish on  $L \times T^*_0 D^m$.

		Consider the homotopy $\lambda_{T^*D^n, L, t} + \lambda_{T^*D^m}$ on $T^*D^n \times T^*D^m$ and the induced homotopy $s_t$ on the sectorial boundary $\partial(T^*D^n\times T^*D^m) = \partial T^*D^n \times T^*D^m \cup T^*D^n \times \partial T^*D^m$.
		We view $s_t$ as a homotopy of the stop for the fixed Weinstein structure $\lambda_{T^*D^n, L } + \lambda_{T^*D^m}$; so we can proceed as in the first part of Proposition \ref{prop: bordered_to_interior_homotopy_Weinstein} and insert the movie construction of $s_t$ near the sectorial boundary of $T^*D^n \times T^*D^m$. The result is a bordered Weinstein homotopy $\lambda_{T^*D^n \times T^*D^m, L \times T^*_0 D^m, t}$  on a slightly larger copy  $(T^*D^n \times T^*D^m)'$ that agrees with the original structure $\lambda_{T^*D^n, L} + \lambda_{T^*D^m, std}$ on $T^*D^n \times T^*D^m \subset (T^*D^n \times T^*D^m)'$  and   agrees with $s_t$ 
		near the sectorial boundary of $(T^*D^n \times T^*D^m)'$.
		Since 
		$\lambda_{T^*D^n \times T^*D^m, L \times T^*_0 D^m, t}$   agrees with the original structure $\lambda_{T^*D^n, L} + \lambda_{T^*D^m, std}$ on $T^*D^n \times T^*D^m$, this homotopy vanishes on $L \times T^*_0 D^n$. See the top-right square of Figure \ref{fig: movie_handle}, which is $(T^*D^n \times T^*D^m)'$ (the smaller subrectangle is  $T^*D^n \times T^*D^m 
		\subset (T^*D^n \times T^*D^m)'$). 
		
		By construction, $\lambda_{T^*D^n \times T^*D^m, L \times T^*_0 D^m, t}$  agrees with the homotopy $\lambda_{X, L, t} + \lambda_{Z}$ near the sectorial boundary of $T^*D^n \times T^*D^m$ (since $s_t$ is the restriction of  $\lambda_{X, L, t} + \lambda_{Z}$ to a neighborhood of 
		$\partial(T^*D^n \times T^*D^m)$). 
		Hence we can take a homotopy on $X \times Z$ that is $\lambda_{T^*D^n \times T^*D^m, L \times T^*_0 D^m, t}$ in $T^*D^n \times T^*D^m$ and is $\lambda_{X, L, t} + \lambda_{Z}$ elsewhere on $X \times Z$.
		This is a bordered homotopy and is relative to $L \times T^*_0 D^m$. Note that the resulting form at time 1 $\lambda_{X \times Z, L \times T^*_0 D^n}$ agrees with $\lambda_X+ \lambda_Z$ except near the co-cores $C_{X \times Z}$; so after carving out  $C_{X \times Z}^{L\times T^*_0 D^n}$, 	we call the resulting sector $(X \times Z)_{L \times T^*_0 D^n}$.  See Figure \ref{fig: movie_handle}. 
        This completes the proof of \eqref{item. claim two of carving}.
        
		Finally, we observe that since all our constructions happen locally near the co-cores of the index $n$ handles, this Weinstein homotopy on $X \times Z$ extends to a Weinstein homotopy on $Y \times Z$ and hence the bottom square in  equation \eqref{comm: comparison_prop_XLZ} also commutes.

		\begin{figure}
			\centering
			\includegraphics[scale=0.15]{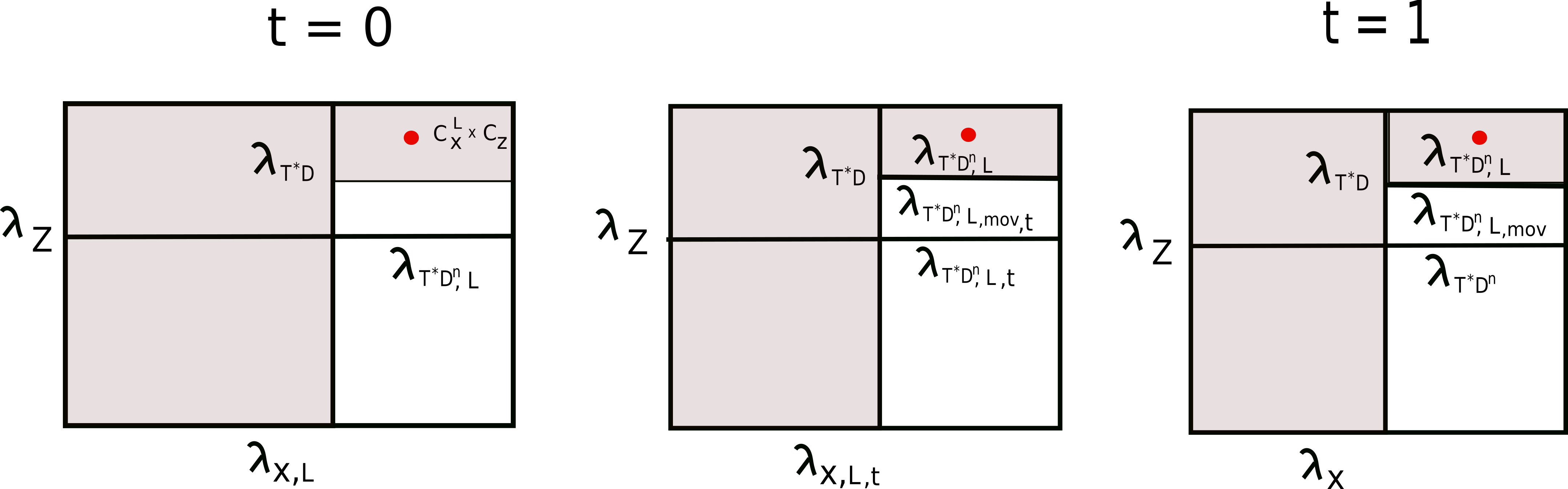}
			\caption{Weinstein homotopy from $(X \times Z, \lambda_{X, L} + \lambda_{Z}	)$ to $(X \times Z, \lambda_{X \times Z, L \times T^*_0 D^m}	)$
				via forms vanishing on $C_X^L \times C_Z$.  The grey region denotes the area where the homotopy is constant. The upper right square denotes a neighborhood $(T^*D^n \times T^*D^m)'$ of the co-core $C_X \times C_Z$, which contains a smaller subrectangle $T^*D^n \times T^*D^m$ with the co-core $C_X^L \times C_Z$.  
			}
			\label{fig: movie_handle}
		\end{figure}

        To finish the proof of \eqref{item. claim three of carving} we must construct the diagram \eqref{comm: comparison_prop_XZK} for $Z$ and $K$. This diagram is constructed similarly to \eqref{comm: comparison_prop_XLZ}, except that the homotopies occur in the $Z$-coordinates instead of the $X$-coordinates as in the previous paragraphs.  In particular, the resulting local model for 
		$(T^*D^n \times T^*D^m,\lambda_{T^*D^m\times T^*D^n,  T^*_0 D^n \times K})$ is the pullback of 	$(T^*D^m \times T^*D^n,\lambda_{T^*D^m\times T^*D^n, K \times T^*_0 D^n})$  
		via the the \textit{swap map} 
		$\phi: T^*D^m \times T^*D^n \rightarrow T^*D^n \times T^*D^m$ given by $\phi(x, y)
		= (y,x)$; this observation will be important for the proof of Theorem \ref{thm: comparison} below. 
	\end{proof}	 
	\begin{remark}
		Note that in general, the homotopy from
		$(X, \lambda_{X,L}) \times Z \setminus \coprod_{C_X, C_Z} C_X^L \times C_Z$ to 
		$$(X\times Z)_{L \times T^*_0 D^m}:= 
		(X \times Z, \lambda_{X\times Z,L\times T^*_0 D^m}) \setminus \coprod_{C_{X\times Z}} C_{X\times Z}^{L\times T^*_0 D^n}
		$$
		is a homotopy of the movie constructions. These movie constructions are induced by the  Weinstein homotopy that is the product of  the homotopy $(T^*D^n, \lambda_{T^*D^n, std})$ to $(T^*D^n, \lambda_{T^*D^n, L})$ with the constant Weinstein structure on $Z\backslash C_Z$. So if the homotopy from $(T^*D^n, \lambda_{T^*D^n, std})$ to 
		$(T^*D^n, \lambda_{T^*D^n, L})$ is a Weinstein homotopy, then so the homotopy from $(X, \lambda_{X,L}) \times Z \setminus \coprod_{C_X, C_Z} C_X^L \times C_Z$. This uses our assumption that the class of Weinstein homotopies are preserved under taking products. 
	\end{remark}
	\begin{remark}
		For the proof of this key result, it is crucial that we carve out the \textit{same} Lagrangian $L$ from all $n$-handles of $X$ when constructing $X_L$; in principle it is possible to construct a subdomain $X_{L,K}$ by carving out arbitrary Lagrangian disks $L$ and $K$ from different $n$-handles in an unrelated way. 
		However, this $X_{L,K}$ will not be equivalent to $X \times (T^*D^n)_L$ or $X\times (T^*D^n)_K$. In fact, such a `mixed' construction will not be homotopy invariant, even in $\weincrit$. 
		In the proof of Proposition \ref{prop: commutative_diagram_sector_localization}, this appears in the statement that 
		we carve out $\coprod_{C_X, C_Z} C_{X}^L \times C_Z$ and this is the same as $\coprod_{C_{X \times Z}} C_{X \times Z}^{L \times T^*_0 D^m}$. 
	\end{remark}
	
	\begin{examples}\label{ex: stabiilize T^*Dn_L}
		Setting $Z = T^*D^k$ in part 1) of 	Proposition \ref{prop: commutative_diagram_sector_localization} shows that there is a strict subcritical cobordism from 
		$X_L \times T^*D^k$ to $(X, \lambda_{X, L}) \times T^*D^k \setminus \coprod_{C_{X}} {C^L_X \times T^*_0 D^k}$
		and an isomorphism up to deformation of the latter to $(X\times T^*D^k)_{L\times T^*_0 D^k}$. 	In particular, to prove results about $X_L$ and maps between such sectors in $\weincrit$, it suffices to prove results about $(X\times T^*D^k)_{L \times T^*_0 D^k}$ for any $k$ and maps between such sectors. We also observe that the subcritical cobordism from 
		$X_L \times T^*D^k$ to $(X, \lambda_{X, L}) \times T^*D^k \setminus \coprod_{C_{X}} {C^L_X \times T^*_0 D^k}$ is topologically non-trivial since the stop of $X_L \times T^*D^k$ is $X_L \times T^*S^{k-1}$	while the stop of 
		$(X\times T^*D^k)_{L \times T^*_0 D^k}$ is $X \times T^*S^{k-1}$, {ie }the stop of $X\times T^*D^k$ itself.
        
An illuminating example is the case $X = T^*S^n$ with the standard Weinstein structure induced by the Morse function on the zero-section $S^n$ with a single index $0$ critical point and a single index $n$ critical point; take $L$ to be a cotangent fiber $T^*_0 D^n \subset T^*D^n$. Then $(T^*S^n)_{T^*_0 D^n}$ is the result of carving out the single co-core of $T^*S^n$, {ie }removing the single index $n$ Weinstein handle; hence $(T^*S^n)_{T^*_0 D^n}$ is just $B^{2n}_{std}$, the standard Weinstein ball with a single index $0$ critical point. Therefore, $(T^*S^n)_{T^*_0 D^n} \times T^*D^n$ is equivalent to $B^{2n}_{std} \times T^*D^n$. If we take $T^*D^n$ to have the Weinstein structure induced by a vector field on the zero-section $D^n$ that is inward pointing towards the origin, {ie }the left diagram of Figure \ref{fig: Weinstein_TDn}, then $T^*D^n$ has a single index $n$ handle and $(T^*D^n)_{T^*D^n_0}$  is a subcritical sector of the form $T^*S^{n-1}\times T^*D^1_+$, where $T^*D^1_+$ has  subcritical Weinstein structure induced by the vector field on $D^1 = [0,1]$ with exactly one zero of index $0$ at $0$ and is rightward-pointing otherwise. So $T^*S^n \times (T^*D^n)_{T^*D^n_0}$ is equivalent to $T^*S^n \times T^*S^{n-1} \times T^*D^1_+$. Finally, $T^*S^n \times T^*D^n = T^*(S^n \times D^n)$ also has a single critical point of index $2n$ and hence $T^*(S^n \times D^n)_{T^*_0 D^{2n}} = T^*(S^n \times D^n \backslash D^{2n})_+$. Here $T^*(S^n \times D^n \backslash D^{2n})_+$ has Weinstein structure induced by the vector field on $(S^n \times D^n) \backslash D^{2n}$ that vanishes on the boundary of $S^n \times D^n$ and points outward on $\partial D^n$. So all three sectors are subcritical but none are diffeomorphic, even if we forget the sectorial boundary.
	\end{examples}
	
	\subsection{Swapping Lagrangian disks}
	\label{sec: swap}
	
	Note that if we take $L = K$ and $Z = T^*D^n$ in Proposition \ref{prop: commutative_diagram_sector_localization}, the map in top row of equation \eqref{comm: comparison_thm_statement1} is the top row in Theorem \ref{thm: comparison} and the top row of 
	equation \eqref{comm: comparison_thm_statement2} is the bottom row in Theorem \ref{thm: comparison}. By Proposition \ref{prop: commutative_diagram_sector_localization}, the top row of equation \eqref{comm: comparison_thm_statement1}, \eqref{comm: comparison_thm_statement2} is related to the bottom row of these equations by equivalences in $\weincrit$. Hence it suffices to prove that the bottom rows of equation \eqref{comm: comparison_thm_statement1}, \eqref{comm: comparison_thm_statement2} 
	\begin{equation}
	\begin{tikzcd} 
	(X \times T^*D^n)_{L \times T^*_0 D^n} \arrow{r}{(f \times Id)_{L \times T^*_0 D^n}}  &  \ \ \ (Y \times T^*D^n)_{L \times T^*_0 D^n} \\
	(X \times T^*D^n)_{T^*_0 D^n\times L} \arrow{r}{(f \times Id)_{T^*_0 D^n\times L}}  &  \ \ \ (Y \times T^*D^n)_{T^*_0 D^n\times L} 
	\end{tikzcd} 
	\end{equation}
	are equivalent in $\weincrit$. The data used to define these first, second maps is just the Lagrangian embeddings $L \times T^*_0 D^n \subset T^*D^n \times T^*D^n $, $T^*_0 D^n \times L \subset T^*D^n\times T^*D^n$ respectively. 	 Hence it suffices to show that two Lagrangian disks are isotopic in $T^*D^n \times T^*D^n$. In the following, we prove this if $n$ is even; a different version of this result via Lagrangian cobordisms appeared in the third author's previous work~\cite{Tanaka_generation}. 
	
	For any symplectic manifold $M$, there is a \textit{swap} symplectomorphism $S: M\times M \rightarrow M\times M$ given by $S(x,y) = (y, x)$. We will need to use the following proposition for the swap symplectomorphism when $M = T^*D^n$. 
	
	\begin{proposition}\label{prop: swap_isotopic_Id}
		If $n$ is even, then $S: (T^*D^n \times T^*D^n, \lambda_{T^*D^n, std} +\lambda_{T^*D^n, std}) \rightarrow (T^*D^n \times T^*D^n, \lambda_{T^*D^n, std} +\lambda_{T^*D^n, std})$ is isotopic to the identity through strict sectorial isomorphism. Furthermore, there is a  strict sectorial isomorphism $\phi: T^*D^n \times T^*D^n \rightarrow T^*D^n \times T^*D^n$ which is the identity map near the sectorial boundary and agrees with $S$ in a smaller copy of $T^*D^n\times T^*D^n$ in the interior of $T^*D^n \times T^*D^n$.		 	
	\end{proposition}
	\begin{proof}   
		There is a symplectomorphism $\phi: T^*D^n\times T^*D^n \rightarrow T^*(D^n\times D^n)$ given by the pullbacks of the projection maps $D^n \times D^n \rightarrow D^n$. 	
		We first observe that there is swap map $s: D^n \times D^n \rightarrow D^n \times D^n$ similarly given by $s(x_1, x_2) = (x_2, x_1)$ on the zero-section. 	This is a diffeomorphism and hence there is an induced map 
		$T^*s: T^*(D^n \times D^n) \rightarrow T^*(D^n \times D^n)$ given by 
		$T^*s(x_1, x_2, p_1 dx_1 + p_2 dx_2) = (s^{-1}(x_1, x_2), s^*(p_1 dx_1 + p_2 dx_2) = (x_2, x_1, p_1 dx_2 + p_2 dx_1) = (x_2, x_1, p_2dx_1 + p_1 dx_2)$. In particular, $T^*s$ is the swap map $S: T^*D^n \times T^*D^n \rightarrow T^*D^n \times T^*D^n$.

		Now we note that since $n$ is even, $s$ is an orientation-preserving linear map and hence is isotopic to the identity through diffeomorphisms of $D^n \times D^n$ (viewed as a manifold with boundary after smoothing the corners). Let $s_t$ be this diffeotopy between $Id$ and $s$. Then $T^*s_t$ is an isotopy of sectorial symplectomorphisms of $T^*D^n \times T^*D^n$ between the identity and the swap symplectomorphism as desired. Furthermore, $T^*s_t$ preserves the standard Liouville form on $T^*D^n \times T^*D^n$ since it is induced by a diffeotopy $s_t$ of the zero-section. Finally, we take $\phi$ to be $T^*s_{movie}$, where $s_{movie}$ is a diffeomorphism of $D^n \times D^n$ which is the identity near the boundary, $s_1$ in a smaller copy of $D^n \times D^n$ in the interior, and interpolates via $s_t$ between these two regions. 
	\end{proof}
	We will call $\phi$ the \textit{cut-off} swap map since it is the identity map near the sectorial boundary of $T^*D^n \times T^*D^n$. 	
	
	\begin{corollary}\label{cor: Lagrangian_disk_isotopy}
		If $i_L: L \hookrightarrow T^*D^n, i_K: K \hookrightarrow T^*D^n$ are Lagrangian embeddings and $n$ is even, then $(i_L, i_K): L\times K \rightarrow T^*D^n \times T^*D^n$ is isotopic to $S\circ (i_L, i_K): L \times K \rightarrow T^*D^n \times T^*D^n$. 
	\end{corollary}
	\begin{proof}
		By Proposition \ref{prop: swap_isotopic_Id}, there is an isotopy $T^*s_t$ of sectorial symplectomorphisms of $T^*D^n \times T^*D^n$ from $Id$ to $S$. Hence $T^*s_t \circ (i_L, i_K): L\times K \hookrightarrow T^*D^n\times T^*D^n$ is an isotopy of Lagrangian embeddings between $(i_L, i_K)$ and $S \circ (i_L, i_K)$. The key point is that a sectorial symplectomorphism (or more generally, proper inclusion of sectors) takes Lagrangians to Lagrangians. 
	\end{proof}

	\subsection{Proof of Theorem \ref{thm: comparison}}
	
	In this section, we complete the proof of the comparison result Theorem \ref{thm: comparison}. We first show that the cut-off swap map defines strict isomorphisms between the relevant sectors. 
	\begin{proposition}\label{prop: swap_localization}
		If $n = \frac{1}{2}\dim X = \frac{1}{2}\dim Z$ is even, then there is a strictly commuting diagram where the vertical maps are strict isomorphisms: 
		\begin{equation}\label{comm: swapping}
		\begin{tikzcd} 
		(X \times Z)_{L \times T^*_0 D^n} \arrow{r}{(f\times Id_Z)_{L\times T^*_0 D^n}}	
		\arrow{d}{
			\phi_{X, Z,L, T^*_0 D^n }} & 
		(Y \times Z)_{L \times T^*_0 D^n} \arrow{d}{\phi_{Y, Z,L, T^*_0 D^n }}\\
		(X \times Z)_{T^*_0 D^n \times L} \arrow{r}{(f\times Id_Z)_{T^*_0 D^n \times L}}  &  (Y \times Z)_{T^*_0 D^n \times L}
		\end{tikzcd}
		\end{equation}
		
	\end{proposition}
	\begin{proof}
		Recall that the cut-off swap map $\phi: (T^*D^n \times T^*D^n, \lambda_{T^*D^n \times T^*D^n, std})   \rightarrow (T^*D^n \times T^*D^n, \lambda_{T^*D^n \times T^*D^n, std})$ from Proposition \ref{prop: swap_isotopic_Id}
		is a strict Liouville isomorphism that agrees with the swap map $S$ in the interior and is the identity near the sectorial boundary. 
		In particular, $\phi(L \times T^*_0 D^n) = T^*_0 D^n \times L$. Hence we can take $\phi_{X, Z, L, T^*_0 D^n}: X \times Z \rightarrow X \times Z$ to be $\phi$ near all the co-cores $C_{X \times Z}$ of $X \times Z$
		and the identity elsewhere. Then $\phi_{X, Z, L, T^*_0 D^n}$ induces a symplectomorphism 
		$\phi_{X, Z, L, T^*_0 D^n}: (X\times Z)_{L \times T^*_0 D^n} \rightarrow (X \times Z)_{T^*_0 D^n \times L}$.

		Next, we observe that this map is actually a strict Liouville isomorphism. This is because by  construction, the swap map is a strictly isomorphism between $(T^*D^n\times T^*D^n, \lambda_{T^*D^n \times T^*D^n, L \times T^*_0 D^n})$ and $(T^*D^n\times T^*D^n, \lambda_{T^*D^n \times T^*D^n,  T^*_0 D^n \times L})$; see the last paragraph of the proof of Proposition \ref{prop: commutative_diagram_sector_localization}. 
		Furthermore, these forms are standard near their sectorial boundaries and so on a trivial enlargement of these sectors, the cut-off swap map $\phi$ is also a strict isomorphism. 
		It also induces a strict isomorphism once we carve out the disks $L \times T^*_0 D^n$ and $\phi(L \times T^*_0 D^n) = T^*_0 D^n \times L$. So the map $\phi_{X, Z, L, T^*_0 D^n}$, which extends $\phi$ by the identity, is also a strict isomorphism. 
		Since $\phi_{X,Z, L, T^*_0 D^n}$ is defined near the co-cores of the handles of $X \times Z$, it extends to a similar map on $Y \times Z$ and get the desired commutative diagram.	
	\end{proof}

	\begin{proof}[Proof of Theorem \ref{thm: comparison}]
		
		If $\dim L = \frac{1}{2} \dim X$ is even, then Theorem \ref{thm: comparison} follows from Proposition \ref{prop: commutative_diagram_sector_localization} and Proposition \ref{prop: swap_localization} applied to $Z = (T^*D^n, \lambda_{T^*D^n, std})$, which has 
		a single isolated critical point of index $n$ so that $(T^*D^n)_L$ is obtained from $T^*D^n$ by removing a single copy of $L$; we also need to use  Propositions \ref{prop: convert_homotopy_equivalence}, \ref{prop: convert_subdomain_to_proper_inclusion} converting isomorphisms up to deformation and subcritical subdomain inclusions into equivalences in $\weincrit$.
		
		If $\dim L= \frac{1}{2}\dim X$ is odd, we first note that 
		$$
		f_L: X_L \rightarrow Y_L
		$$
		is equivalent in $\weincrit$ to 
		$$
		(f\times Id_{T^*D^1})_{L \times T^*_0 D^1}: (X \times T^*D^1)_{L \times T^*_0 D^1} \rightarrow 
		(Y\times T^*D^1)_{L \times T^*_0 D^1}
		$$
		by part 1) of Proposition \ref{prop: commutative_diagram_sector_localization} and Proposition \ref{prop: convert_subdomain_to_proper_inclusion} transforming subcritical subdomain inclusions to equivalences in $\weincrit$. 
		Then $L \times T^*_0 D^1$ has even dimension and we proceed as before to show that $(f\times Id)_{L \times T^*_0 D^1}$ is equivalent to 
		$$
		(f \times Id_{T^*D^1}) \times Id_{(T^*D^{n+1})_{L \times T^*D^1_0}}: (X \times T^*D^1) \times (T^*D^{n+1})_{L \times T^*_0 D^1} \rightarrow 
		(Y\times T^*D^1)\times (T^*D^{n+1})_{L \times T^*_0 D^1}
		$$
		in $\weincrit$.
		Since there is a subcritical cobordism from $(T^*D^n)_L \times T^*D^1$ to 
		$(T^*D^{n+1})_{L \times T^*_0 D^{n+1}}$ again by Part 1) of Proposition \ref{prop: commutative_diagram_sector_localization}, this proper inclusion is equivalent to 	
		$$
		f \times Id_{T^*D^1} \times Id_{(T^*D^n)_L} \times Id_{T^*D^1}: (X \times T^*D^1) \times T^*D^{n}_{L} \times T^*D^1 \rightarrow 
		(Y\times T^*D^1)\times T^*D^{n}_{L} \times T^*D^1
		$$
		Finally, we observe that this morphism is conjugate to the morphism 
		$$
		f \times Id_{ (T^*D^n)_L} \times Id_{T^*D^1}\times Id_{T^*D^1}: X \times (T^*D^n)_L \times T^*D^1 \times T^*D^1 \rightarrow Y \times (T^*D^n)_L \times T^*D^1 \times T^*D^1
		$$ 
		via the swap maps
		$X \times T^*D^1 \times (T^*D^n)_L \times T^*D^1 \rightarrow X \times (T^*D^n)_L \times T^*D^1 \times T^*D^1$. The latter is the stabilization of $f \times (T^*D^n)_L$ as desired. 	
	\end{proof}
	
	Finally, we note that the above proof of Theorem \ref{thm: comparison} proves a slightly stronger result than in the statement of that theorem. Namely, the only non-strict maps that appear in the proof are isomorphisms up to Weinstein homotopy ({ie }the non-strict map in Proposition \ref{prop: commutative_diagram_sector_localization} is the identity map, a diffeomorphism).
	\begin{corollary}\label{cor: X_L, Xx TD_L,iso}
		$X_L$ and $X\times (T^*D^n)_L$ are isomorphic up to Weinstein homotopy, stabilization, and subcritical cobordism. 
	\end{corollary}
	Using the fact that $\times (T^*D^n)_L$ preserves isomorphisms up to Weinstein (Liouville) homotopy, we have the following corollary, which implies Corollary \ref{cor: intro_X_P_properties} from the Introduction. 
	\begin{corollary}\label{cor: geometric_invariance}
		If $X, X'$ are isomorphic up to Weinstein (Liouville) homotopy, then $X_L$ and $X_L'$ are isomorphic up to Weinstein (Liouville) homotopy,  stabilization, and subcritical cobordism. In particular, this holds for flexibilization. 
	\end{corollary}

	\section{Idempotency of P-flexibilization}\label{sec: idempotency}
	
	In this section, we prove that  $\times (T^*D^n)_L$ (and also $(\ )_L$) is a idempotent functor of $\weincrit$, {ie }there is a natural transformation $\eta: Id \rightarrow \times (T^*D^n)_L$ so that $\eta_X \times Id_{(T^*D^n)_L}, \eta_{X\times (T^*D^n)_L}$ are equivalences in $\weincrit$ for all Weinstein sectors $X$. 
	
	\subsection{\texorpdfstring{A natural transformation for $\times (T^*D^n)_L$}{A natural transformation for product with a carved disk}}\label{sec: first_natural_transformation}

	First, we define a natural transformation $\eta: Id \rightarrow \times (T^*D^n)_L$, using a slightly modified model for  $(T^*D^n)_L$. 
	Let ${D^n}' = [0,2] \times D^{n-1}$ be a larger $D^n = [0,1]^n$.  Then we take a Morse-Bott function $f'$ on ${D^n}{'}$, homotopic relative to the boundary to the standard one $f$, so that there is an embedding $\phi_1 \coprod \phi_2: (D^n, f) \coprod (D^n, f) \hookrightarrow ({D^n}{'}, f')$ with
	$\phi_1(D^n) = [0, 1] \times D^{n-1},  \phi_2(D^n) = [1, 2] \times D^{n-1}$ so that $f'$ pulls back to $f$. See Figure \ref{fig: Modified_TDn}.
	Then the function $f'$ induces a Morse-Bott Weinstein structure $({T^*D^n}', \lambda_{T^*{D^n}{'}})$ on $T^*{D^n}{'}$.  It has two co-cores which are cotangent fibers over two different points in $D^n$, say at $x_1, x_2$.
	
	\begin{figure}
		\centering
		\includegraphics[scale=0.15]{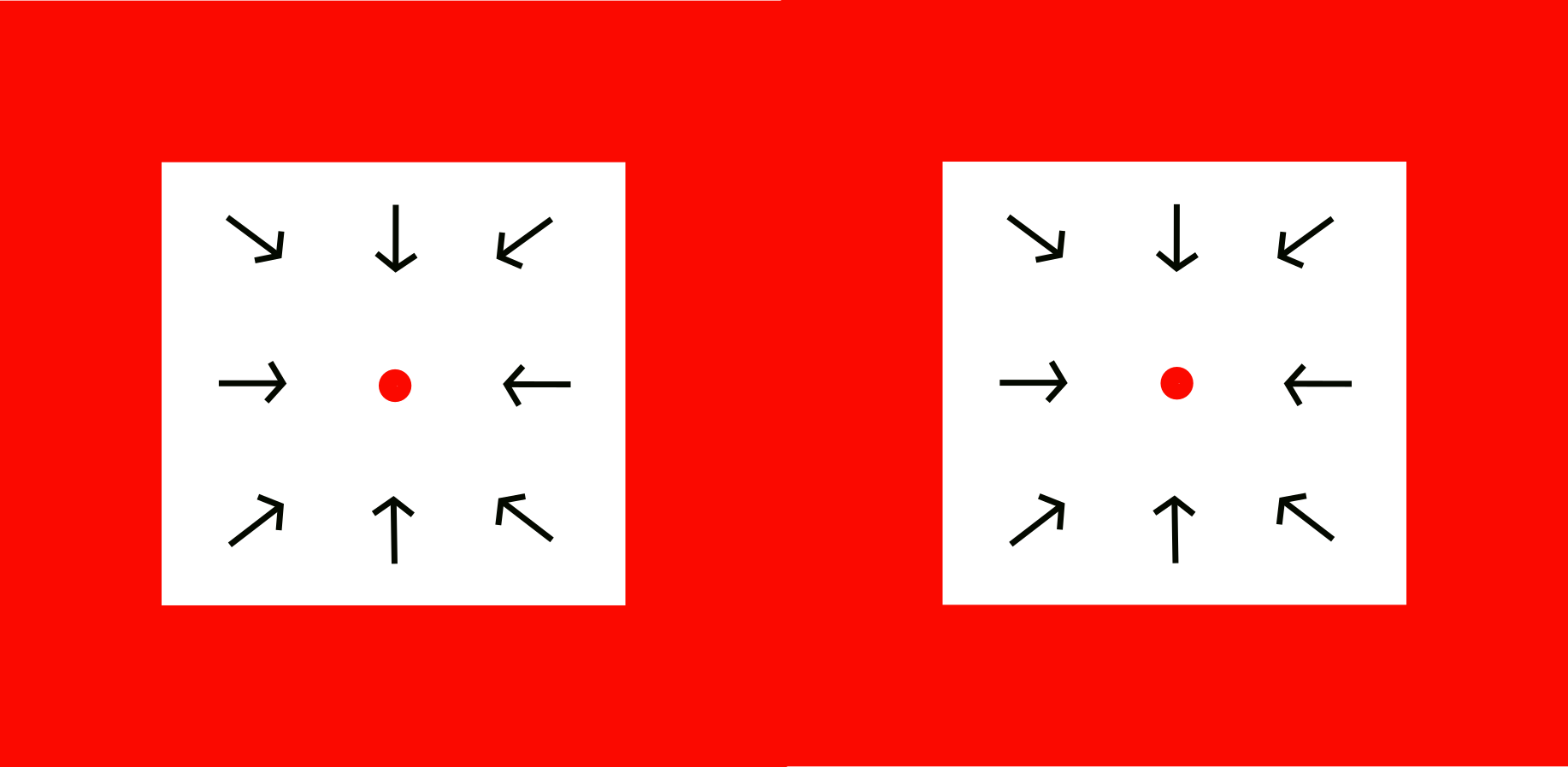}
		\caption{Modified Morse-Bott Weinstein structure $({T^*D^n}', \lambda'_{T^*D^n, std})$
			with zero locus of Liouville vector field in red. 		
		}
		\label{fig: Modified_TDn}
	\end{figure}
    The codimension zero embeddings $\phi_1, \phi_2$ induce (covariant) strict proper
	inclusions 
	$$
	T^*\phi_1, T^*\phi_2: (T^*D^n, \lambda_{T^*D^n, std}) \hookrightarrow ({T^*D^n}', \lambda_{{T^*D^n}'})
	$$  
	so that $T^*\phi_i(T^*_0 D^n) = T^*_{x_i} {D^n}'$ for $i = 1,2$. 
	Then as in the previous sections, there is an interior Weinstein homotopy of $(T^*D^n, \lambda_{std})$ to $({T^*D^n}, \lambda_{T^*D^n, L})$ that has $L$ as a co-core. 
	We can apply this homotopy to the subset $ T^*\phi_2(T^*D^n)\subset T^*{D^n}{'}$ and denote this sector $(T^*{D^n}{'}, \lambda_{{T^*D^n}{'}, T^*\phi_2(L)} )$; it has 
	$T^*\phi_1(T^*_0 D^n) = T^*_{x_1} {D^n}'$ and $T^*\phi_2(L) = (T^*_{x_2} {D^n}')^L$ as co-cores (and possibly some other co-cores). 
	Since this homotopy occurs in the complement of $T^*\phi_1(T^*D^n)$,  there is still a strict proper inclusion
	$$
	T^*\phi_1: (T^*D^n, \lambda_{std}) \hookrightarrow ({T^*D^n}{'}, \lambda_{{T^*D^n}^{'}, T^*\phi_2(L)})
	$$
	Since $T^*\phi_2(L)$ is in the complement of the image of $T^*\phi_1$, there is an induced proper inclusion 
	$$
	\eta_{T^*D^n}:= T^*\phi_1: T^*D^n \hookrightarrow ({T^*D^n}')_L
	$$
	where  $({T^*D^n}')_L:= ({T^*D^n}',\lambda_{{T^*D^n}', T^*\phi_2(L)})\setminus T^*\phi_2(L)$. 
	It will be helpful to keep $(T^*D^n)_L$ and the larger version $({T^*D^n}')_L$ separate, even though  $T^*\phi_2: (T^*D^n)_L\rightarrow
	({T^*D^n}')_L$ is a strict sectorial equivalence. For the rest of this section, we consider the functor $ \times ({T^*D^n}')_L$ instead of the functor $\times {(T^*D^n)_L}$ (which are equivalent functors by the previous sentence). 
	Then the morphism $\eta_{T^*D^n}$ induces a natural transformation 
	$$
	\eta: Id \rightarrow \times ({T^*D^n}')_L
	$$
	in $\weinstab$, 
	where $\eta_X:  X \rightarrow  X \times ({T^*D^n}')_L$ is defined by 
	the proper inclusion on the stabilized version of $X$:
	$$
	Id_X \times \eta_{T^*D^n}:  X \times T^*D^n \rightarrow X \times ({T^*D^n}')_L
	$$
	
	Since $\times ({T^*D^n}')_L$ is a functor of $\weinstr$,  we can apply it to the morphism $\eta_X$ to produce the morphism
	$$
	Id_X \times \eta_{T^*D^n} \times Id_{({T^*D^n}')_L}: X  \times T^*D^n \times ({T^*D^n}')_L \rightarrow  X \times ({T^*D^n}')_L \times ({T^*D^n}')_L
	$$
	Similarly, there is the morphism $\eta_{X\times ({T^*D^n}')_L}$  defined by the proper inclusion 
	$$
	Id_X \times Id_{({T^*D^n}')_L} \times \eta_{T^*D^n}: X \times ({T^*D^n}')_L \times T^*D^n \rightarrow X \times ({T^*D^n}')_L \times ({T^*D^n}')_L
	$$
	Using a swapping identifications of $X \times ({T^*D^n}')_L \times T^*D^n$ with $X \times T^*D^n \times ({T^*D^n}')_L$, these two morphisms are conjugate in $\weinstr$. So one morphism is an equivalence if and only if the other is an equivalence.  
	
	The following theorem is the main result of this section.
	
	\begin{theorem}\label{thm: idempotency_functor}	
		For any Liouville sector $X$,
		the proper inclusions	$Id_{({T^*D^n}')_L} \times \eta_X, \eta_{X\times ({T^*D^n}')_L}$ are  equivalences in $\weincrit$. In particular,	
		$\times ({T^*D^n}')_L$ is a localization functor of $\weincrit$ via  the natural transformation $\eta: Id \rightarrow \times ({T^*D^n}')_L$.
	\end{theorem}

	Since $\eta_X$ is defined to be $Id_X \times \eta_{T^*D^n}$, to prove Theorem \ref{thm: idempotency_functor} it suffices to prove this result for $X = T^*D^n$, {ie }that the (unstabilized) map
	$$
	Id_{({T^*D^n}')_L}\times\eta_{T^*D^n} : ({T^*D^n}')_L \times T^*D^n \rightarrow 
	({T^*D^n}')_L \times ({T^*D^n}')_L
	$$
	is an equivalence in $\weincrit$.      
    Note that attaching subcritical handles to these sectors is crucial,	since  $T^*D^n \times ({T^*D^n}')_L$ and $({T^*D^n}')_L \times ({T^*D^n}')_L$ have different singular cohomology in degree $2n-1$ (since $({T^*D^n}')_L$ is diffeomorphic to $T^*D^n$ with a subcritical handle attached, as discussed in Section \ref{sec: topology_flex_handles}) and hence are not even diffeomorphic after stabilizing.

	\subsection{\texorpdfstring{A natural transformation for $(\ )_L$}{A natural transformation for carving}}\label{sec: second_natural_transformation}
	
	To prove Theorem \ref{thm: idempotency_functor}, we introduce a natural transformation $\theta: Id \rightarrow (\ )_L$ for the non-homotopy invariant P-flexibilization functor  $(\ )_L$. As we explain, it will be easier to prove idempotency for this natural transformation.
	Recall that $X_L$ is a non-strict Weinstein subdomain of $X$. By Proposition \ref{prop: convert_subdomain_to_proper_inclusion}, this yields a morphism $X \rightarrow X_L$ in $\weincrit$. 
	To make this morphism explicit, we  use a similar construction as in Section \ref{sec: first_natural_transformation}. As we discuss in Remark \ref{rem: different_definition} below, the definition of $(\ )_L$ here will be slightly different than in Sections \ref{sec: first_look_flexibilization}, \ref{sec: comparison_P_flexibilization} but are all equivalent in $\weincrit$ by the comparison result Theorem \ref{thm: comparison}.

	As in Section \ref{sec: first_natural_transformation},
    we consider the Morse-Bott function $f'$ on ${D^n}'$ depicted in Figure \ref{fig: Modified_TDn} which induces the Weinstein structure $({T^*D^n}', \lambda_{{T^*D^n}'})$
	and have two strict proper inclusions 
	$$
	Id_X \times T^*\phi_1, Id_X \times T^*\phi_2: X\times (T^*D^n,\lambda_{std}) \rightarrow X \times ({T^*D^n}', \lambda_{{T^*D^n}'})
	$$

	The co-cores of $X \times({T^*D^n}', \lambda_{{T^*D^n}'})$ are $C_X \times T^*_{x_1} D^n$ and $C_X \times T^*_{x_2} D^n$. 
	Next, recall that there is an interior Weinstein homotopy of $X \times (T^*D^n,\lambda_{std}) = (X\times T^*D^n, \lambda_{X\times  T^* D^n})$ to $(X \times T^*D^n, \lambda_{X\times  T^* D^n, L \times T^*_0 D^n}) 
	$ that has $C_X^L \times T^*_0 D^n$ as co-cores; we recall that near the product of co-cores $T^*D^n \times T^*D^n$, the Weinstein structure $\lambda_{T^*D^n \times T^*D^n, L \times T^*_0 D^n}$ is obtained by swapping the Weinstein structure
	$\lambda_{T^*D^n \times T^*D^n, T^*_0 D^n\times L}$.
	We can apply this homotopy to 
	$Id_X\times T^*\phi_2(X \times T^*D^n) \subset X \times {T^*D^n}'$,  while keeping the Weinstein structure 
	on 	$Id_X\times T^*\phi_1(X \times T^*D^n) \subset X \times {T^*D^n}'$
	fixed. 
	We will call the resulting structure
	$
	(X \times {T^*D^n}', \lambda_{X \times {T^*D^n}', L \times T^*\phi_2(T^*_0 D^n})
	$
	which has $C_{X}^L \times T^*\phi_{2}(T^*_0 D^n) = C_{X}^L\times T^*_{x_2} D^n$ as a co-core.
	
	There is still a strict proper inclusion 
	$$
	Id_X \times T^*\phi_1: X \times T^*D^n \rightarrow (X \times T^*D^n{'},
	\lambda_{X \times {T^*D^n}{'}, L\times T^*\phi_{2}(T^*_0 D^n)}  )
	$$
	Since $C_{X}^L \times T^*\phi_2(T^*_0 D^n)$ are disjoint from the image of $Id_X \times T^*\phi_1$, there is an induced proper inclusion 
	$$
	\theta_{X, L}:= 	Id_X \times T^*\phi_1: X \times T^*D^n \rightarrow (X\times {T^*D^n}{'})_{L \times T^*_{x_2} D^n}
	$$   
	where we define
	\begin{equation}\label{eqn: XTD^n_{alternative}}
	(X\times {T^*D^n}{'})_{L \times T^*_{x_2} D^n}
	:= (X\times {T^*D^n}{'}, \lambda_{X \times {T^*D^n}{'}, L\times T^*\phi_2(T^*_{0} D^n)})\setminus 
	\coprod_{C_X} C_X^L \times T^*_{x_2} D^n
	\end{equation}
That is, $\theta_{X, L}$ is obtained from $Id_X \times T^*\phi_1$ by restricting its codomain $X\times {T^*D^n}{'}$ to the subsector $(X\times {T^*D^n}{'})_{L \times T^*_{x_2} D^n}$.

	Using $\theta_{X, L}$, we get a natural transformation 
	$$
	\theta:  Id \rightarrow (\_  \times {T^*D^n}{'})_{L \times T^*_{x_2} D^n}
	$$
	in $\weinparam$; that is, for any proper inclusion $i: X \hookrightarrow Y$ of parametrized Weinstein sectors, the following diagram commutes
	\begin{equation}
	\begin{tikzcd}
	X \times T^*D^n \arrow{d}{\theta_{X, L}} \arrow{r}{f \times Id_{T^*D^n}} & Y \times T^*D^n  \arrow{d}{\theta_{Y, L}}\\
	(X  \times {T^*D^n}{'})_{L\times T^*_{x_2} D^n} \arrow{r}{(f \times Id)_{L\times T^*_{x_2} D^1}}& (Y \times {T^*D^n}{'})_{L\times T^*_{x_2} D^n}
	\end{tikzcd}
	\end{equation}

	\begin{remark}\label{rem: different_definition}
		Note that the definition of $(X\times {T^*D^n}{'})_{L \times T^*_{x_2} D^n}$ in equation \eqref{eqn: XTD^n_{alternative}}  differs slightly from the definition of $(X\times T^*D^n)_{L \times T^*_0 D^n}$ given in  Section \ref{sec: first_look_flexibilization}
		since we use $({T^*D^n}', \lambda_{{T^*D^n}'})$ instead of $(T^*D^n, \lambda_{std})$. 
		However, there is a strict sectorial equivalence  from $(X\times T^*D^n)_{L \times T^*_0 D^n}$ to $(X\times {T^*D^n}')_{L \times T^*_{x_2} D^n}$
		via the map $T^*\phi_2$. Also,  $(X\times T^*D^n)_{L \times T^*_0 D^n}$ is equivalent to $X_L$ by Example \ref{ex: stabiilize T^*Dn_L}. Hence, we can consider $\theta_{X, L}$ as a natural transformation between $Id$ and the usual functor $(\ )_L$. 
	\end{remark}

	\subsection{Comparison of natural transformations}
	
	In this section, we prove a comparison result between the two natural transformations $\eta$ and $\theta$ discussed in the previous sections. 
	\begin{theorem}\label{thm: idempotency_comparison}
		For any Weinstein sector $X$, there is 
		an equivalence $\phi_X' :  (X \times {T^*D^n}')_{L \times T^*_0 D^n} \tilde{\rightarrow} X \times ({T^*D^n}')_L$ and a homotopy commutative diagram in $\weincrit$: 
		\begin{equation}
		\begin{tikzcd} 
		X  \times T^*D^n \arrow{r}{\theta_{X}}	
		\arrow{dr}{\eta_X} & 
		(X \times {T^*D^n}')_{L \times T^*_{x_2} D^n} 
		\arrow{d}{\phi_X'}\\	
		& 
		\ \ X \times ({T^*D^n}')_L
		\end{tikzcd}
		\end{equation}		 
	\end{theorem}
    \begin{remark}
        We note that the morphisms $\theta_X$ and $\eta_X$ are morphisms in $\weinstab$ since they are given by proper inclusions. On the other hand, the equivalence $\phi_X'$ only holds in $\weincrit$, hence the necessity of working in $\weincrit$ in Theorem  \ref{thm: idempotency_comparison}.         
    \end{remark}
	
	\begin{remark}
		We also note that $\phi_X'$ is conjugate to the equivalence $\phi_X: X_L \times T^*D^n\tilde{ \rightarrow} X \times (T^*D^n)_L$ from the previous section.
	\end{remark}

	The proof of Theorem \ref{thm: idempotency_comparison} follows from the following result, which is its analog before stabilizing and inverting subcritical handles. 
	
	\begin{proposition}\label{prop: comparison_natural_transformations}
		If $n = \dim L = \frac{1}{2}\dim X$ is even, there is a strictly commuting diagram of symplectic embeddings: 
		\begin{equation}\label{comm: comparison_natural_transformations}
		\begin{tikzcd} 
		X \times T^*D^n  \arrow{r}{\eta_{X} = Id_X \times \eta_{T^*D^n}}	
		\arrow{d}{Id} &  X \times ({T^*D^n}^{'})_{L}  \arrow{d}{i_{X, (T^*D^n, L)}}\\
		X \times T^*D^n \arrow{d}{Id} \arrow{r}  &  X \times ({T^*D^n}', 
		\lambda_{{T^*D^n}', T^*\phi_2(L)})\setminus \coprod_{C_X}{C_X \times T^*\phi_2(L) } \arrow{d}{Id}
		\\
		X \times T^*D^n \arrow{r}\arrow{d}{Id}  &  (X \times {T^*D^n}')_{T^*_0 D^n \times L } \arrow{d}{\phi_{X, T^*D^n, L, T^*_0 D^n}} \\
		X \times T^*D^n \arrow{r}{ \theta_{X} }  &  (X \times {T^*D^n}')_{L\times T^*_{x_2} D^n }
		\end{tikzcd}
		\end{equation}
		where the horizontal maps are strict proper inclusions, $i_{X, (T^*D^n, L)}$ is a strict subcritical subdomain inclusion, and the map $Id$  on the middle-right is an isomorphism up to Weinstein homotopy, and $\phi_{X, T^*D^n, L, T^*_0 D^n}$ is a strict Liouville isomorphism.
		\begin{proof}
			The proof is essentially a relative version of the proof of Proposition \ref{prop: commutative_diagram_sector_localization}. 
			There is a strictly commuting diagram 
			\begin{equation}
			\begin{tikzcd} 
			T^*D^n  \arrow{r}{\eta_{T^*D^n}}	
			\arrow{d}{Id} &   ({T^*D^n}^{'})_L  \arrow{d}{i_{(T^*D^n,L)}}\\
			T^*D^n  \arrow{r}{ T^*\phi_1 }  & \ \  ({T^*D^n}', \lambda_{{T^*D^n}', T^*\phi_2(L)})
			\end{tikzcd}
			\end{equation}
			and taking the product with $X$, we get the strictly commuting diagram
			\begin{equation}
			\begin{tikzcd} 
			X \times T^*D^n  \arrow{r}{\eta_{X}}	
			\arrow{d}{Id} &   X \times ({T^*D^n}^{'})_L  \arrow{d}{Id_X \times i_{(T^*D^n, L)}}\\
			X \times T^*D^n   \ \ \arrow{r}{Id_X \times  T^*\phi_1 }  & \ \  X\times  ({T^*D^n}', \lambda_{{T^*D^n}', T^*\phi_2(L)})
			\end{tikzcd}
			\end{equation}
			Then as in the proof of Proposition \ref{prop: commutative_diagram_sector_localization}, 
			$X \times ({T^*D^n}')_L \xrightarrow{\subset} X \times  ({T^*D^n}', \lambda_{{T^*D^n}', T^*\phi_2(L)})$ is a strict Weinstein subdomain inclusion, and the complementary cobordism has critical points with co-cores $C_X \times T^*\phi_2(L)$. Since these co-cores are in the complement of the inclusion 
			$Id_X \times T^*\phi_1$,
			this induces the top square of  equation \eqref{comm: comparison_natural_transformations}. 
			
			Next, we observe that the identity is a bordered Weinstein equivalence from 
			$$X \times ({T^*D^n}', \lambda_{{T^*D^n}', T^*\phi_2(L)}) \setminus \coprod_{C_X} C_X \times T^*\phi_2(L)$$
			to 
			$$
			(X \times {T^*D^n}')_{T^*_0 D^n \times L}:= (X \times {T^*D^n}', \lambda_{X \times {T^*D^n}', T^*_0 D^n \times T^*\phi_2(L)}) 
			\setminus \coprod_{C_X} C_X \times T^*\phi_2(L)$$
			To see this, we proceed as in the proof of Proposition \ref{prop: commutative_diagram_sector_localization} and note that the form $\lambda_X+ \lambda_{{T^*D^n}', T^*\phi_2(L)}$ looks like $(T^*D^n \times T^*D^n, \lambda_{T^*D^n, std} + \lambda_{T^*D^n, L})$ near  $C_X \times T^*_{x_2} D^n \subset X \times T^*\phi_2(T^*D^n)$. Then we  use part 1) of Proposition \ref{prop: bordered_to_interior_homotopy_Weinstein} to obtain a  homotopy of forms on $T^*D^n \times T^*D^n$ 
			from $\lambda_{T^*D^n, std} + \lambda_{T^*D^n, L}$ to $\lambda_{T^*D^n \times T^*D^n, T^*_0 D^n \times L}$, which agrees with 
			$\lambda_{T^*D^n, std} + \lambda_{T^*D^n, L, t}$ near $\partial T^*D^n \times T^*D^n$, is constant near $T^*D^n \times \partial T^*D^n$, and is $\lambda_{T^*D^n, std} + \lambda_{T^*D^n, L}$ on a slightly smaller copy of $T^*D^n \times T^*D^n$. 
			We then extend this homotopy to $X \times {T^*D^n}'$ by taking the homotopy
			$\lambda_{X} + (T^*\phi_2)_*(\lambda_{T^*D^n, L, t})$ on the rest of
			$X \times T^*\phi_2(T^*D^n)$ and the  constant homotopy on $X \times T^*\phi_1(T^*D^n)$. 
			Since this homotopy happens in the complement of
			$X \times T^*\phi_1(T^*D^n)$, the second square of equation \eqref{comm: comparison_natural_transformations} also commutes.  See Figure \ref{fig: movie_handle}; to make that figure appropriate for this proof, the $Z$ in that figure there should be $X$ and the $X$ should be ${T^*D^n}'$.
			
			Finally, the sector $(X \times {T^*D^n}')_{L \times T^*_{x_2} D^n}$ is constructed
			via the local model  $(T^*D^n \times T^*D^n, \lambda_{T^*D^n \times T^*D^n, L \times T^*_0 D^n})$, which is the swap of $(T^*D^n \times T^*D^n, \lambda_{T^*D^n \times T^*D^n, T^*_0 D^n\times L})$. So the cut-off swap map $\phi_{X, T^*D^n, L, T^*_0 D^n}$ is a strict Liouville isomorphism. Since this map is supported away from $Id_X \times T^*\phi_1(X \times T^*D^n)$, the last square also strictly commutes.
		\end{proof}
		\begin{proof}[Proof of Theorem \ref{thm: idempotency_comparison}]
			If $n$ is even, this follows from Proposition \ref{prop: comparison_natural_transformations} and Propositions \ref{prop: convert_homotopy_equivalence}, \ref{prop: convert_subdomain_to_proper_inclusion}. We note that the top commuting square in Proposition \ref{prop: comparison_natural_transformations}
			is a pullback diagram since the left vertical map is the identity map; hence Proposition	\ref{prop: convert_subdomain_to_proper_inclusion} can be applied in this setting.

			If $n$ is odd, we first have to stabilize and show that the resulting maps $\eta_{X\times T^*D^1},\theta_{X \times T^*_0 D^1}$ for Lagrangian $L \times T^*_0 D^1$ are related by a homotopy commutative diagram in $\weincrit$ to the unstabilized maps $\eta_X, \theta_X$. This follows from the first half of Proposition \ref{prop: comparison_natural_transformations} 
			applied to $X = T^*D^1$ that discusses the maps $i_{T^*D^1, (T^*D^n, L)}$ and $Id$ (the proposition is stated for $X$ with $\frac{1}{2}\dim X = \dim L = n$ but the first half holds for any $X$).
		\end{proof}
		
	\end{proposition}
	
	\subsection{Proof of idempotency}\label{ssec: idempotency_proof}
	
	Recall that to prove Theorem \ref{thm: idempotency_functor}, it suffices to prove that 
	$$
	Id_{({T^*D^n}')_L}\times\eta_{T^*D^n} : ({T^*D^n}')_L \times T^*D^n \rightarrow 
	({T^*D^n}')_L \times ({T^*D^n}')_L
	$$
	is an equivalence in $\weincrit$. 
	Since $({T^*D^n}')_L$ is equivalent to $(T^*D^n)_L$ in $\weincrit$, it suffices to prove that  
	$$
	Id_{(T^*D^n)_L}\times\eta_{T^*D^n} : (T^*D^n)_L \times T^*D^n \rightarrow 
	(T^*D^n)_L \times ({T^*D^n}')_L
	$$
	where we use $(T^*D^n)_L$ instead of $({T^*D^n}')_L$ for the first term in the product; this will turn out to be slightly more convenient. 
	Then by the comparison result  Theorem \ref{thm: idempotency_comparison} applied to $X = (T^*D^n)_L$, it suffices to prove that 
	$$
	\theta_{(T^*D^n)_L}: (T^*D^n)_L \times T^*D^n\rightarrow ((T^*D^n)_L\times {T^*D^n}')_{L\times T^*_{x_2} D^n}
	$$
	is an equivalence in $\weincrit$. 
	We will do this for a particular special choice of model for $(T^*D^n)_L$, after first stabilizing $L$.

	Let $C$ be a Lagrangian co-core of $(T^*D^n)_L$ and $C^L$ be a copy of $L$ embedded in a Weinstein neighborhood of $C$.
	Then	$\theta_{(T^*D^n)_L}$ is the proper inclusion
	$$
	(T^*D^n)_L \times T^*D^n\rightarrow (T^*D^n)_L\times {T^*D^n}' \backslash \coprod_{C} C^L \times T^*_{x_2} D^n
	$$
	induced by $T^*D^n \rightarrow {T^*D^n}'$. 
	Next, suppose that each $C^L \subset (T^*D^n)_L$ is the co-core of an index $n$ handles that is in \textit{cancelling position} with an index $n-1$ handles. That is, there is a Weinstein homotopy that cancels both of these handles (which is the case if the attaching sphere of the index $n$ handle intersects the belt sphere of the index $n-1$ handle exactly once). In this case, the union of the index $n$ and index $n-1$ handles is a trivial cobordism (up to Weinstein homotopy) and $C^L \subset (T^*D^n)_L$  is a Lagrangian unknot. Then $\theta_{(T^*D^n)_L}$ is an equivalence since carving out $C^L$ from  ${(T^*D^n)_L}$ is the same as adding a subcritical handle to $(T^*D^n)_L$ (and similarly, carving out 
	$C^L \times T^*_{x_2} D^n$ from  $(T^*D^n)_L \times {T^*D^n}'$ is the same as adding subcritical handles to the complement of the image of $(T^*D^n)_L \times T^*D^n \rightarrow (T^*D^n)_L \times {T^*D^n}'$). However a priori, $(T^*D^n)_L$ might have many co-cores $C$ and so we cannot control $C^L$ and assume that it is a Lagrangian unknot. The following key proposition shows that we can control the new co-cores \textit{after} stabilization.

	\begin{proposition}\label{prop: two_cocores}
		For any regular Lagrangian disk $L^n \subset (T^*D^n, \lambda_{T^*D^n, L, t})$, 
		there is an interior Weinstein homotopy of $(T^*D^n \times T^*D^1, \lambda_{T^*D^n, L} + \lambda_{T^*D^1, std})$, relative to $L \times T^*_0 D^n$, to a Weinstein structure $\lambda_{T^*D^n\times T^*D^1, L \times T^*_0 D^1}$ with exactly two index $n+1$ critical points whose Lagrangian co-cores are $T^*_0  D^n \times T^*_{-1/2} D^1$ and $L^n \times T^*_0 D^1$.
	\end{proposition}
	
	We momentarily postpone the proof of Proposition \ref{prop: two_cocores} until the end of this section and  now explain how to prove idempotency for the functor $\times (T^*D^{n+1})_{L\times T^*_0 D^1}$ 
	assuming this result. 
	By Proposition \ref{prop: two_cocores}, $T^*D^{n+1}$ has a Weinstein presentation with two co-cores $C_1 = T^*_0 D^n \times T^*_{-1/2} D^1$ and $C_2 = L \times T^*_0 D^1$. So $T^* D^{n+1}_{L\times T^*_0 D^1}:= T^*D^{n+1}\setminus L \times T^*_0 D^1$ has a single co-core $C_1 = T^*_0 D^n \times T^*_{-1/2} D^1$. Recall that we need to show that $C_1^{L\times T^*_0 D^1}$ is the co-core of an index $n+1$ handle that cancels with a subcritical handle.

	To do so, it is helpful to consider
	$$
	C_1^{L \times T^*_0 D^1} \coprod C_2 
	$$
	as a Lagrangian link, {ie }union of disjointly embedded Lagrangians. 	We say that a link $C \coprod C' \subset X$ of Lagrangian disks is \textit{parallel} if there is a strict proper inclusion 
	$i: T^*D^n \hookrightarrow X$ so that 
	$i(T^*_0 D^n \coprod T^*_x D^n) = C \coprod C'$ for some $x \ne 0$ in $D^n$. This perspective is helpful in light of the following observation. 		
	\begin{proposition}\label{prop: parallel_links_unknot}
		Let $C \coprod C' \subset X$ be a parallel Lagrangian link of Lagrangian disks. Then $C' \subset X \setminus C$ is Lagrangian co-core of an index $n$ handle that cancels an index $n-1$ handle. 
	\end{proposition}
	\begin{proof}
		Take a Morse function $f$ on $D^n$ whose interior critical points are two index $n$ critical points (at $0$ and $x$) and one critical point of index $n-1$; this Morse function can be obtained by starting with the standard structure that has a single critical point of index $n$ and creating a cancelling pair of critical points of index $n$ and $n-1$. 
		Then the  induced Weinstein structure on $T^*D^n$ has two index $n$ critical points with Lagrangian co-cores $T_0^* D^n$ and $T^*_x D^n$. Then $T^*D^n \backslash T^*_0 D^n$ has a single index $n$ handle which is cancelling with the other index $n-1$ handle and has co-core  $T^*_x D^n$; this proves the result for $T^*D^n$. 	Since $C\coprod C' \subset X$ is modelled on $T^*_0 D^n \coprod T^*_x D^n \subset T^*D^n$, the same result holds for $X$. 
	\end{proof}
	Finally, we note that 	$$
	C_1^{L \times T^*_0 D^1} \coprod C_2 
	$$ 
	is the link 
	$$
	L \times T^*_{-1/2} D^1 \coprod L \times T^*_0 D^1
	$$
	which is parallel. Indeed, if $\phi_L: T^*D^n \rightarrow T^*D^n$ is a parametrization of $L$, then $\phi_L \times Id_{T^*D^1}: T^*D^n \times T^*D^1 \rightarrow T^*D^{n} \times T^*D^1$ has the property that $\phi_L \times Id_{T^*D^1}(T^*_0 D^n \times T^*_{-1/2} D^1) =L \times T^*_{-1/2} D^1$
	and  $\phi_L \times Id_{T^*D^1}(T^*_0 D^n \times T^*_0 D^1) = L \times T^*_{0} D^1$. Hence  
	$$
	L \times T^*_{-1/2} D^1 \subset (T^*D^{n+1})_{L\times T^*_0 D^1} = T^*D^{n+1} \backslash L \times T^*_0  D^1
	$$
	is the co-core of a cancelling $n+1$ handle. This completes the proof of Theorem \ref{thm: idempotency_functor} for 
	the functor $\times (T^*D^{n+1})_{L\times T^*_0 D^1}$ using the stabilized Lagrangian $L\times T_0 D^1$. 
	
	\subsubsection{Simple Weinstein presentation, after stabilization}
	Next, we complete the proof of Proposition \ref{prop: two_cocores}, showing that $T^*D^{n+1}$ admits a simple Weinstein presentation with two co-cores $T^*_0 D^n\times T^*_{-1/2} D^1, L \times T^*_0 D^1$. Then we will complete the proof of Theorem \ref{thm: idempotency_functor}, idempotency for the functor $\times T^*D^{n}_{L}$. 
	\begin{proof}[Proof of Proposition \ref{prop: two_cocores}]
		In this proof, it will be useful to consider Weinstein sectors as associated to stopped Weinstein domains since we will consider separate homotopies of the stop and of the domain and then combine them into sectorial homotopies. To that end, we recall the following notation from Definitions \ref{defn. sectorial divisor}, \ref{defn. stopped domain}: 
		$[X, F]$  denotes a sector $X$ and its sectorial divisor $F$ and $\overline{(X_0, \Lambda)}$ denotes the sectorial completion of a stopped domain $(X_0, \Lambda)$. 	Since we will need to carve out the Lagrangian $L$ or $L \times T^*_0 D^1$, all our  homotopies are \textit{relative} to these Lagrangians, {ie }through forms vanishing on the Lagrangians.

		Since $L^n \subset [T^*D^n, T^*\partial D^n]$ is a regular disk, there is a bordered Weinstein homotopy of $[T^*D^n, T^*\partial D^n]$
		to a sector $(T^*D^n, \lambda_{T^*D^n, L})= \overline{(T^*L^+ \cup C^{2n}, \partial D^n)}$ so that $L \subset T^*D^n$ corresponds to the zero-section of $T^*L^+$. 
		Here $(T^*L^+ \cup C^{2n}, \partial D^n)$ is a stopped domain; $T^*L^+$ is the canonical Weinstein \textit{domain} structure on cotangent bundles with outward pointing Liouville vector field \textit{everywhere} and an index $0$ critical point on $L$, with $\partial L$ mapping to the contact boundary; $C^{2n}$ is a Weinstein cobordism and the attaching spheres of $C^{2n}$ are disjoint from $\partial L$. 
		Similarly, $F = T^*_0 D^1 \subset T^*D^1$ is a regular Lagrangian and so we can Weinstein homotope
		$[T^*D^1, T^* \partial D^1]$  to  $\overline{(T^*F^+, \partial D^1)}$. Here $(T^*F^+, \partial D^1)$ is just $(B^2, \pm 1)$, a ball with  a single index $0$ critical point in the interior and two points for stops. So $\overline{(T^*F^+, \partial D^1)}$ has two index $1$ critical points in the interior (corresponding the linking disks of the two stops), one index $0$ critical point in the interior (corresponding to the index $0$ critical point of $B^2$),  and two index $0$ critical points on the boundary (corresponding to the stops). 
		
		Next, we consider the product of these two sectors $(T^*D^n, \lambda_{T^*D^n, L})\times T^*D^1$.  The sector associated to the product of two stopped domains coincides with the product of sectors associated to two stopped domains. So this sector is
		$\overline{((T^*L^+ \cup C)\times T^*F^+, \Lambda)}$, where $\Lambda = (T^*L^+\cup C) \times T^*\partial D^1 \cup T^*\partial D^n \times T^*F^+$. Since this sector is bordered homotopy to $(T^*D^n, \lambda_{T^*D^n, std}) \times (T^*D^1, \lambda_{T^*D^1, std})$, the stop $\Lambda$ is a Weinstein homotopic to the stop of $T^*D^n \times T^*D^1$ which is $T^* \partial D^{n+1}$. Next, we proceed as in Part 3) of Proposition \ref{prop: bordered_to_interior_homotopy_Weinstein} and append the movie construction of this homotopy to $\overline{((T^*L^+ \cup C)\times T^*F^+, \Lambda)}$ and get a bordered homotopy rel $L \times F$ from $(T^*D^n, \lambda_{T^*D^n, L})\times T^*D^1$ to  $\overline{(((T^*L^+ \cup C) \times T^*F^+)', T^* \partial D^{n+1})}$. Here $((T^*L^+ \cup C)\times T^*F^+)'$ is a slight enlargement of $(T^*L^+ \cup C)\times T^*F^+$ with the same Lagrangian co-cores; hence we will identify these two domains.

		Next, we will Weinstein homotope the domain part $(T^*L^+ \cup C) \times T^*F^+$ of the stopped domain $((T^*L \cup C) \times T^*F^+, T^*\partial D^{n+1})$. First, we observe that $(T^*L \cup C)\times T^*F^+$ is a subcritical Weinstein domain since $T^*F^+$ is subcritical. In fact, the Weinstein cobordism $C \times T^*F^+$ in the complement of $T^*(L \times F)^+$ is subcritical. 
		Since $C$ is a disk, this cobordism is a smooth h-cobordism and hence is smoothly trivial if $n \ge 2$; if $n = 1$, there is only two Lagrangian disks in $T^*D^1$, the fiber and the unknot in which case the result can be checked explicitly. So $C^{2n}\times T^*F^+$ is smoothly trivial Weinstein cobordism, which is also subcritical, and hence by the h-principle for subcritical cobordisms~\cite{gromov_hprinciple}, it is Weinstein homotopic to the trivial cobordism. Furthermore, this Weinstein homotopy is relative to $L \times F$ since the smoothly isotopies of the attaching spheres can done away from  $\partial (L \times F)$
		and the isotropic isotopies are $C^0$-close to the smooth isotopies. In particular, $(T^*L^+ \cup C) \times T^*F^+$ is homotopic, relative $L \times F$, to $T^*(L\times F)^+$ ; note that $T^*(L\times F)^+$ is abstractly just a  Weinstein ball $B^{2n+2}$. 
		
		Now we return to the sector 
		$(T^*D^n, \lambda_{T^*D^n, L})\times T^*D^1$, which we have already proven is homotopic, relative $L \times F$, to $\overline{((T^*L^+ \cup C) \times T^*F^+, T^* \partial D^{n+1} )}$. In the previous paragraph, we homotoped the domain $(T^*L^+ \cup C) \times T^*F^+$, relative $L \times F$, to $T^*(L\times F)^+$. Since $(T^*L^+ \cup C) \times T^*F^+$ is just a Weinstein subdomain of the sector $\overline{((T^*L^+ \cup C) \times T^*F^+, T^* \partial D^{n+1} )}$, we can extend this homotopy of domains to an interior homotopy of the sector. In conclusion, $(T^*D^n, \lambda_{T^*D^n, L})\times T^*D^1$ is homotopic, relative $L \times F$, to
		$[(T^*(L \times  F)^+, T^* \partial D^{n+1})]$.
		Abstractly, this stopped domain  is just $(B^{2n+2},  \partial D^{n+1})$. However, the Lagrangian $L \times F \subset T^*(L \times F)^+$ is non-trivially linked with the stop 
		$\partial D^{n+1}$ and hence non-trivial information is retained.
		
		The Weinstein sector 
		$\overline{(T^*(L \times  F)^+, T^* \partial D^{n+1})}$ has the following critical points.  The domain 
		$T^*(L \times  F)^+$ has one index $0$ critical point (lying in $L \times F$) while the stop 
		$T^* \partial D^{n+1}$ has two critical points of index $0, n$, which give rise to index $1, n+1$ critical points  in a neighorhood $T^*\partial D^{n+1} \times T^*[0,1]$ of the stop $T^*\partial D^{n+1}$. The co-core of the index $n+1$ critical point is precisely a cotangent fiber of  $T^*\partial D^{n+1} \times T^*[0,1]$, which viewed in $T^*D^{n+1}$ is  $T^*_x D^{n+1}$ for some $x$ near $\partial D^{n+1}$.  After further Weinstein homotopy, we can assume that $x = 0 \times \{-1/2\} \in D^n \times D^1$.
		
		Next, we again use the fact that $L$ (and hence $L \times F$) is a disk. In this case, there is a Weinstein homotopy of $\overline{(T^*(L \times  F^1)^+, T^* \partial D^{n+1})}$ supported in a small neighorhood of $L \times F$ 
		to a new structure with an additional index $n$ and $n+1$ 
		critical point so that the co-core of the latter is $L \times F$; see~\cite{EGL}. Since this Weinstein homotopy is supported in a small neighborhood of $L \times F$, it does not create any new critical points outside and does not change the co-core $T^*_0 D^{n+1}$ of the other index $n+1$ critical point lying in $T^*\partial D^{n+1} \times T^*[0,1]$. 
		Finally, we can homotope this structure so that this index $0$ and index $1$ critical point lying in $T^* \partial D^{n+1}\times  T^*[0,1]$ are cancelled, producing a structure with a total of three critical points, one of index $n$ and two of index $n+1$ whose co-cores are $T^*_0 D \times T^*_{-1/2} D^1$ and $L \times F = L \times T^*_0 D^1$.
	\end{proof}
	
	\begin{remark}
		Proposition 
		\ref{prop: two_cocores} is  false unless we first stabilize the Lagrangian disk $L \subset T^*D^n$. 
		Otherwise $T^*D^n \setminus L \cup T^*_0 D^n = B^{2n} \cup H^{n-1} \setminus L$ would be subcritical, which is not always the case for Lagrangian disks in $L$. Indeed, by \cite{MS},  for every $n \ge 4$, there are Lagrangian disks $L \subset B^{2n}$ so that $B^{2n} \setminus L$ is not subcritical.
		However, when we stabilize, $L \times T^*_0 D^1$ becomes isotopic to the Lagrangian unknot and so $B^{2n+2} \setminus L \times T^*_0 D^1$ is automatically subcritical.  
		
		For general $L \subset T^* D^n$, there cannot be a Weinstein structure on $T^*D^{n+1}$ with $L\times T^*_0 D^1$  as the only co-core since this implies that $L \times T^*_0 D^1$ (and also $L$) is the generator of 
		the Fukaya category of $T^*D^{n+1}$. Therefore, Proposition 
		\ref{prop: two_cocores}, where there are precisely two co-cores, is the best possible scenario. 
	\end{remark}
	\begin{remark}\label{rem: use_hprinciple}
		Proposition \ref{prop: two_cocores} is the only place in the paper where we actually use any h-principle (Gromov's h-principle for subcritical isotropics~\cite{gromov_hprinciple}). We also observe that if $L \subset T^*D^n$ is already isotopic to the zero-section $D^n$ as a Lagrangian in the unstopped domain $B^{2n}$, then there is no need to stabilize or use this h-principle since we are already in the situation that $(T^*D^n, \lambda_{T^*D^n, L})$ is homotopic, relative $L$, to $\overline{(T^*L^+, T^* \partial D^n)}$. 
		For example, the $D_P$ disks from~\cite{abouzaid_seidel_recombination} used to construct $X[P^{-1}]$ are graphical and hence isotopic  	
		to $D^n \subset B^{2n}$.  Hence our proofs that $X[P^{-1}]$ (including the flexible case $X[\{0\}^{-1}]$)  is idempotent and independent of presentation does not depend on any h-principle.		
	\end{remark}

	Finally, we complete the proof that the functor $\times (T^*D^n)_L$ is idempotent, using the already proven fact that $\times (T^*D^{n+1})_{L\times T^*_0 D^1}$ is idempotent. 
	\begin{proof}[Proof of Theorem \ref{thm: idempotency_functor}]
		Suppose that $L \subset T^*D^n$ is a regular Lagrangian and we fix a Weinstein homotopic form $\lambda_{T^*D^n, L}$ on $T^*D^n$ so that $\lambda_{T^*D^n, L}|_L = 0$. Then by Proposition \ref{prop: two_cocores},  $\lambda_{T^*D^n,L} +\lambda_{T^*D^1, std}$ is homotopic, relative $L \times T^*_0 D^1$, to the structure  $\lambda_{T^*D^{n+1}, L\times T^*_0 D^1}$ which has only two co-cores. 
		Setting $X = T^*D^1$ in Proposition \ref{prop: comparison_natural_transformations}, we have that the morphism $$
		\eta_{T^*D^n}: T^*D^n \rightarrow (T^*D^n)_L
		$$
		which is defined via $\lambda_{T^*D^n, L}$,
		is conjugate in $\weincrit$ to the stabilized morphism 
		$$
		\eta_{T^*D^{n+1}}: T^*D^{n+1} \rightarrow (T^*D^{n+1})_{L\times  T^*_0 D^1}
		$$
		which is defined via 
		$\lambda_{T^*D^{n+1}, L\times T^*_0 D^1}$. More precisely, this is because the homotopy used to construct $\lambda_{T^*D^{n+1}, L\times T^*_0 D^1}$ consists of two steps, a movie of a boundary homotopy (precisely as in the proof of Proposition \ref{prop: comparison_natural_transformations}) 	and an interior homotopy, relative to $L \times T^*_0 D^1$, (which does not affect the proof of Proposition \ref{prop: comparison_natural_transformations}).  
		We also have already proven idempotency for $(T^*D^{n+1})_{L\times T^*_0 D^1}$,  {ie }that the morphism 
		$$
		Id_{(T^*D^{n+1})_{L\times T^*_0 D^1}} \times
		\eta_{T^*D^{n+1}}:
		(T^*D^n)_{L\times  T^*_0 D^1}\times	T^*D^{n+1} \rightarrow (T^*D^n)_{L\times  T^*_0 D^1}\times (T^*D^n)_{L\times  T^*_0 D^1}
		$$
		is an equivalence in $\weincrit$. 
		Since $(T^*D^{n+1})_{L\times  T^*_0 D^1}$ is equivalent to $(T^*D^n)_{L}$ in $\weincrit$ by Example \ref{ex: stabiilize T^*Dn_L},
		
		$$
		Id_{(T^*D^n)_L}\times	\eta_{T^*D^n}:(T^*D^n)_L\times T^*D^n \rightarrow (T^*D^n)_L\times (T^*D^n)_L
		$$	
		is also an equivalence in $\weincrit$ as desired. 	\end{proof}

	\begin{remark}\label{remark: two sets of primes, proof}  
		The construction of the simple Weinstein presentations in Proposition \ref{prop: two_cocores} can be used to prove the claim in Example \ref{example: two sets of primes} - that $(T^*D^n[P^{-1}]) \times (T^*D^n[Q^{-1})$ is equivalent to $(T^*D^n)[(P \cup Q)^{-1}]$ in $\weincrit$. More  generally, we prove that $(T^*D^n)_{L} \times (T^*D^n)_K$ is equivalent to $(T^*D^n)_{L \coprod K}$. 
		First note that $(T^*D^n)_L \times (T^*D^n)_K$ is equivalent
		to $((T^*D^n)_L)_K$ in $\weincrit$ by Theorem \ref{thm: comparison}. 
		Now, by Proposition \ref{prop: two_cocores}, $T^*D^{n+1}$ has 
		a Weinstein presentation with two co-cores $L \times T^*_0 D^1$ and $T^*_0 D^n \times T^*_{-1/2} D^1$. Therefore, we have the following equalities: 
		$$
		((T^*D^{n+1})_{L\times T^*_0 D^1})_{K\times T^*_0 D^1}
		= {(T^*D^{n+1})_{L\times T^*_0 D^1}} \backslash  K \times T^*_{-1/2} D^1 
		= T^*D^{n+1} \backslash (L \times T^*_0 D^1 \coprod K \times T^*_{-1/2} D^1 )
		$$
		The latter sector is precisely
		$(T^*D^{n+1})_{L \times T^*_0 D^1 \coprod K \times T^*_{-1/2} D^1}$. Assuming that $L \coprod K$  is disjointly embedded into $T^*D^n$ via a map $T^*\phi_1 \coprod T^*\phi_2: T^*D^n \coprod T^*D^n \hookrightarrow T^*D^n$, then 
		$(T^*D^n)_{L \coprod K}$ is equivalent to 
		$(T^*D^{n+1})_{L \times T^*_0 D^1 \coprod K \times T^*_{0} D^1}$ by Example \ref{ex: stabiilize T^*Dn_L}, which is isomorphic to 	 $(T^*D^{n+1})_{L \times T^*_0 D^1 \coprod K \times T^*_{-1/2} D^1}$.
		
		This proves that  $(T^*D^n)[(P \coprod Q)^{-1}]$ is equivalent to $T^*D^n[P^{-1}] \times
		T^*D^n[Q^{-1}]$, where we allow possibly repeated integers in the set $P \coprod Q$. However by Theorem \ref{thm: idempotency_functor}, $(T^*D^n)[(P \coprod Q)^{-1}]$  is equivalent to $(T^*D^n)[(P \cup Q)^{-1}]$, where no integers are repeated in the set $P \cup Q$. 
	\end{remark}
	
	\subsubsection{Unlinking, after stabilization}
	
	In this section, we briefly discuss certain unlinking phenomena which appeared implicitly in the proof of Theorem \ref{thm: idempotency_functor}. 	A crucial step in the proof was the fact that
	$$
	C_1^{L \times T^*_0 D^1} \coprod C_2 
	$$ 
	is a parallel link. This fact was easy to check since this link was the stabilization via two different cotangent fibers $T^*_{-1/2} D^1, T^*_{0} D^1$ (stabilization was already used to prove that Proposition \ref{prop: two_cocores}).  
	
	Furthermore, in Example \ref{example: non-homotopy}, the main issue was that it is not clear that the Lagrangian link $(T^*_{x_1} S^n)^L \coprod (T^*_{x_2} S^n)^L$ is parallel, even though the two components of this link are isotopic to each other; see Figure  \ref{fig: parallel_links} for a schematic depiction of the situation. 	Indeed, if this link were parallel, then
	$T^*S^n \backslash (T^*_{x_1} S^n)^L \coprod (T^*_{x_2} S^n)^L$ would coincide with 
	$T^*S^n \backslash (T^*_{x_1} S^n)^L$ up to subcritical handles by Proposition \ref{prop: parallel_links_unknot} 
	(and so $T^*S^n_L$ and ${T^*S^n}'_L$ would be equivalent up to subcritical handles). 
	However, note that $(T^*_{x_1} S^n)^L \coprod (T^*_{x_2} S^n)^L$ is \textit{component-wise} isotopic to a parallel link $(T^*_{x_1} S^n)^L \coprod {(T^*_{x_1} S^n)^L}'$ (obtained by taking a small push-off of $(T^*_{x_1} S^n)^L$ in its Weinstein neighborhood) since $(T^*_{x_2} S^n)^L$ is isotopic to $(T^*_{x_1} S^n)^L$.
	
	\begin{figure}
		\centering
		\includegraphics[scale=0.2]{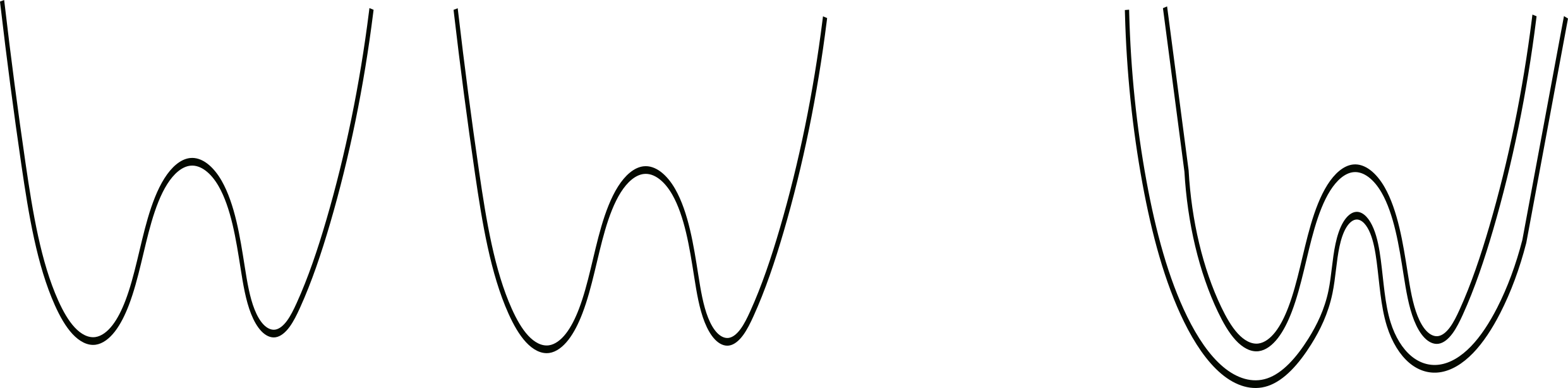}
		\caption{Schematic depiction of non-parallel links 
			$(T^*_{x_1} S^n)^L \coprod (T^*_{x_2} S^n)^L$ 		on the left versus a parallel link $(T^*_{x_1} S^n)^L \coprod {(T^*_{x_1} S^n)^L}'$ on the right. 
		}
		\label{fig: parallel_links}
	\end{figure}
	
	More generally, consider two Lagrangian links $L \coprod K$ and $L' \coprod K' \subset  X$  that are component-wise isotopic; that is, $L$ is Lagrangian isotopic to $L'$ and $K$ is Lagrangian isotopic to $K'$ in $X$. This does not imply that $L \coprod K$ is Lagrangian isotopic as a link to $L' \coprod K'$ since the isotopy from $L$ to $L'$ may intersect the isotopy from $K$ to $K'$, {ie }$L \coprod K$ are linked differently from $L' \coprod K'$; see \cite{mak_smith-links} for examples of Lagrangian linking. In the following proposition, we show that all Lagrangian links become unlinked after a single stabilization, which as explained above is a key phenomena underlying Theorems \ref{thm: comparison} and \ref{thm: idempotency_functor}. 
	
	\begin{proposition}\label{prop: unlinking_stabilization}
		Let $L \coprod K, L' \coprod K' \subset  X$ be Lagrangian links that are component-wise isotopic. Then the stabilized Lagrangian links $L\times T^*_0 D^1 \coprod K \times T^*_0 D^1, L' \times T^*_0 D^1 \coprod K' \times T^*_0 D^1$ are isotopic as links in $X \times T^*D^1$.
	\end{proposition}
	
	\begin{proof}
		First, the link $L\times T^*_0 D^1 \coprod K \times T^*_0 D^1$ is isotopic to  $L\times T^*_0 D^1 \coprod K \times T^*_{1/2} D^1$ obtained by isotoping $T^*_0 D^1$ to $T^*_{1/2} D^1$ in $T^*D^1$ on the second component of the link, keeping $X$ coordinates fixed.  Then we isotope $L\times T^*_0 D^1$ to $L' \times T^*_0 D^1$ and simultaneously isotope $K\times T^*_{1/2} D^1$ to 
		$K'\times T^*_{1/2} D^1$, keeping the  $T^*D^1$ coordinates fixed. This is an isotopy of Lagrangian links since the two component-wise isotopies occur at different cotangent fibers $T^*_0 D^1$ and $T^*_{1/2} D^1$. 
		Finally, we isotope the link $L' \times T^*_0 D^1 \coprod K' \times T^*_{1/2} D^1$ to 
		$L' \times T^*_0 D^1 \coprod K' \times T^*_{0} D^1$
		by isotoping 	$T^*_{1/2} D^1$ back to $T^*_{0} D^1$ in $T^*D^1$ on the second component, again keeping the $X$ coordinates fixed.
	\end{proof}

	\bibliographystyle{abbrv}
	\bibliography{sources}

\end{document}